\documentclass[a4paper,10pt]{amsart}
\usepackage[utf8]{inputenc}
\usepackage{amssymb}
\usepackage{amsmath}
\usepackage{bbm}
\usepackage{mathrsfs}
\usepackage[all]{xy}
\usepackage[left=2cm,right=2cm,top=2cm,bottom=2cm]{geometry}
\usepackage{manfnt}
\usepackage{marvosym}
\usepackage{graphicx}
\usepackage{tikz}

\title{On derivations of free algebras over operads and the generalized divergence}
\date{}

\author{Geoffrey Powell}
\address{Univ Angers, CNRS, LAREMA, SFR MATHSTIC, F-49000 Angers, France}
\email{Geoffrey.Powell@math.cnrs.fr}
\urladdr{http://math.univ-angers.fr/~powell/}
\thanks{This work was partially supported by the ANR Project {\em ChroK}, {\tt ANR-16-CE40-0003}.}
\newtheorem{THM}{Theorem}

\newtheorem{thm}{Theorem}[section]
\newtheorem{prop}[thm]{Proposition}

\newtheorem{cor}[thm]{Corollary}
\newtheorem{lem}[thm]{Lemma}
\theoremstyle{definition}
\newtheorem{defn}[thm]{Definition}
\newtheorem{exam}[thm]{Example}

\theoremstyle{remark}
\newtheorem{rem}[thm]{Remark}
 
\newtheorem{nota}[thm]{Notation}
\newtheorem{hyp}[thm]{Hypothesis}

\renewcommand{\epsilon}{\varepsilon}
\renewcommand{\theta}{\vartheta}
\renewcommand{\phi}{\varphi}

\newcommand{\calc}{\mathscr{C}}

\newcommand{\nat}{\mathbb{N}}
\newcommand{\zed}{\mathbb{Z}}
\newcommand{\sym}{\mathfrak{S}}
\newcommand{\rat}{\mathbb{Q}}
\newcommand{\op}{^\mathrm{op}}
\newcommand{\ob}{\mathsf{Ob}\hspace{3pt}}

\newcommand{\End}{\mathrm{End}}
\newcommand{\aut}{\mathrm{Aut}}

\newcommand{\sets}{\mathbf{Set}}

\renewcommand{\hom}{\mathrm{Hom}}
\newcommand{\dash}{\hspace{-3pt}-\hspace{-3pt}}

\newcommand{\sgnrep}{\mathrm{sgn}}
\newcommand{\fb}{\mathbf{\Sigma}}
\newcommand{\finj}{{\mathbf{FI}}}
\newcommand{\fipt}{\finj_*}
\newcommand{\rfi}{\!\downarrow_{\finj}}
\newcommand{\spmon}[1][R]{{\mathcal{S}(#1)}}
\newcommand{\sppt}[1][R]{{\mathcal{S}(#1)}_\bullet}
\newcommand{\modr}{\mathbf{mod}_R}
\newcommand{\Rmod}{\mathbf{Mod}_R}
\newcommand{\df}{\mathsf{d}}
\newcommand{\opd}{\mathscr{O}}
\newcommand{\ppd}{\mathscr{P}}
\newcommand{\assopd}{\mathfrak{Ass}}
\newcommand{\uassopd}{\mathfrak{uAss}}
\newcommand{\lieopd}{\mathfrak{Lie}}
\newcommand{\comopd}{\mathfrak{Com}}
\newcommand{\prelieopd}{\mathfrak{preLie}}
\newcommand{\der}{\mathrm{Der}}
\newcommand{\derpl}{\der^{(1)}_{\mathrm{preLie}}}
\newcommand{\derlie}{\der^{(1)}_{\mathrm{Lie}}}
\newcommand{\derpt}{\der_\bullet}

\newcommand{\derkerphi}{\der_{\langle \mathsf{disjoint} \rangle}}
\newcommand{\alg}{\mathrm{Alg}}
\newcommand{\id}{\mathrm{Id}}

\newcommand{\addelta}{\mathbf{\Phi}}
\newcommand{\gendiv}{\mathsf{Div}}

\newcommand{\vbar}{{\overline{V}}}
\newcommand{\imderlie}[1][\opd]{\mathbf{Im}^{#1}}
\newcommand{\imderliespec}[1][\opd]{\mathbf{Im}^{#1}_{\mathrm{special}}}
\newcommand{\derspec}{\derpt^{\mathrm{special}}}
\newcommand{\rt}{\mathrm{root}}
\newcommand{\tbar}{\overline{\mathsf{T}}}
\newcommand{\trq}[1][-]{\overline{\gendiv^\opd_{#1}}}
\newcommand{\trqlie}[1][-]{\overline{\gendiv^\lieopd_{#1}}}
\newcommand{\trqass}[1][-]{\overline{\gendiv^\assopd_{#1}}}
\newcommand{\rk}{\mathrm{rank}}
\newcommand{\modopd}[1][\opd]{\mathrm{Mod}^{#1}}
\newcommand{\tr}{\mathsf{T}}
\newcommand{\gen}{\mathscr{G}}
\newcommand{\fo}[1][\gen]{\opd_{\langle #1 \rangle}}
\newcommand{\rpt}[1][\gen]{\mathcal{T}^{\mathrm{rp}}_{#1}}
\newcommand{\rptpos}[1][\gen]{\mathcal{T}^{\mathrm{rp+}}_{#1}}
\newcommand{\bin}{B}
\newcommand{\brpt}{\mathcal{T}^{\mathrm{brp}}_{\bin}}
\newcommand{\fob}{\opd_{\langle \bin_3 \rangle}}
\numberwithin{equation}{section}
\begin{document}

\begin{abstract}
For $\mathscr{O}$ a reduced operad, a generalized divergence from the derivations of a free $\mathscr{O}$-algebra to a suitable trace space is constructed. In the case of the Lie operad, this corresponds to 
Satoh's trace map and, for the associative operad, to the double divergence of Alekseev, Kawazumi, Kuno and Naef. The generalized divergence is shown to be a $1$-cocycle for the usual Lie algebra structure on derivations. These results place the previous constructions into a unified framework; moreover, they are natural with respect to the operad.  

An important new ingredient is the use of naturality with respect to the category of finite-rank free modules and split monomorphisms over a commutative ring $R$. This allows the notion of torsion for such functors to be exploited. 

Supposing that the ring $R$ is a PID and that the operad $\mathscr{O}$ is binary,  the main result relates the kernel of the generalized divergence to the sub Lie algebra of the Lie algebra of derivations that is  generated by the elements of degree one  with respect to the grading induced by  arity.
\end{abstract}

\maketitle
\section{Introduction}

For $V$ a free, finite-rank abelian group, Satoh  \cite{MR2864772}  defined  and exploited the trace map 
\[
V^\sharp \otimes \lieopd (V) \rightarrow |T (V)|,
\]
where $V^\sharp$ is the dual of $V$, $\lieopd (V)$ is the free Lie algebra on $V$ and the codomain is the quotient of the tensor algebra $T(V)$ by the subgroup of commutators $[T(V), T(V)]$. One can identify $V^\sharp \otimes \lieopd (V)$ as the module $\der (\lieopd (V))$ of derivations of $\lieopd (V)$, so that the Satoh trace  has the form
\[
\der (\lieopd (V))\rightarrow |T (V)|.
\]
This has been studied by Enomoto and Satoh \cite{MR2846914} and is sometimes referred to as the Enomoto-Satoh trace.

In \cite{MR3758425}, Alekseev, Kawazumi, Kuno and Naef used a related map, the double divergence, $\mathrm{Div}$. For this,  the free Lie algebra $\lieopd (V)$ is replaced by the free associative algebra $T(V)$. (The authors of  \cite{MR3758425} work with the completed algebra $\widehat{T(V)}$, but this distinction is not important here, where only the uncompleted version of the double divergence is considered.) The double divergence  has the form 
\[
\der(T(V)) \rightarrow |T(V)\otimes T(V)\op|,
\]
where $|-|$ again denotes the passage to the quotient modulo commutators. 

To generalize the above, take $\opd$ to be a reduced operad with the arity one operations $\opd (\mathbf{1})$ generated by the unit. In Section \ref{sect:contract_trace} a {\em generalized divergence} is constructed from $\der (\opd (V))$, the derivations of the free $\opd$-algebra on $V$, to a suitable `trace space'. This is of the form
\[
 \gendiv^\opd _{V} : 
 \der (\opd (V)) \rightarrow |\derpt (\opd (R\oplus V))|, 
\]
where $R$ is a commutative ring and $V$ is a  finite-rank free  $R$-module. The codomain is formed from $\derpt (\opd (R\oplus V))$, which is defined using  {\em pointed derivations}; this has a natural unital associative algebra structure, which generalizes that arising from the associative algebra structure on the arity $1$ term $\opd (\mathbf{1})$ of the operad $\opd$.

Using the associative algebra structure, one can pass to the quotient modulo commutators:
 \[
|\derpt (\opd (R \oplus V))|:= \derpt (\opd (R \oplus V)) / [\derpt (\opd (R \oplus V)),\derpt (\opd (R \oplus V))].
\]
The generalized divergence  $\gendiv ^\opd _{V}$ is given by composing the generalized contraction map
\[
 \addelta^\opd _{V} : 
 \der (\opd (V)) \rightarrow \derpt (\opd (R\oplus V))
 \]
that is given in Corollary \ref{cor:naturality_addelta_opd} with the evident quotient map. The construction of the generalized contraction as well as the algebra structure on pointed derivations can be formulated purely in terms of the structure of the operad $\opd$. (Indeed, an alternative approach using the operadic framework is outlined in Appendix \ref{sect:env_alg}.)

The derivations have a natural $\nat$-grading when $\opd$ is reduced. If $\opd (\mathbf{1})=R$, generated by the unit, then in degree zero, $\derpt (\opd (R\oplus V))$ coincides with $R$  and $\der^0(\opd (V))$ identifies with $  \End_R (V) \op$ (here the $(-)\op$ is due to the conventions used for defining the algebraic structure on derivations in terms of the operad structure of $\opd$).
Then, in degree zero, the generalized divergence identifies as 
the usual trace, 
$
\mathrm{Tr} : 
\End_R (V) \op 
\rightarrow 
R$.

The generalized divergence has more structure, which is essential input in the  applications. Namely, the Lie algebra $\der (\opd (V))$ acts on  $|\derpt (\opd (R \oplus V))|$ and one has:

 \begin{THM}[Theorem \ref{thm:1_cocycle}]
\label{THM:cocycle}
The generalized divergence 
  \[
 \gendiv ^\opd _{V}  :
  \der (\opd (V)) \rightarrow 
  |\derpt (\opd (R \oplus V))|
 \]
 is a $1$-cocycle for the Lie algebra $  \der (\opd (V)) $.
 \end{THM}

For the Lie operad, the generalized divergence gives the  Satoh trace and, for the associative operad,  one recovers the  double divergence. Theorem \ref{THM:cocycle} corresponds  to known properties of the Satoh trace and of the double divergence respectively.

For a reduced operad $\opd$, using the $\nat$-grading of the Lie algebra $\der (\opd (V))$, one has the subalgebra $\der^+ (\opd (V))$ of elements of positive degree, termed the positive derivations. Understanding the Lie algebras $\der^+ (\opd (V))$ and $\der (\opd (V))$ is a major goal. 

The Lie structure of $\der (\opd (V))$ arises from a preLie structure. (Recall that a preLie structure is given by a binary operation for which the associator need not vanish, but that  satisfies the right symmetric condition (see Section \ref{subsect:prelie}); this is sufficient for the commutator to define a Lie algebra structure.)  Hence one can form  $\derlie (\opd (V))$, the sub Lie algebra of $\der (\opd (V))$ generated by the elements of degree one, and $ \derpl (\opd (V))$, the sub preLie algebra generated by the elements of degree one. By construction, there are natural inclusions 
\[
\derlie (\opd (V)) \subset \derpl (\opd (V)) \subset \der^+(\opd (V))\subset \der (\opd (V)).
\]

The inclusion $\derlie (\opd (V)) \subset \der^+ (\opd (V))$ is in general a proper inclusion. This makes the following all the more striking: 

\begin{THM}[Theorem \ref{thm:derpl}]
\label{THM:prelie}
Let $\opd$ be a binary operad. Then, for $V$ a free, finite rank $R$-module  such that $\rk_R (V) \neq 1$, the inclusion  
$
\derpl (\opd (V)) 
\hookrightarrow 
\der^+(\opd (V))
$
is an isomorphism.
\end{THM}

Here, the restriction to the binary operad case is necessary, since $\derpl (\opd (V))$ is generated by $\hom (V, \opd_2 (V))$, depending only on the generators $\opd (\mathbf{2})$ of arity two. For example, if $\opd$ is generated non-trivially by ternary operations, $\derpl(\opd (V))=0$, whereas $\der^+ (\opd (V))\neq 0$. 

The result  can be interpreted as showing that all the difficulty in understanding $\derlie (\opd (V))$ comes from the passage from the preLie structure on derivations to the associated Lie structure.

The  generalized divergence can be applied to analyse the sub Lie algebra $\derlie (\opd (V))$, inspired by the main result of \cite{MR2864772}.  A first point is to study the image  $\imderlie (V)$ of  $\derlie (\opd (V))$ in $|\derpt^+ (\opd (R \oplus V)|$ under the generalized divergence. Determining  $\imderlie (V)$ is difficult in general. However, there is an upper bound for $\imderlie (V)$  which is deduced by using Theorem \ref{THM:cocycle} (see Section \ref{subsect:imlie}); this is an important ingredient in the structure results. 

The generalized divergence then gives rise to the key commutative diagram:
\[
\xymatrix{
K^\opd (V) 
\ar@{^(->}[r]
\ar@{^(->}[d]
&
\derlie (\opd (V))
\ar@{^(->}[d]
\ar@{->>}[r]
&
\imderlie (V)
\ar@{^(->}[d]
\\
\mathrm{Ker} \gendiv^\opd _V
\ar@{^(->}[r]
&
\der^+ (\opd (V))
\ar[r]^{\gendiv^\opd _V}
\ar[d]
&
|\derpt^+ (\opd (R \oplus V ) )|
\ar@{->>}[d]
\\
&
|\derpt^+ (\opd (R \oplus V ) )|/ \imderlie (V)
\ar@{=}[r]
&
|\derpt^+ (\opd (R \oplus V ) )|/ \imderlie (V),
}
\] 
in which  the top row is a short exact sequence. Theorem \ref{THM:cocycle} implies that $K^\opd (V) \subset \mathrm{Ker} \gendiv^\opd _V$ are sub Lie algebras of $\der^+ (\opd (V))$.  As explained in Section \ref{subsect:howto}, this diagram forms the basis of the strategy to obtain information on $\derlie (\opd (V))$ and $\der^+ (\opd (V))$. 

These constructions are natural with respect to the category $\spmon$ of split monomorphisms between finite-rank free $R$-modules (see Section \ref{sect:naturality}). The analysis of the  functoriality with respect to $\spmon$  extends consideration of the action of  $\aut_R (V)$ on  $ \der (\opd (V)) $;  the latter already provides a powerful tool (for example, see  \cite{MR2846914} in the case of the Lie operad).

Working with functors on $\spmon$,  one has an appropriate notion of torsion (see Section \ref{sect:torsion}). Explicitly, if $R$ is a PID, for a functor $F$ on $\spmon$ an element $x \in F(V)$ is $t$-torsion, for $t\in \nat$, if  $F(j_t)(x)=0$, where $j_t : V \rightarrow V\oplus R^t$ is the split inclusion. This notion of torsion provides the quantitative content to the following:

\begin{THM}[Cf. Theorem \ref{thm:ses_up_to_torsion}]
\label{THM:main}
Suppose that $R$ is a PID. For $\opd$ a binary operad, the inclusion $\derlie (\opd (V)) \subset 
 \der^+ (\opd (V))$ and the generalized divergence $\gendiv^\opd_V : \der^+ (\opd (V)) 
\longrightarrow |\derpt^+ (\opd (R \oplus V))|$ induce a sequence 
\[
0
\rightarrow 
\derlie (\opd (V)) \rightarrow
 \der^+ (\opd (V)) 
\longrightarrow |\derpt^+ (\opd (R \oplus V))| / \imderlie (V) 
 \rightarrow 0
\]
that is natural in $V \in \ob \spmon$. This is exact up to torsion, as functors on $\spmon$.
\end{THM}
     
The result stated in the text, Theorem \ref{thm:ses_up_to_torsion}, gives a precise bound on the torsion that is independent of $V$; this is essential for the intended applications.

In the case of the Lie operad, Theorem \ref{THM:main} refines to give a  natural sequence
\[
0
\rightarrow 
\derlie (\lieopd (V))
 \rightarrow
  \der^+ (\lieopd (V))
  \rightarrow   
|\overline{T}(V)|/ V   
\rightarrow 
0,
\]
where $\overline{T}(V)$ denotes the augmentation ideal of the tensor algebra. Here the image of $\derlie (\lieopd (V))$ in $|\overline{T}(V)|$ is $V$, in particular is concentrated in degree one.
 The middle homology of this sequence, viewed as a functor of $V\in \ob \spmon$,  is $3$-torsion (see Proposition \ref{prop:ker_Gamma_lieopd}).
 This is related to Satoh's result on the kernel of the Satoh trace (see \cite{MR2864772}). Satoh works with a fixed $V$ and imposes an upper bound on the degree of derivations considered; the usage of  torsion for functors on $\spmon$ circumvents this restriction.

Theorem \ref{THM:main} can  also be refined in the case of the associative operad. In this case, one has a natural sequence 
\[
0
\rightarrow 
\derlie (\assopd (V))
 \rightarrow
  \der^+ (\assopd (V))
  \rightarrow   
 |\overline{T}(V)| \otimes |\overline{T} (V) \op| 
\rightarrow 
0
\]
Now the middle homology, viewed as a functor of $V\in \ob \spmon$, is $4$-torsion  (see Proposition \ref{prop:ker_Gamma_assopd}). The difference as compared to the Lie operad case arises from the fact that $\imderlie [\assopd]$ is  larger than $\imderlie [\lieopd]$ (it is not concentrated in degree one) and has to be taken into account. 

These structures are compatible via the maps induced by the morphism of operads $\lieopd \rightarrow \assopd$ encoding the associated Lie algebra of an associative algebra. Namely, the above sequences fit into the commutative diagram:
\[
\xymatrix{
\derlie (\lieopd (V))
\ar[r]
\ar[d]
&
  \der^+ (\lieopd (V))
 \ar[r]
 \ar[d]
 &
|\overline{T}(V)|/ V   
\ar[d]
\\
\derlie (\assopd (V))
 \ar[r]
 &
  \der^+ (\assopd (V))
  \ar[r]
  &   
 |\overline{T}(V)| \otimes |\overline{T} (V) \op| .
}
\]
Here the right hand  square corresponds to the relationship between the Satoh trace and the double divergence that was one of the inspirations for this work.

One can also consider the case of the (non-unital) commutative operad, $\comopd$. Here, the symmetry of the generating operation means that the behaviour is very different. Indeed, when working over $R= \rat$, Proposition \ref{prop:case_comopd} shows that $\derlie (\comopd (-))$ coincides with $\der^+ (\comopd (-))$ and the generalized divergence is surjective. Thus, for the commutative operad over $\rat$, it is unnecessary to appeal to Theorem \ref{THM:main}.

Many of the proofs of these results reduce to working with the free binary operad $\fob$  on a set $\bin$ (the index $3$ refers to trivalency of vertices - see Appendix \ref{sect:free}). This has the advantage of arising from a set-theoretic operad and, in particular, is encoded by rooted binary planar $\bin_3$-trees (trees with internal vertices labelled by elements of the set $\bin$) and the operation of grafting. This allows the structures which enter into play to be made entirely explicit. More generally, this holds working with the free operad $\fo$ of a graded set $\gen$.

Appendix \ref{sect:env_alg} outlines an alternative operadic construction of the algebra $\derpt (\opd (R\oplus V))$. Namely, to any $\opd$-algebra $A$, one can associate its enveloping algebra $U_\opd A$.  Taking $A$ to be the free $\opd$-algebra on $V$, one has the enveloping algebra $U_\opd \opd (V)$ and this is naturally isomorphic to the algebra of pointed derivations introduced above. The body of the  text works in terms of pointed derivations, for which the additional structure that is required is more transparent.

\subsection{Organization of the paper}
The paper is presented in three parts with two appendices. 

Part \ref{part:background} covers background: the notions of naturality that are required are introduced in Section \ref{sect:naturality} and torsion is reviewed in Section \ref{sect:torsion}; Sections \ref{sect:fb} and  \ref{sect:opd} introduce derivations for operads and their properties.

Part \ref{part:gen_trace} is dedicated to the generalized divergence and its properties, as well as introducing the subalgebras that are studied in the third part. The  preLie structure on derivations is introduced in Section \ref{sect:prelie} and pointed derivations in Section \ref{sect:pointed}, where it is shown that the corresponding preLie structure is associative. The generalized contraction map and the divergence are introduced in Section \ref{sect:contract_trace}, where the $1$-cocycle condition, Theorem \ref{THM:cocycle}, is established.  Section \ref{sect:structure} introduces the sub Lie algebra of derivations generated by degree one, together with the preLie version of this construction;  Section \ref{subsect:imlie} studies the image of this sub Lie algebra  under the generalized divergence and Section \ref{subsect:howto} explains the general strategy for analysing these structures that is applied in Part \ref{part:structure} in the case of a binary operad.

Part \ref{part:structure} contains the main structural results for the case of a binary operad. Section \ref{sect:structure} provides techniques for working with binary operads and gives the proof of Theorem \ref{THM:prelie}.  Sections \ref{sect:derlie} and \ref{sect:kertrace} are devoted to the analysis of the sub Lie algebra of derivations that is generated in degree one and the proof of  Theorem \ref{THM:main}

Appendix \ref{sect:free}, reviews material on planar trees and free operads that is used in the text and Appendix \ref{sect:env_alg} outlines an alternative approach to the algebra structure on pointed derivations.

\subsection{Standard notation}

\begin{itemize}
\item[-]
$R$ always denotes a unital, commutative ring.
\item[-] 
For $A$ an associative $R$-algebra, 
 $
|A|$ denotes the quotient $A/ [A,A]$,  where $[A,A]$ is the sub $R$-module generated by the commutators $[x,y]:= xy - yx$. 
\item[-] 
For $n \in \nat$, $\mathbf{n}$ denotes the set $\{1, \ldots , n\}$. The automorphism group of $\mathbf{n}$ is denoted $\sym_n$. 
\item[-] 
For $m \leq n \in \nat$, unless indicated otherwise, $\sym_m \subset \sym_n$ denotes the inclusion of groups corresponding to the canonical inclusion $\mathbf{m} \subseteq \mathbf{n}$.  
\item[-] 
For finite groups $H \subset G$, $\downarrow_H^G$ denotes restriction and $\uparrow _H^G$ induction.
\end{itemize}

\subsection{Acknowledgements} The author owes a clear debt to Takao Satoh, since many of the arguments here have been inspired by reading and reinterpreting \cite{MR2864772}. The germ of the idea that these results might fit into a general operadic framework was planted by a talk by Nariya Kawazumi at Strasbourg in February 2020, which used   a compatibility between the Satoh trace and the double divergence. The author thanks  Nariya Kawazumi and Takao Satoh for their interest. 

He is especially grateful to Christine Vespa for numerous comments on earlier versions of this document; in particular,  these have helped improve the exposition.

\tableofcontents

\part{Background}
\label{part:background}

\section{Naturality} 
\label{sect:naturality}

Fix $R$ a commutative, unital ring and let $\Rmod$ denote the category of left $R$-modules and  $\modr$ the full subcategory  with objects  the free, finite-rank $R$-modules.

\subsection{Duality for $\modr$}
 This Section serves to review some basic results for duality of $R$-modules. 

\begin{nota}
\label{nota:duality_modr}
For $V \in \ob \modr$, let $V^\sharp$ denote the dual  $R$-module $\hom _R (V, R)$, which is an object of $\modr$ that is  non-canonically isomorphic to $V$. 
\end{nota}

\begin{prop}
\label{prop:duality_modr}
The duality functor $^\sharp : \modr \op \rightarrow \modr$ is an equivalence of categories. 
\end{prop}

\begin{lem}
\label{lem:strong_duality}
For $V \in \ob \modr$ and $M \in \ob \Rmod$, 
\begin{enumerate}
\item 
the natural double duality morphism $V \rightarrow (V^\sharp) ^\sharp$ given by $v \mapsto (f \mapsto f(v))$ for $v \in V$ and $f \in V^\sharp$ is an isomorphism;
\item 
the natural morphism $V^\sharp \otimes M \rightarrow \hom_R (V, M)$ given by  $f \otimes m \mapsto (v \mapsto f(v)m)$, for $f \in V^\sharp, m \in M, v \in V$, is an isomorphism.
\end{enumerate} 
\end{lem}

\begin{prop}
\label{prop:strong_duality}
For $V \in \ob \modr$ and $M, N \in \ob \Rmod$, there is a natural isomorphism:
\[
\hom_R (N \otimes V , M)
\cong 
\hom_R (N, V^\sharp \otimes M).
\]
\end{prop}

\begin{proof}
The standard natural isomorphism $\hom_R (N \otimes V , M) \cong \hom_R (N, \hom_R (V, M))$ sends a morphism $f : N \otimes V \rightarrow M$ to $g : N \rightarrow \hom_R(V, M)$ given by $g(n) (v) := f(n \otimes v)$, for $n \in N$ and $v \in V$. By Lemma \ref{lem:strong_duality}, $\hom_R(V,M)$ is isomorphic to $V^\sharp \otimes M$.  
\end{proof}

\begin{rem}
\label{rem:strong_duality}
In the statement of Proposition \ref{prop:strong_duality}, using the double duality isomorphism of Lemma \ref{lem:strong_duality}, one can replace $V$ by $V^\sharp$, giving the natural isomorphism:
\[
\hom_R (N \otimes V^\sharp , M)
\cong 
\hom_R (N, V \otimes M).
\]
\end{rem}

\subsection{The category $\spmon$} 
Not all the constructions of this paper are  functorial with respect to $\modr$. Frequently  one has to work with the category of split monomorphisms introduced below; the fact that using split monomorphisms provides a suitable context has already been remarked upon (see \cite[Remark 2.36]{MR3975077}, for example). 

\begin{defn}
\label{defn:spmon}
Let $\spmon$ denote the category with free, finite-rank $R$-modules for objects and $\hom_{\spmon} (V, W) = \{ (i:V\rightarrow W,\  r: W \rightarrow V) \ | \ r i = 1_V \}$, the set of split monomorphisms. The notation $(i,r)$ will be used to denote an element of this set.
\end{defn}

\begin{rem}
\ 
\begin{enumerate}
\item 
Forgetting the retract provides a forgetful functor $\spmon \rightarrow \modr$, which takes values in the subcategory of monomorphisms.
\item 
For $s, t \in \nat$, $\hom_{\spmon} (R^s, R^t)=0$ if $s>t$.
\item 
$\spmon$ is an EI-category (i.e., all endomorphisms are isomorphisms). More explicitly, for $V  \in \ob \modr$,   $ \hom _{\spmon} (V, V)\cong\aut_{\modr} (V) $; an automorphism $\alpha$ corresponds to the pair $(\alpha, \alpha^{-1})$.
\item 
For $(i,r) \in \hom_{\spmon}(V ,W)$,  $i$ and $r$ induce an isomorphism $W \cong V \oplus \mathrm{coker}(i)$. In particular, $\mathrm{coker}(i)$ is a finitely-generated projective $R$-module; it is stably-free (see \cite[Definition I.1.2]{Weibel_K}) but is not, in general, a free  $R$-module. 

Over certain commutative rings, all finitely-generated stably-free modules are free;  for example, this holds if $R$ is a PID.
\end{enumerate}
\end{rem}

The following  builds upon Proposition \ref{prop:duality_modr}:

\begin{prop}
\label{prop:duality:spmon}
\ 
\begin{enumerate}
\item 
The duality functor $^\sharp$ induces an equivalence of categories $^\sharp :\spmon \rightarrow \spmon$ that sends $(i,r) \in \hom_{\spmon} (V, W)$ to $(r^\sharp, i^\sharp) \in \hom_{\spmon} (V^\sharp, W^\sharp)$ 
\item 
In particular, the functor $^\sharp$ induces a functor $\spmon \rightarrow \modr$, sending a morphism $(i, r) \in \hom_{\spmon} (V, W)$ to the morphism $r^\sharp : V^\sharp \rightarrow W^\sharp$.
\end{enumerate} 
\end{prop}

Moreover, one has the  following standard result, which is key for comparing the notions of torsion in Section \ref{sect:torsion} (see Proposition \ref{prop:equality_kappa}).

\begin{prop}
\label{prop:spmon_transitivity}
Suppose that all finitely-generated stably-free $R$-modules are free. Then, for $s \leq t$, $\hom_{\spmon} (R^s, R^t)$ is a transitive $\aut_{\modr} (R^t)$-set, generated by the morphism corresponding to the canonical splitting $R^t = R^s \oplus R^{t-s}$. 
\end{prop}

One significance of $\spmon$ here is that it allows the construction of  diagonal functors  associated to a bifunctor, as follows:

\begin{lem}
\label{lem:janus}
Let $G :  \modr \op \times \modr \rightarrow \calc $ be a bifunctor with values in a category $\calc$. 
There are natural associated functors: $\df G : \spmon \rightarrow \calc $ and $\df\op G: \spmon\op \rightarrow \calc$, where 
$\df G (V) = \df \op G (V) = G (V, V) $ and, for $(i,r): V \rightarrow W$,  $(\df G) ((i,r)) = G(r,i) : G(V, V) \rightarrow G(W,W)$ and $(\df\op G) ((i, r))= G(i, r) : G(W,W) \rightarrow G(V,V)$.

Moreover, the composite $(\df \op G)((i, r)) \circ (\df G)((i,r)) : G(V,V) \rightarrow G(V,V) $ is the identity. 
\end{lem}

\subsection{Restricting to $\finj$}
\label{subsect:restrict_to_finj}

Certain proofs of this paper are carried out by restricting functoriality to the category of finite sets and injections, via the free $R$-module functor. This has the advantage that it allows the dual basis to be exploited, via Proposition \ref{prop:duality_FI}.

\begin{nota}
Let $\finj$ denote the category of finite sets and injections.
\end{nota}

\begin{rem}
\ 
\begin{enumerate}
\item 
The category $\finj$ is an EI-category.
\item 
A functor from $\finj$  to $\Rmod$ is referred to as an $\finj$-module, the ring $R$ usually being understood from the context.
\end{enumerate}
\end{rem}

\begin{lem}
\label{lem:finj_embed_spmon}
The free $R$-module functor $R[-]$ induces a faithful embedding
$ 
R[-] : \finj 
\rightarrow 
\spmon
$ 
that sends an injection of finite sets $i: S \hookrightarrow T$ to the pair $( R[i], r(i) ) : R[S] \rightarrow R[T]$, where $r(i)$ is the retract that sends the generators in $T \backslash i (S)$ to zero.
\end{lem}

\begin{rem}
The functor $R[-] : \finj 
\rightarrow 
\spmon$ induces a restriction functor from functors on $\spmon$  to functors on $\finj$.
\end{rem}

The following Proposition makes explicit the close relationship between $\finj$ and $\spmon$, denoting by $\spmon^{\mathrm{iso}}$ the maximal subgroupoid of  $\spmon$ (i.e., the subcategory that contains all the objects and in which the morphisms are the isomorphisms of $\spmon$).

\begin{prop}
\label{prop:spmon_finj}
Suppose that all finitely-generated stably-free $R$-modules are free. Then the smallest subcategory of $\spmon$ containing  $\spmon^{\mathrm{iso}}$ and 
the essential image of $R[-] : \finj \rightarrow \spmon$ is the category $\spmon$ itself.
\end{prop}

\begin{proof}
This follows directly from Proposition \ref{prop:spmon_transitivity}.
\end{proof}

The following duality property is one advantage of working with $\finj$ rather than $\spmon$:

\begin{prop}
\label{prop:duality_FI}
The association $S \mapsto (R[S])^\sharp$ defines a functor 
$ 
 (R[-])^\sharp : \finj \rightarrow \spmon$.
 The functor $(R[-])^\sharp$ is naturally isomorphic to $R[-] : \finj \rightarrow \spmon$.
\end{prop}

\begin{proof}
The first statement follows by combining the functor of Lemma \ref{lem:finj_embed_spmon} with the equivalence of categories $^\sharp :\spmon \rightarrow \spmon$
of Proposition \ref{prop:duality:spmon}. Explicitly: given $S \hookrightarrow T$, the inclusion 
 $ (R[S])^\sharp \hookrightarrow  (R[T])^\sharp$ is the dual to the projection $R[T] \twoheadrightarrow R[S]$ and the projection $(R[T])^\sharp \twoheadrightarrow (R[S])^\sharp$ is the dual of the $R$-linearization $R[S] \rightarrow R[T]$. 

For the second statement, one checks that the dual basis gives an isomorphism 
$
(R[S])^\sharp \cong R[S]$  
that is natural as functors from $\finj$ to $\spmon$.
\end{proof}

\begin{rem}
The second statement of the Proposition should be compared with the case of the duality functor of Proposition \ref{prop:duality:spmon}, which is an equivalence of categories $^\sharp : \spmon \rightarrow \spmon$. This is {\em not} naturally equivalent to the identity functor, since there is no natural isomorphism $V^\sharp \cong V$ in $\modr$. 
\end{rem}

\subsection{Pointed variants}
\label{subsect:pointed}

We will use  pointed variants of $\finj$ and $\spmon$, notably in introducing the notion of {\em pointed} derivations (see Section \ref{sect:pointed}). 
 Propositions \ref{prop:fipt} and \ref{prop:relate_sppt} show that the pointed categories are closely related to their respective non-pointed versions. 

\begin{nota}
\label{nota:fipt}
Let $\fipt$ be the category of finite pointed sets and basepoint preserving injections. 
Write $(S, z)$ for a  finite set $S$ with basepoint $z \in S$.
\end{nota} 

\begin{rem}
\label{rem:fipt}
\ 
\begin{enumerate}
\item 
$\fipt$ is equivalent to the undercategory $\mathbf{1}/ \finj$.
\item 
Forgetting the basepoint gives a forgetful functor $\fipt  \rightarrow \finj$. 
\item 
Adding a disjoint basepoint $S \mapsto S_+$ induces a functor $(-)_+ : \finj \rightarrow \fipt$. 
\end{enumerate}
\end{rem}

\begin{prop}
\label{prop:fipt}
The functor $(-)_+ : \finj \rightarrow \fipt$ is an equivalence of categories. 
\end{prop}

\begin{proof}
The functor $(-)_+$ is clearly essentially surjective, since a finite pointed set $(S, y)$ is isomorphic to $(S\backslash \{y \})_+$. Hence it remains to prove that the functor $(-)_+$ is fully faithful. 

Fidelity is clear; to show that it is full, consider a morphism $(S,y) \rightarrow (T,z)$ of $\fipt$. Since $y \mapsto z$, this morphism is determined by its restriction to $S\backslash \{y\}$. Since the underlying map is injective, this must have the form $S\backslash \{y\} \hookrightarrow T\backslash \{z\} \subset T$. From this one concludes rapidly.
\end{proof}

Similarly, one can consider the undercategory $R/ \spmon$, equipped with the forgetful functor $R/ \spmon \rightarrow \spmon$.

\begin{rem}
\label{rem:sppt}
An object of $R/ \spmon$ is a free, finite-rank $R$-module $V$ equipped with a split monomorphism 
$\xymatrix{R \ar@{_(->}[r] & V \ar@<-1ex>[l] }.$ 
This may be denoted $(V, R)$, where the structure morphisms are clear. One has the induced splitting $V \cong \overline{V} \oplus R$, where $\overline{V}$ is a a finitely-generated projective $R$-module which is stably-free. 
\end{rem}

\begin{defn}
\label{defn:sppt}
Let $\sppt$ be the full subcategory of $R / \spmon$ of objects $(V, R)$ such that $\overline{V}\cong  V/R$ is a free $R$-module, with associated forgetful functor $\sppt \rightarrow \spmon$.
\end{defn}

\begin{rem}
If all finitely-generated stably-free modules are free, then $\sppt = R / \spmon$. 
\end{rem}

The categories of interest are related by the following, the $\spmon$-analogue of Proposition \ref{prop:fipt}:

\begin{prop}
\label{prop:relate_sppt}
\ 
\begin{enumerate}
\item 
The functor $ R \oplus - \ : \modr \rightarrow \modr$ refines to a functor 
$
 R \oplus - \ : \spmon \rightarrow \sppt
 $, $W \mapsto (R \oplus W, R)$, 
 where the structure morphisms $\xymatrix{R \ar@{_(->}[r] & R \oplus W \ar@<-1ex>[l] }$ are given by the canonical inclusion and projection. 
\item 
The functor $ R \oplus - \ : \spmon \rightarrow \sppt$ is an equivalence of categories. 
\end{enumerate}
\end{prop}

\begin{proof}
The first  statement is straightforward. 

For the second, it is clear that $R \oplus -  : \spmon \rightarrow \sppt$ is essentially surjective and faithful, hence to show that it is an equivalence of categories, we require to  show that it is full. 

Consider a morphism $(i, r): (V,R) \rightarrow (W, R)$ of $\sppt$. This corresponds to a diagram: 
\[
\xymatrix{
R 
\ar[r]
\ar@{=}[d]
\ar@{}[rd]|\circlearrowright
&
V
\ar[r]
\ar[d]^i 
&
R 
\ar@{=}[d]
\\
R
\ar@{=}[d]
\ar[r]
&
W
\ar[r]
\ar[d]^r
\ar@{}[rd]|\circlearrowright
&
R
\ar@{=}[d]
\\
R
\ar[r]
&
V\ar[r]
&
R,
}
\]
in which the rows are given by the structure morphisms of $(V,R)$ and $(W,R)$,  the indicated squares commute and the horizontal and vertical composites are the identity. One checks that the remaining two squares also commute, so that the diagram is commutative. 

Write $\overline{V}$ (respectively $\overline{W}$) for the kernel of the structure morphism $V \rightarrow R$ (resp. $W \rightarrow R$); by hypothesis these lie in $\modr$. Then the commutative diagram shows that $i$ restricts to $\overline {i} : \overline{V} \rightarrow \overline{W}$ and $r$ restricts to $\overline{r} : \overline{W} \rightarrow \overline{V}$, 
giving a morphism $(\overline{i}, \overline{r}) \in \hom_{\spmon}(\overline{V}, \overline{W})$. Moreover, the morphism $(i,r)$ is the image of $(\overline{i}, \overline{r})$ under $R \oplus - $, using the canonical isomorphisms $V \cong R \oplus \overline{V}$ and $W \cong R \oplus \overline{W}$. 
\end{proof}

\begin{prop}
\label{prop:R_free_pt}
 The free $R$-module functor  induces a faithful embedding
$
R[-] : \fipt 
\rightarrow 
\sppt
$. 
This fits into a diagram that is  commutative up to natural isomorphism:
\[
\xymatrix{
\finj 
\ar[r]^{R[-]} 
\ar[d]_{(-)_+}
&
\spmon 
\ar[d]^{ R \oplus - } 
\\
\fipt
\ar[r]_{R[-]} 
&
\sppt,
}
\] 
in which the vertical functors are  given by Propositions \ref{prop:fipt} and \ref{prop:relate_sppt}.
\end{prop}

The following notation will be used throughout the paper:

\begin{nota}
\label{nota:fipt_sppt}
For $(S,z)\in \ob \fipt$, write $R[S, z]$ for the corresponding object of $\sppt$, with corresponding splitting $R[S,z] \cong R[S\backslash\{z\}] \oplus R z $, where $Rz$ is the free $R$-module generated by $z$. 
\end{nota}

\section{Torsion}
\label{sect:torsion}

This Section reviews the  notions of torsion that are used in formulating the main results. 
 The study of torsion for $\finj$-modules is a standard technique that is of significant interest in its own right.  

\subsection{Torsion for functors on $\finj$}
\label{subsect:finj_torsion}

\begin{defn}
\label{defn:torsion_finj}
For $F$ a functor $\finj \rightarrow \Rmod$, 
\begin{enumerate}
\item 
an element $x \in F(S)$ is torsion if there exists $i: S\rightarrow T$ in $\finj$ such that $F(i) (x)=0$; 
\item 
$F$ is torsion if every element is torsion. 
\end{enumerate}
\end{defn}

\begin{rem}
The full subcategory of torsion $\finj$-modules is abelian and is a localizing subcategory. The latter point allows one to localize away from the torsion $\finj$-modules. This is not exploited here, since we are interested in bounding the torsion {\em explicitly}.
\end{rem}

The notion of torsion is refined using the following, for which we recall that $\mathbf{n}= \{ 1, \ldots, n \}$:

\begin{lem}
\label{lem:i_n}
For $n \in \nat$, disjoint union of finite sets induces a functor $- \amalg \mathbf{n} : \finj \rightarrow \finj$. There is a natural transformation
$ 
i_n : \mathrm{Id}_{\finj} \rightarrow  (- \amalg \mathbf{n})
$ 
given by the canonical inclusion $S \hookrightarrow S \amalg \mathbf{n}$, for $S \in \ob \finj$. 
\end{lem}

\begin{defn}
\label{defn:n-torsion_finj}
For $F$ a functor $\finj \rightarrow \Rmod$ and $n \in \nat$,
\begin{enumerate}
\item 
an element $x \in F(S)$ is $n$-torsion if $F(i_n) (x)=0$;
\item 
$F$ is $n$-torsion if every element of $F$ is $n$-torsion. 
\end{enumerate}
\end{defn}

\begin{rem}
\ 
\begin{enumerate}
\item 
An $n$-torsion functor is, in particular, a torsion functor.
\item 
An element $x \in F(S)$ is torsion if and only if there exists $n \in \nat$ such that $F(i_n) (x)=0$.
\item
A functor $F$ is $0$-torsion if and only if it is zero.
\item 
An $n$-torsion functor is $m$-torsion for any $m \geq n \in \nat$.
\end{enumerate}
\end{rem}

\begin{prop}
\label{prop:n_torsion_finj}
For $n \in \nat$, a functor $F : \finj \rightarrow \Rmod$ is $n$-torsion if and only if 
$
F(i_n) : F \rightarrow F \circ (-\amalg \mathbf{n})
$
is zero.
\end{prop}

One has the notion of surjectivity up to torsion:

\begin{defn}
\label{defn:surj_torsion}
A natural transformation $\phi : F \rightarrow G$ between functors from $\finj$ to $\Rmod$ is 
\begin{enumerate}
\item 
surjective up to torsion if $\mathrm{coker}\  \phi$ is torsion; 
\item 
$n$-surjective, for $n \in \nat$, if $\mathrm{coker} \ \phi$ is $n$-torsion.
\end{enumerate}
\end{defn}

\begin{rem}
The natural transformation $\phi$ is $0$-surjective if and only if it is surjective. If $\phi$ is $n$-surjective, for some $n \in \nat$, then it is surjective up to torsion.
\end{rem}

\subsection{Torsion for functors on $\spmon$}
There are analogous notions of torsion for functors on $\spmon$.

\begin{defn}
\label{defn:torsion_spmon}
Let $F$ be a functor $\spmon \rightarrow \Rmod$. 
\begin{enumerate}
\item 
An element $x \in F(V)$ is torsion if the smallest subfunctor $\langle x \rangle \subset F$ containing $x$ has finite support (i.e., if $\langle x \rangle (W) = 0$ for $\rk_R W \gg 0$).
\item 
$F$ is torsion if every element is torsion.
\end{enumerate}  
Suppose that all finitely-generated stably-free $R$-modules are free and let $n \in \nat$. 
\begin{enumerate}
\item 
The functor $F$ is $n$-torsion if $F(j_n) : F \rightarrow F\circ (- \oplus R^n)$ is zero, where $j_n : V \hookrightarrow V \oplus R^n$ is the natural split inclusion in $\spmon$. 
\item 
A natural transformation $\phi : F \rightarrow G$ of functors on $\spmon$ is $n$-surjective if $\mathrm{coker}\  \phi$ is $n$-torsion.
\end{enumerate}
\end{defn}

Given $F$ as above, as in Section \ref{subsect:restrict_to_finj}, one can consider the restriction 
$F\rfi : \finj \rightarrow \Rmod$. The  respective notions of torsion are compatible:

\begin{prop}
\label{prop:equality_kappa}
Suppose that all finitely-generated stably-free $R$-modules are free. 
Then for $n\in \nat$:
\begin{enumerate}
\item 
a functor $F : \spmon \rightarrow \Rmod$ is $n$-torsion if and only if $F \rfi : \finj \rightarrow \Rmod$ is $n$-torsion;
\item 
a natural transformation $F \rightarrow G$ of functors on $\spmon$ is $n$-surjective if and only if $F\rfi \rightarrow G\rfi$ is $n$-surjective as functors on $\finj$;
\item
if $F$ is $n$-torsion, then it is a torsion functor on $\spmon$.
\end{enumerate}
\end{prop}

\begin{proof}
This follows directly from Proposition \ref{prop:spmon_transitivity}.
\end{proof}

\section{$\fb\op$-modules and Schur functors}
\label{sect:fb}

This Section reviews the framework underlying algebraic operads. 
 
\subsection{Basic structure}

Let $\fb$ be the category of finite sets and bijections.
The following definitions are standard:

\begin{defn}
\label{defn:fbop-modules_tensor}
\ 
\begin{enumerate}
\item 
The category of $\fb\op$-modules is the category of functors from $\fb\op$ to $\Rmod$.
\item 
The tensor product $\otimes$ of $\fb \op$-modules is given for $\fb\op$-modules $B_1$, $B_2$ by:
\[
(B_1 \otimes B_2) (S) := \bigoplus _{S = S_1 \amalg S_2} B_1(S_1)  \otimes B_2(S_2),
\]
where the sum ranges over decompositions of the finite set $S$ into two subsets.
\end{enumerate}
\end{defn}

\begin{rem}
The category $\fb$ has a small skeleton with objects $\{ \mathbf{n} \ |\  n \in \nat \}$. Thus the category of $\fb\op$-modules is equivalent to the category of symmetric sequences: this has objects given by sequences $\{ B(\mathbf{n}) | n \in \nat \}$ of right $R[\sym_n]$-modules; morphisms are equivariant morphisms between such sequences.
\end{rem}

The Schur functor construction defines a functor from $\fb\op$-modules to functors from $\Rmod$ to $\Rmod$:

\begin{defn}
\label{defn:schur_functor}
For $B$ a $\fb\op$-module, the associated Schur functor $B(-)$ is defined on $V \in \ob \Rmod$ by
$ 
 B (V) := 
 \bigoplus _{n\in \nat} B_n (V)$, 
 where $B_n (V):= B (\mathbf{n}) \otimes_{\sym_n} V^{\otimes n}.$

The Schur functor $V \mapsto B(V)$ is  $\nat \cup \{-1\}$-graded by placing $B_n (V)$ in degree $n-1$. If $B (\mathbf{0}) =0$, then it is $\nat$-graded.
\end{defn}

\begin{rem}
\ 
\begin{enumerate}
\item
The notation $B(-)$ is used here, for $B$ a $\fb\op$-module, to indicate two different structures: the underlying functor on $\fb\op$ as well as its associated Schur functor. The context should make clear which interpretation is intended. 
\item 
For many of the proofs of this Section, one can reduce to the case where $B$ is concentrated in a single arity, i.e., there exists $n \in \nat$ such that $B(\mathbf{m})=0$ if $m\neq  n$. In this case, the associated Schur functor is just $B_n (-)$.
\end{enumerate}
\end{rem}

\begin{exam}
\label{exam:opd_lie_uass_ass}
Important examples of  $\fb\op$-modules are derived from algebraic operads, such as the following:
\begin{enumerate}
\item 
The (non-unital) commutative operad, $\comopd$, which has $\comopd (\mathbf{0}) = 0$ and, for $n>0$, $\comopd (\mathbf{n})=R$, the trivial representation of $\sym_n$.
 The associated Schur functor identifies as the augmentation ideal $\overline{S} (V)$ of the free commutative algebra $S(V)$ on $V$.
\item 
The Lie operad, $\lieopd$; for small arities one identifies $\lieopd (\mathbf{0})=0$, $\lieopd (\mathbf{1})=R$
 and $\lieopd (\mathbf{2}) = \sgnrep_2$, the signature representation of $\sym_2$. The associated Schur functor, $\lieopd (V)$, gives the free Lie algebra on $V$. (See \cite[Section 13.2]{LV}  for $\lieopd$ and   \cite{MR1231799} for free Lie algebras.) 
\item 
The unital associative operad, $\uassopd$; for  $0\leq n \in \nat$, $\uassopd (\mathbf{n}) \cong R[\sym_n]$. The Schur functor $\uassopd (V)$ is the free associative, unital algebra on $V$, which identifies as the tensor algebra $T(V)$. 
\item 
The (non-unital) associative operad, $\assopd$; for  $0<n \in \nat$, $\assopd (\mathbf{n}) \cong R[\sym_n]$, whereas $\assopd (\mathbf{0}) =0$. The Schur functor $\assopd (V)$  identifies as the augmentation ideal  $\overline{T}(V)\subset T(V)$, the free associative non-unital algebra on $V$. 
\end{enumerate}
\end{exam}

The tensor product of $\fb\op$-modules and the tensor product of functors on $\Rmod$ are compatible via the Schur functor construction:

\begin{prop}
\label{prop:tensor_fbop_schur}
For $\fb\op$-modules $B_1, B_2$ with tensor product $B_1 \otimes B_2$ as an $\fb\op$-module, there is a natural isomorphism in $V \in \ob \Rmod$: 
\[
(B_1 \otimes B_2) (V) 
\cong 
B_1 (V) \otimes B_2 (V).
\]
\end{prop}

\begin{proof}
This is proved as \cite[Proposition 5.1.5]{LV} for the case $R$ a field. The general case is proved using the same argument.
\end{proof}

For $B$ a $\fb\op$-module,  precomposing the Schur functor $B(-)$ with the functor $\oplus : \Rmod \times \Rmod \rightarrow \Rmod$ gives the bifunctor $(V, W) \mapsto B (V \oplus W)$ with values in $\Rmod$.  

For $n \in \nat$, there is a natural isomorphism of left $\sym_n$-modules:
\[
(V \oplus W) ^{\otimes n} \cong \bigoplus_{i+j =n} (V^{\otimes i} \otimes W^{\otimes j} )\uparrow_{\sym_i \times \sym_j}^{\sym_n},
\]
where 
$\sym_i$ is $\aut (\mathbf{i})$ and $\sym_j$ is $\aut (\mathbf{n} \backslash \mathbf{i})$, considered as subgroups of $\sym_n= \aut (\mathbf{n})$.

This leads to the natural decomposition of $B (V \oplus W)$:
\begin{eqnarray}
\label{eqn:bifunctor}
B (V \oplus W) \cong \bigoplus _{n \in \nat} \bigoplus_{i+j =n}
 \Big(
B(\mathbf{n}) \downarrow^{\sym_n} _{\sym_i \times \sym_j}
\Big)
 \otimes_{\sym_i \times \sym_j} (V^{\otimes i} \otimes W^{\otimes j}). 
\end{eqnarray}

\begin{defn}
\label{defn:BM;N}
For $B$ a $\fb\op$-module, let $B (-;-)$ be the functor on $\Rmod^{\times 2}$ such that, for 
$(V, W) \in \ob \Rmod^{\times 2}$, $B(V; W)$ is the direct summand of $B (V \oplus W)$ of terms that are linear in $W$:
\[
B (V ; W)
:= 
 \bigoplus _{0<n \in \nat}
 \Big(
B(\mathbf{n}) \downarrow^{\sym_n} _{\sym_{n-1}}
\Big)
 \otimes_{\sym_{n-1}} (V^{\otimes n-1} \otimes W).
\]

Let $V \mapsto B (V; V)$ be the functor $\Rmod \rightarrow \Rmod$ given by precomposing the bifunctor $B (-;-)$ with the diagonal functor $\Rmod \rightarrow \Rmod^{\times 2}$, $V \mapsto (V,V)$.
\end{defn}

\begin{rem}
The bifunctor $B(-;-)$ can be viewed as a special case of the construction of the infinitesimal composition product for $\fb\op$-modules that is given in \cite[Section 6.1.1]{LV}. 
\end{rem}

\begin{defn}
\label{defn:sigma_tau}
Let $\tau$, $\sigma$ be the endofunctors of $\fb\op$-modules defined for  $B$ a $\fb\op$-module and $n \in \nat$ by  
\begin{enumerate}
\item
$
\tau B (\mathbf{n} ) := B(\mathbf{n+1}) \downarrow^{\sym_{n+1}}_{\sym_n};
$
\item 
$\sigma B (\mathbf{0}) :=0$ and, for $n >0$,
$
\sigma B (\mathbf{n} ) := B(\mathbf{n-1}) \uparrow_{\sym_{n-1}}^{\sym_n}.
$
\end{enumerate}
\end{defn}

\begin{prop}
\label{prop:sigma_tau_adjunctions}
The functor $\sigma$ is  both left and right adjoint to $\tau$. 
\end{prop}

\begin{proof}
That $\sigma$ is left adjoint to $\tau$ is tautological, since induction is defined to be the left adjoint to restriction. For finite groups, induction is naturally equivalent to coinduction,  hence $\sigma$ is also right adjoint to $\tau$.
\end{proof}

\subsection{Identifying the linear bifunctor in terms of $\fb\op$-modules}

\begin{lem}
\label{lem:sigma_fbop}
For $B$ a $\fb\op$-module and $V \in \ob \Rmod$, there is a natural isomorphism 
$
\sigma B (V) \cong B(V) \otimes V.
$ 
\end{lem}

\begin{proof}
By definition of $\sigma B$, $\sigma B (V) = \bigoplus_{0< n \in \nat} B(\mathbf{n-1}) \uparrow _{\sym_{n-1}}^{\sym_n} \otimes _{\sym_n} V^{\otimes n}$. The right hand side is isomorphic to 
\[
 \bigoplus_{0< n \in \nat} (B(\mathbf{n-1})  \otimes _{\sym_{n-1}} V^{\otimes n-1 }) \otimes V.
\]
Reindexing and using that $\otimes$ distributes over $\bigoplus$, this is isomorphic to $B(V) \otimes V$, as required.
\end{proof}

\begin{prop}
\label{prop:sigma_tau_BM;M}
For $B$ a $\fb\op$-module and $V , W \in \ob \Rmod$, there is a natural isomorphism 
$
B( V; W) \cong  \tau B (V) \otimes W.
$

As functors on $\Rmod$ (using the structure given by Definition \ref{defn:BM;N} for the domain),  there is a natural isomorphism 
$ 
B (V;V) \cong \sigma \tau B (V).
$ 
\end{prop}

\begin{proof}
The first statement follows from the explicit description of $B (V; W)$ arising from (\ref{eqn:bifunctor}). The second statement then follows from Lemma \ref{lem:sigma_fbop}.
\end{proof}

\begin{exam}
\label{exam:tau_opd}
Consider the operads introduced in Example \ref{exam:opd_lie_uass_ass}.
\begin{enumerate}
\item 
For $\comopd$ and $0<n \in \nat$, $\comopd(\mathbf{n}) \downarrow^{\sym_n}_{\sym_{n-1}}$ is the trivial $\sym_{n-1}$-module $R$. In particular, for  $V \in \ob \modr$, there is a natural isomorphism $\tau \comopd (V) \cong S(V)$  of functors from $\modr$ to $\Rmod$.
\item 
For $\lieopd$ and $0<n \in \nat$, $\lieopd(\mathbf{n}) \downarrow^{\sym_n}_{\sym_{n-1}}$ is isomorphic as a right $\sym_{n-1}$-module to $R[\sym_{n-1}]$ \cite{MR1231799}. (This may be seen by considering the basis of $\lieopd (\mathbf{n})$ indexed by iterated commutators of the form $[x_{\zeta(1)},[x_{\zeta (2)}, [\ldots, [x_{\zeta(n-1)} , x_n]\ldots ]$, for $\zeta \in \sym_{n-1}$.) Hence the underlying $\fb\op$-module of $\tau \lieopd$ is that of $\uassopd$, the operad encoding unital associative algebras.

In particular, for  $V \in \ob \modr$, there is a natural isomorphism $\tau \lieopd (V) \cong T(V)$  of functors from $\modr$ to $\Rmod$.
\item 
For $\assopd$,   $\assopd(\mathbf{n}) \downarrow^{\sym_n}_{\sym_{n-1}} \cong R[\sym_n]\downarrow^{\sym_n}_{\sym_{n-1}}$. As a right $\sym_{n-1}$-module, this is a direct sum of $n$ copies of $R[\sym_{n-1}]$ (these can be considered as being indexed by the elements of $\zed/n \subset \sym_{n}$, where the cyclic group is generated by the cycle $(1, \ldots , n)$). 

The underlying $\fb\op$-module of $\tau \assopd$ is isomorphic to the tensor product $\uassopd \otimes \uassopd$ of  $\fb\op$-modules. This is most easily interpreted via the isomorphism of Schur functors 
$ 
\tau \assopd (V) 
\cong 
T(V) \otimes T(V)$ 
 for $V \in \ob \modr$.  
This isomorphism follows from Proposition \ref{prop:sigma_tau_BM;M}, which shows that $\tau B (V) \cong B (V; R)$, for any $\fb\op$-module $B$. In the case $B = \assopd$, one checks that $ \assopd (V; R) \cong T(V) \otimes T(V)$.  
\end{enumerate}
\end{exam}

\subsection{The morphism $\delta^B_V$}
\label{subsect:delta}

The natural morphism $\delta^B_V$ introduced below in Definition \ref{defn:delta_V} is a special case of the following:

\begin{lem}
\label{lem:construction_f}
For $B$ a $\fb\op$-module and $V, W \in \ob \Rmod$, there is a  natural morphism of $R$-modules:
\[
\hom_R (V, W) 
\rightarrow 
\hom_R (B(V), B(V;W))
\]
that sends $f : V \rightarrow W$ to the composite
$
B(V) \stackrel{B ((1_V,f) )}{\longrightarrow}
B (V \oplus W) 
\twoheadrightarrow 
B (V;  W)$, 
where the second morphism is the projection to the terms linear in $W$. 
\end{lem}

\begin{proof}
 It is straightforward to reduce to the case   where $B$ is the $\fb\op$-module concentrated in arity $\mathbf{n}$.
 
First consider the case  $B (\mathbf{n}) = R|\sym_n]$. Then the  morphism  $B(V)\rightarrow  B(V;W)$ of the statement is of the form
\begin{eqnarray}
\label{eqn:V_W}
V^{\otimes n} \rightarrow (V^{\otimes n-1} \otimes W) \uparrow_{\sym_{n-1}}^{\sym_n}.
\end{eqnarray}
It is determined by the adjoint morphism $ V^{\otimes n}\downarrow_{\sym_{n-1}}^{\sym_n} \rightarrow V^{\otimes n-1} \otimes W$ with underlying morphism
\[
V^{\otimes n-1} \otimes f : V^{\otimes n} \rightarrow V^{\otimes n-1} \otimes W,
\]
i.e., is given by the functor $V^{\otimes n-1} \otimes- $, which is $R$-linear. From this, the result follows in this case.

For general $B$ concentrated in arity $\mathbf{n}$, the associated morphism is obtained by applying 
$B(\mathbf{n}) \otimes_{\sym_n} -$ to the morphism (\ref{eqn:V_W}). This has the form:
\begin{eqnarray*}
B(\mathbf{n}) \otimes_{\sym_n} V^{\otimes n} \rightarrow B(\mathbf{n}) \otimes_{\sym_n} (V^{\otimes n-1} \otimes W) \uparrow_{\sym_{n-1}}^{\sym_n}
&\cong& 
B(\mathbf{n})\downarrow_{\sym_{n-1}}^{\sym_n}  \otimes_{\sym_{n-1} }( V^{\otimes n-1} \otimes W )
\\
&\cong & 
 B (V;W),
\end{eqnarray*}
where the final isomorphism follows from Proposition \ref{prop:sigma_tau_BM;M}. 
This gives the required result.
\end{proof}

\begin{rem}
\label{rem:inf_comp_mor} 
Lemma \ref{rem:inf_comp_mor}  is a particular case of the infinitesimal composite of morphisms (cf. \cite[Section 6.1.3]{LV}) when working with $\fb\op$-modules; in particular, the linearity statement is related to \cite[Proposition 6.1.3]{LV}. 
\end{rem}

\begin{defn}
\label{defn:delta_V}
For $V \in \ob \Rmod$ and $B$ an $\fb\op$-module,  let 
$
\delta^B_V : B (V) \rightarrow B (V; V)
$ 
be the morphism corresponding to the identity on $V$ under the  construction of Lemma \ref{lem:construction_f}.
\end{defn}

\begin{prop}
\label{prop:delta_morphism}
\ 
\begin{enumerate}
\item
For $B$ an $\fb\op$-module, 
$
\delta^B_V : B (V) \rightarrow B (V; V)
$
defines a natural transformation of functors from  $\Rmod$ to $\Rmod$.
\item 
Via the equivalence of Proposition \ref{prop:sigma_tau_BM;M}, $\delta^B_V$ identifies as the morphism of Schur functors that is induced by the unit $B \rightarrow \sigma \tau B$  for the adjunction $\tau \dashv \sigma$ given by Proposition \ref{prop:sigma_tau_adjunctions}.
\item 
The morphism $\delta^B_V$ is natural with respect to the $ \fb\op$-module $B$.
\end{enumerate}
\end{prop}

\begin{proof}
For the first two statements, one can reduce to the case where $B$ is concentrated in a single arity, say $n$.

In the case $V= W$, the morphism (\ref{eqn:V_W}) given in the proof of Lemma  \ref{lem:construction_f} identifies as 
\[
V^{\otimes n} \rightarrow (V^{\otimes n}) \downarrow_{\sym_{n-1}}^{\sym_n} \uparrow _{\sym_{n-1}}^{\sym_n},
\]
the unit of the restriction-coinduction adjunction. 

Then $\delta^B_V$ identifies as 
\[
B(\mathbf{n}) \otimes_{\sym_n} 
\big(
V^{\otimes n} \rightarrow (V^{\otimes n}) \downarrow_{\sym_{n-1}}^{\sym_n} \uparrow _{\sym_{n-1}}^{\sym_n}
\big),
\]
which is clearly natural in $V$. 

For the second statement, the above morphism is  of the form 
\[
B(\mathbf{n}) \otimes_{\sym_n} V^{\otimes n}
\rightarrow 
B(\mathbf{n}) \otimes_{\sym_n} \big((V^{\otimes n}) \downarrow_{\sym_{n-1}}^{\sym_n} \uparrow _{\sym_{n-1}}^{\sym_n}\big)
\cong 
B(\mathbf{n})\downarrow_{\sym_{n-1}}^{\sym_n} \otimes_{\sym_{n-1}} (V^{\otimes n}) \downarrow_{\sym_{n-1}}^{\sym_n}.
\]
This identifies with the morphism 
\[
\big (
B(\mathbf{n}) \rightarrow B(\mathbf{n})  \downarrow_{\sym_{n-1}}^{\sym_n} \uparrow _{\sym_{n-1}}^{\sym_n}
\big)
\otimes_{\sym_n} V^{\otimes n},
\]
where $B(\mathbf{n}) \rightarrow B(\mathbf{n})  \downarrow_{\sym_{n-1}}^{\sym_n} \uparrow _{\sym_{n-1}}^{\sym_n}$
 is the unit of the restriction-coinduction adjunction. 

The naturality with respect to $B$ is an immediate consequence of the above identification.
\end{proof}

\begin{rem}
Proposition \ref{prop:delta_morphism} shows that $\delta^B_V$ could  be {\em defined} as the natural transformation induced by the morphism of $\fb\op$-modules $B \rightarrow \sigma \tau B$. 
\end{rem}

\section{Derivations of algebras over operads}
\label{sect:opd}

This  Section introduces derivations for algebras over algebraic operads. The definition is recalled in Section \ref{subsect:opd_deriv} and the naturality with respect to $\spmon$  is treated in Section \ref{subsect:nat_deriv}. The grading that is induced by the operadic arity is defined in Section \ref{subsect:grading_deriv}; this is important since it is used to define positive derivations.

\subsection{Algebras over operads and their derivations}
\label{subsect:opd_deriv}

Fix an operad $\opd$ in $R$-modules; this has underlying $\fb\op$-module (in $R$-modules) given by the sequence of right $\sym_n$-modules $\opd (\mathbf{n})$, for $n \in \nat$. The operad structure on this $\fb\op$-module is equivalent to a monad structure on the associated Schur functor $\opd (-)$. In particular, for $V$ in $\Rmod$, there are natural transformations 
$
\eta_V : 
V  \rightarrow  \opd (V),$
 $\mu_V : \opd (\opd (V))  \rightarrow  \opd (V)
$ 
satisfying the  unit and associativity axioms. 

Recall that an operad $\opd$ is  said to be reduced if $\opd (\mathbf{0}) =0$.

\begin{rem}
\label{rem:partial_comp}
An operad structure can also be defined in terms of partial compositions (cf. \cite[Section 5.3.4]{LV}). This is equivalent to the fact that the operad multiplication is determined by the natural transformation
\[
\mu'_V : 
\opd (V; \opd (V) ) \rightarrow \opd(V),
\] 
where $\opd (-;-)$ is the bifunctor given by Definition \ref{defn:BM;N}.

The latter is obtained from the morphism $\mu_V$ by the composite of the natural inclusion $\opd (V; \opd (V)) \subset \opd (V \oplus \opd (V))$ with the morphism induced by $ \eta_V + \mathrm{Id}_{\opd (V)} : V \oplus \opd (V) \rightarrow \opd (V)$, followed by the product $\mu_V : \opd (\opd (V))\rightarrow \opd (V)$.
\end{rem}

An algebra over the operad $\opd$ is an $R$-module $A$ that is an algebra over the monad $\opd (-)$. In particular, it is equipped with a structure morphism $\gamma _A : \opd (A) \rightarrow A$ that satisfies the appropriate axioms.

\begin{nota}
Denote by $\opd\dash\alg$ the category of $\opd$-algebras.
\end{nota}

 A module over the $\opd$-algebra $A$ is an $R$-module $M$ that is equipped with structure morphisms 
$
\gamma_{A;M} :  \opd (A; M)  \rightarrow  M$,  
$\eta_{A;M} : M  \rightarrow  \opd (A;M)$ 
satisfying the associativity and unit axioms (cf. \cite[Section 12.3.1]{LV}).

\begin{exam}
If $A$ is an $\opd$-algebra, then $A$ is an $A$-module, with multiplication $\gamma_{A;A} : \opd (A; A) \rightarrow A$ induced by $\gamma_A$ and the unit $\eta_{A;A} : A \rightarrow \opd (A; A) $ given by the composite $A \stackrel{\eta_A}{\rightarrow} \opd (A) \stackrel{\delta^\opd_A} {\rightarrow} \opd (A; A)$, where $\delta^\opd_A$ is the natural morphism introduced in Section \ref{subsect:delta}.
\end{exam}

\begin{defn}
\label{defn:deriv}
For $A$ an $\opd$-algebra and $M$ an $A$-module, the $R$-module of derivations $\der_A (A, M)$ is the submodule of morphisms $d : A \rightarrow M$ of $R$-modules for which the following diagram commutes:
\[
\xymatrix{
\opd (A) 
\ar[d]_{\gamma_A}
\ar[r]^{\delta^\opd_A} 
&
\opd (A; A) 
\ar[r]^{\opd (A; d) }
&
\opd (A; M) 
\ar[d]^{\gamma_{A;M}}
\\
A
\ar[rr]_d 
&&
M.
}
\]
\end{defn}

When $A$ is a free $\opd$-algebra,  derivations are determined by their restriction to the module of generators (cf. \cite[Section 12.3.8]{LV}). Here we restrict to the case $A = \opd (V)$, for $V \in \ob \modr$:

\begin{prop}
\label{prop:deriv_free}
For $V \in \ob \modr$ and $M$ an $\opd(V)$-module, the restriction $\hom_R (\opd (V) , M)\rightarrow \hom_R (V, M)$, $d \mapsto d|_V$,  induced by  the canonical inclusion $V \hookrightarrow \opd (V)$ induces a natural isomorphism
\[
\der_{\opd(V)} (\opd(V), M) 
\cong 
\hom_R (V, M).
\] 
\end{prop}

\begin{rem}
The derivation associated to an $R$-module morphism $f: V \rightarrow \opd (V)$ is 
given by the composite 
$$
\opd(V) \stackrel{\delta^\opd_V} {\rightarrow }
\opd (V; V) 
\stackrel{\opd (1_V; f) }{\rightarrow }
\opd (V; \opd (V) ) 
\stackrel{\mu'_V}{\rightarrow}
\opd (V),$$
 where $\mu'_V$ is the partial composition operation.
\end{rem}

\begin{nota}
\label{nota:der}
For $V \in \ob \modr$, write $\der (\opd (V))$ for $\der_{\opd(V)} (\opd (V), \opd (V))$.
\end{nota}

\subsection{Naturality for derivations}
\label{subsect:nat_deriv}

For $\opd$ an operad in $R$-modules, one has the bifunctor on $\modr$ defined by 
$
(V, W) \mapsto \hom_R (V, \opd (W)). 
$ 
By Proposition \ref{prop:deriv_free}, the diagonal terms identify as 
$ 
\hom_R (V, \opd (V) ) \cong \der (\opd (V)).
$

\begin{prop}
\label{prop:der_split_monos}
The association $V \mapsto \der (\opd (V))$, for $V \in \ob \modr$ defines a functor 
\[
\der (\opd (-) ) : \spmon  \rightarrow \Rmod.
\]
Explicitly, for $(i,r) \in \hom_{\spmon} (V, W)$ and a derivation $d \in \der (\opd (V))$, the image $d^W \in \der (\opd (W))$ of $d$ is determined by the $R$-module morphism $W \rightarrow \opd (W)$ given by the  composite:
\[
 W  \stackrel{r}{\rightarrow } V\stackrel{d|_V} {\rightarrow} \opd (V) \stackrel{\opd (i)} {\rightarrow} \opd (W).
\]

Moreover, this enriches to a functor to split monomorphisms in $\Rmod$. Explicitly, the natural retract $\der(\opd (W))\rightarrow \der (\opd (V))$ sends $e \in \der (\opd (W))$ to the  element in $\hom_R (V, \opd (V)) \cong \der (\opd (V))$ given by the composite 
$
V  \stackrel{i}{\rightarrow } W\stackrel{e|_W} {\rightarrow} \opd (W) \stackrel{\opd (r)} {\rightarrow} \opd (V).
$ 
\end{prop}

The derivation $d^W$ is not in general equal to the composite 
 $
\opd ( W ) \stackrel{\opd (r)}{\rightarrow } \opd (V)\stackrel{d} {\rightarrow} \opd (V) \stackrel{\opd (i)} {\rightarrow} \opd (W)$.
  One does, however, have the following compatibility result:

\begin{prop}
\label{prop:compat_deriv_spmon}
For $(i, r) \in \hom_{\spmon}(V,W)$ and a derivation $e \in \der (\opd (V))$ with image $e^W \in \der(\opd(W))$ under $(i,r)$, the following diagram commutes:
\[
\xymatrix{
\opd (W) 
\ar[r]^{e^W} 
&
\opd (W) 
\\
\opd (V)
\ar[u]^{\opd (i)} 
\ar[r]_e 
&
\opd (V)
\ar[u]_{\opd (i)}.
}
\]
\end{prop}

\begin{proof}
By Propositions \ref{prop:deriv_free} and \ref{prop:der_split_monos}, the  morphism $e^W : \opd (W)\rightarrow \opd (W)$ is given  by the composite in the commutative diagram:
\[
\xymatrix{
\opd (W) 
\ar[r]^(.4){\delta^\opd_W}
&
\opd (W; W)
\ar[rr]^{\opd (\id_W; (e^W)|_W)}
\ar[d]_{\opd (\id_W; r)}
&&
\opd (W; \opd (W)) 
\ar[r]^(.6){\mu'_W} 
&
\opd (W)\\
&
\opd (W; V) 
\ar[rr]_{\opd (\id_W; e|_V)}
&&
\opd (W; \opd (V))
\ar[u]_{\opd (\id_W; \opd (i))},
}
\]
by the construction of $e^W |_W$ from $e|_V$.

Using the naturality of $\delta$ and  of $\mu'$, together with the fact that $r$ is a retract of $i$, one checks that the morphism given by precomposing with $\opd (i) : \opd (V) \rightarrow \opd (W)$ factorizes as required. 
\end{proof}

\subsection{The natural grading on derivations}
\label{subsect:grading_deriv}
 The grading of $\opd (V)$ induced by the arity of the operad (cf. Definition \ref{defn:schur_functor}) induces a natural  grading of $\der (\opd (V))$:

\begin{prop}
\label{prop:grading_der}
The functor $\der (\opd (-))$ takes values in $\nat\cup \{-1 \}$-graded $R$-modules, with grading inherited from $\opd (-)$; namely, for $V \in \ob \modr$,
\[
\der (\opd (V))
\cong 
\bigoplus _{n\geq 0}  \hom (V, \opd_n (V)),
\]
 where $ \hom (V, \opd_n (V))$ is placed in degree $n-1$. If $\opd$ is reduced, then this yields an $\nat$-grading.

This grading is natural with respect to the operad $\opd$. 
\end{prop}

\begin{proof}
The grading is inherited from the natural grading on the Schur functor given in Definition \ref{defn:schur_functor}.
\end{proof}

By the above, when $\opd$ is reduced, $\der (\opd (V))$ is $\nat$-graded, naturally with respect to $V \in \ob \modr$. This allows the  degree zero part to be separated from the rest of the structure, leading to the {\em positive} derivations $\der^+ (\opd (V))$ introduced below. Focussing upon positive derivations is an important standard tool, notably when taking into account additional structure (see Section \ref{subsect:str_pos_deriv}).

\begin{defn}
\label{defn:der+}
For $\opd$ a reduced operad and $V \in \ob \modr$,  let $\der^+ (\opd (V))$ (respectively $\der^0 (\opd (V))$) be the submodule of $\der (\opd (V))$ of elements of strictly positive degree  (resp.  degree $0$).
\end{defn} 

\begin{rem}
For $\opd$ a reduced operad and $V \in \ob \modr$, there is a natural isomorphism of $R$-modules
\[
\der^+ (\opd (V))
\cong
 \bigoplus_{n\geq 2} 
\hom_R (V, \opd_n (V)). 
\]
\end{rem}

When working with reduced operads, a property that is compatible with the $\nat$-grading on 
 $\der (\opd (V))$ usually carries over to positive derivations. For example, one has the following consequence of Proposition \ref{prop:der_split_monos}.

\begin{cor}
\label{cor:tf_der+}
For $\opd$ a reduced operad, the functor $ \der^+ (\opd (-))$ on $\spmon $ is torsion-free.
\end{cor}

\part{The generalized divergence}
\label{part:gen_trace}

\section{The natural preLie and Lie structures on derivations} 
 \label{sect:prelie}
 
 This Section introduces the  preLie structure on the derivations $\der (\opd (V))$, for an operad $\opd$, $V \in \ob \modr$, and its associated Lie algebra, together with their naturality. 
 
Some of the arguments in this and subsequent Sections reduce to working with free operads; the construction of free reduced operads is presented in Section \ref{sect:free} (the restriction to the reduced case is only to simplify the exposition). 
 
\subsection{The preLie and Lie structures}
\label{subsect:prelie}

A preLie algebra in $R$-modules is an $R$-module $X$ equipped with a morphism of $R$-modules $\lhd : X \otimes X \rightarrow X$, $u \otimes v \mapsto u \lhd v$ such that the associator of $\lhd$ is {\em right symmetric}; i.e., $\forall u, v,w \in X$:
\[
u\lhd (v \lhd w) - (u \lhd v) \lhd w = u \lhd (w \lhd v) - (u \lhd w) \lhd v.
\]

\begin{rem}
\label{rem:right_Lie_module}
\ 
\begin{enumerate}
\item 
A preLie algebra $(X, \lhd)$ is {\em Lie admissible};  i.e.,   the operation $[- ,- ] : X \otimes X \rightarrow X$ defined by $[u, v] := u \lhd v - v \lhd u$ gives a Lie algebra structure on $X$. 
\item 
An $R$-module $X$ equipped with a binary operation $\lhd : X \otimes X \rightarrow X$ defines a preLie algebra if and only if the following relation is satisfied for all $u, v,w \in X$:
\[
u \lhd [v,w] = (u \lhd v)\lhd w - (u \lhd w) \lhd v,
\]
where $[v,w]:= v \lhd w - w \lhd v$. 

In particular, $X$ is a preLie algebra if and only if the operation $[-,-]$ defines a Lie algebra structure and $\lhd$ makes $X$ into a  right $X$-module with respect to this Lie algebra structure. 
\end{enumerate}
\end{rem}

\begin{rem}
PreLie algebras are encoded by the preLie operad $\prelieopd$ \cite[Section 13.4]{LV}. The formation of the associated Lie algebra is given by 
a morphism of operads 
$
\lieopd \rightarrow \prelieopd.
$  
\end{rem}

If $X$ is a (non-unital)  associative $R$-algebra, then the product defines a preLie structure on $X$, since the associator vanishes. This construction is encoded by a morphism of operads 
$
\prelieopd \rightarrow \assopd.
$ 
The composite $\lieopd \rightarrow \prelieopd \rightarrow \assopd$  encodes the commutator Lie structure on an associative algebra.

\begin{defn}
\label{defn:lhd}
For $V  \in \ob \modr$, let 
$$
\lhd : \der (\opd (V)) \otimes \der (\opd (V)) \rightarrow \der (\opd (V))
$$
 be the operation defined for $d, e \in \der(\opd (V))$ with respect to the isomorphism 
$\der (\opd (V)) \cong \hom_R (V, \opd (V)),\  d \mapsto d|_V,$ by taking $(d \lhd e)|_V$ to be 
the derivation determined by 
$
V \stackrel{d|_V} {\rightarrow } \opd (V) 
\stackrel{e} {\rightarrow} \opd (V) . 
$ 
\end{defn}

\begin{rem}
\label{rem:opposite}
This definition is dictated by the usual conventions for operadic composition. It corresponds to the {\em opposite} structure when considering the composition of morphisms. 
For instance, if $\opd (\mathbf{0}) =0$ and $\opd (\mathbf{1})=R$, generated by the unit, then in degree zero (using the grading of Proposition \ref{prop:grading_der}), $\der (\opd (V))$ identifies with $\End_R (V)$. The law $\lhd$ corresponds to the opposite of the usual composition multiplication on $\End_R (V)$.
\end{rem}

Recall from Proposition \ref{prop:grading_der} that $\der (\opd (V))$ is graded. 

\begin{thm}
\label{thm:prelie}
For $V  \in \ob \modr$, $(\der (\opd (V)), \lhd)$ is a preLie algebra and this structure is natural with respect to $\spmon$, so that $\der (\opd (-))$ defines a functor 
$$
\der(\opd (-)) : 
\spmon \rightarrow \prelieopd\dash\alg.
$$

If $\opd$ is reduced, this takes values in $\nat$-graded preLie algebras.
\end{thm}

\begin{proof}
The argument to establish the preLie structure is standard. One can proceed as follows when $\opd$ is reduced: using the naturality with respect to the operad $\opd$, one reduces (as in Section \ref{subsect:naturality_lhd} below) to the case where $\opd = \fo$ is a free operad on a graded set of generators $\gen$ (cf. Section \ref{subsect:fo}). In this case, the result follows from properties of the operation of grafting of trees (see  Proposition \ref{prop:rpt_prelie}). The argument generalizes to the non-reduced case.

The naturality with respect to $\spmon$ is as given by Proposition \ref{prop:der_split_monos}. To show that it is compatible with the preLie structure,  consider derivations $d, e \in \der (\opd (V))$  and a morphism $(i,r) \in \hom_{\spmon} (V,W)$ as in the statement. Let $d^W, e^W \in \der (\opd (W))$ denote the images of $d,e$ respectively under $(i,r)$. Then one has a commutative diagram:

\[
\xymatrix{
W 
\ar[d]_r
\ar[r]^{(d^W)|_W}
&
\opd (W) 
\ar[r]^{e^W}
&
\opd (W)
\\
V
\ar[r]_{d|_V}
&
\opd (V) 
\ar[u]|{\opd (i)}
\ar[r]_e
&
\opd (V)
\ar[u]_{\opd (i)},
}
\]
where the commutative square on the left is given by the construction of $d^W$, by Proposition \ref{prop:der_split_monos}, and, on the right, by Proposition \ref{prop:compat_deriv_spmon}. 

By definition of the preLie structure, the composite of the top row is $(d^W \lhd e^W) |_W$ and that of the bottom row is $(d \lhd e)|_V$. Moreover, passing from $W$ to $\opd (W)$ via the bottom row gives the restriction of $(d\lhd e)^W $ to $W$, again by definition of the preLie structure. Hence, the commutativity of the diagram shows that  
 $(d\lhd e)^W = d^W \lhd e^W$, as required.
 
 The  grading statement is a standard consequence of the structure of an operad.
\end{proof}

\begin{rem}
\ 
\begin{enumerate}
\item
By definition, $\der (\opd (V))$ is a sub $R$-module of $\End_R (\opd (V))\op$, which is an associative $R$-algebra, hence a preLie algebra. It is {\em not} a  sub preLie algebra  in general, for the usual reason: the composite of two derivations is not in general a derivation. 

If $\opd$ is reduced and $\opd (\mathbf{1})= R$, generated by the unit, 
upon restriction to degree zero, one {\em does} recover $\End_R (V)\op$ (see Remark \ref{rem:opposite}).
\item 
The retract $\der (\opd (W)) \rightarrow \der(\opd (V))$ associated to $(i,r)$ by Proposition \ref{prop:der_split_monos} is in general only a morphism of $R$-modules, not of preLie algebras.
\item 
Theorem \ref{thm:prelie} can be viewed as a generalization of  \cite[Theorem 1.7.3]{Kap_Manin}, in which Kapranov and Manin show that  $\opd (R)$ has a natural preLie algebra structure.
\end{enumerate}
\end{rem}

Composing with the associated Lie algebra functor $\prelieopd\dash\alg \rightarrow \lieopd\dash\alg$, Theorem \ref{thm:prelie} gives:

\begin{cor}
\label{cor:der_lie}
Derivations yield a functor $
\der(\opd (-)) : 
\spmon \rightarrow \lieopd\dash\alg.
$
 If $\opd$ is reduced, this takes values in $\nat$-graded Lie algebras.
\end{cor}

\begin{rem}
\ 
\begin{enumerate}
\item 
 The preLie structure is much easier to work with than the associated Lie structure, as
 exemplified by Theorem \ref{thm:derpl} and Remark \ref{rem:pl_versus_lie}.
\item 
The functor $\der(\opd (-))$ of Corollary \ref{cor:der_lie} does not in general arise from a functor on $\modr$: the split nature of the morphisms of $\spmon$ is essential so as to define the {\em natural} Lie structure.
\end{enumerate}
\end{rem}

Consider the free (reduced) operad $\fo$ on the graded set of generators $\gen$ (see Section \ref{subsect:fo}); this is simple to work with since it is induced from a non-symmetric operad.  As in Section \ref{sect:free}, the set of $S$-labelled rooted planar $\gen$-trees is denoted $\rpt(S)$, for $S$ a finite set.  Proposition \ref{prop:rpt_prelie} shows that there is a natural preLie structure on the $R$-linearization $R[\rpt (S)]$, where naturality is with respect to the category $\finj$.

\begin{prop}
\label{prop:deriv_fo}
For $\fo$ the free operad on the graded set of generators $\gen$, the restriction of the functor 
$$\der (\fo (-)) : \spmon \rightarrow \prelieopd\dash \alg
$$
 along $
R[-] : \finj \rightarrow \spmon$ is naturally isomorphic to the functor 
$S \mapsto R[\rpt (S)]$ of Proposition \ref{prop:rpt_prelie}.
\end{prop}

\begin{proof}
For $V \in \ob \spmon$, using Lemma \ref{lem:strong_duality}, there are natural isomorphisms 
$$
\der (\fo (V)) \cong \hom _R ( V, \fo (V))
\cong \fo (V) \otimes V^\sharp.
$$
Restriction along $R[-] : \finj \rightarrow \spmon$ gives the functor 
$ 
S \mapsto \der (\fo (R[S]))$ 
on $\finj$ and, by the above, 
$ 
\der (\fo (R[S])) 
\cong \fo (R[S]) \otimes (R[S])^\sharp.
$ 
The right hand side is naturally isomorphic to $ \fo (R[S]) \otimes (R[S])$ as a functor on $\finj$, by Proposition \ref{prop:duality_FI}.

Using the construction of $\fo$ from a non-symmetric operad and by definition of the Schur functor, one sees that $ \fo (R[S]) $ is naturally isomorphic to the free $R$-module on the set of rooted planar $\gen$-trees equipped with a map from the leaves to $S$. Interpreting the additional tensor factor $R[S]$ as the root label, one obtains the  isomorphism  
\[
\der (\fo (R[S])) 
\cong 
R[\rpt (S)]
\]
that is natural with respect to $S \in \ob \finj$.

Since the operad structure of $\fo$ is induced by grafting of $\gen$-trees, one has that, under this isomorphism, the preLie structure on $\der (\fo (R[S])) $ given by Theorem \ref{thm:prelie} identifies with that on $R [\rpt (S)]$ given by Proposition \ref{prop:rpt_prelie}.
\end{proof}

\subsection{Naturality with respect to the operad}
\label{subsect:naturality_lhd}

Consider a morphism of operads $\opd \rightarrow \ppd$.  For $V \in \ob \modr$, this induces a morphism of $R$-modules 
 $
\opd (V) \rightarrow \ppd (V)$.
  More is true: this is a morphism of $\opd(V)$-algebras, in particular induces
\[
\der_{\opd(V)} (\opd (V), \opd (V) ) 
\rightarrow 
\der_{\opd (V)} (\opd (V) , \ppd (V)). 
\]
This can be interpreted as a morphism of $R$-modules 
$\der (\opd (V))
\rightarrow
\der (\ppd (V))$.

\begin{prop}
\label{prop:nat_der}
For $\opd \rightarrow \ppd$ a morphism of operads,
\begin{enumerate}
\item 
 the morphism 
$
\der (\opd (-))
\rightarrow
\der (\ppd (-))
$ 
is a natural transformation of functors from $\spmon$ to $\prelieopd\dash\alg$;
\item
composing with the restriction $\prelieopd\dash\alg \rightarrow \lieopd \dash \alg$, this gives a natural transformation of functors from $\spmon$ to $\lieopd\dash\alg$.
\end{enumerate}
These are compatible with the $\nat \cup \{-1 \}$-gradings.
\end{prop}

\begin{proof}
Naturality with respect to $\spmon$ as a functor to $\Rmod$ is clear. The key point is therefore the naturality of the preLie structure, i.e., that the natural morphism $\der (\opd (V))
\rightarrow
\der (\ppd (V))$ is a morphism of preLie algebras. This follows from the fact that the isomorphism $\hom_R (V, \opd (V)) \cong \der (\opd (V))$ is natural with respect to the operad $\opd$, by construction, together with the explicit form of the construction of $\lhd$. 
\end{proof}

\begin{exam}
\label{exam:der_lie_ass}
Consider the morphism of operads $\lieopd \rightarrow \assopd$ encoding the associated Lie algebra of a (non-unital) associative algebra. Here $\lieopd (V)$ is the free Lie algebra on $V$ and $\assopd (V)$ is the augmentation ideal of the tensor algebra on $V$. 

The morphism $\lieopd (V) \hookrightarrow \assopd (V)$ corresponds to the inclusion of the primitive elements of the tensor algebra on $V$. Then $\der (\lieopd (V)) \rightarrow \der (\assopd (V))$ is the inclusion of the submodule of derivations of $\assopd(V)$ such that $V$ is mapped to primitives. 
\end{exam}

The following is important in reducing arguments to the case of free operads: 

\begin{prop}
\label{prop:surject_der}
For  $\opd \twoheadrightarrow \ppd$ a surjective morphism of operads,  the induced natural transformation 
$
\der (\opd (-)) 
\twoheadrightarrow 
\der (\ppd (-))
$ 
is surjective.
\end{prop}

\begin{proof}
The surjectivity of $\opd \twoheadrightarrow \ppd$ implies that, for any $V \in \ob \spmon$, the morphism of $R$-modules
$ 
\opd (V) \rightarrow \ppd (V)
$ 
is surjective. Since $V$ is projective, this implies that $\hom_R (V, \opd(V)) \rightarrow \hom_R (V, \ppd (V))$ is surjective, whence the result.
\end{proof}

\subsection{The rôle of positive derivations}
\label{subsect:str_pos_deriv}

When $\opd$ is reduced, focussing upon positive derivations (see Definition \ref{defn:der+}) serves to ignore the contribution to $\der (\opd (V))$ from the degree zero part.  It is also an important tool when integrating to a group (see \cite[Section 2.4]{MR3758425} in the case $\opd= \assopd$, for example).

The following  Proposition shows that the analysis  of positive derivations is a natural first step to understanding all derivations as a preLie algebra.

\begin{prop}
\label{prop:semi-direct_product}
For $\opd$  a reduced operad,  $\der^+ (\opd(-))$ and $\der^0 (\opd (-))$ are subfunctors of $\der(\opd (-)) : \spmon \rightarrow \prelieopd\dash\alg$.

Moreover, for $V \in \ob \spmon$, there is a natural split sequence of preLie algebras:
\[
\xymatrix{
\der^+(\opd (V))
\ar@{^(->}[r]
&
\der (\opd (V))
\ar@{->>}[r]
&
\der^0(\opd (V))
\ar@/_1pc/@<-.5ex>[l]
}
\]
and hence  a natural isomorphism 
$
\der (\opd (V)) \cong \der^+(\opd (V)) \rtimes \der^0(\opd (V)) 
$ of the associated Lie algebras.
\end{prop}

\begin{proof}
The result follows from Proposition \ref{prop:grading_der}. Since $\opd$ is reduced,  $\der (\opd (V))$ is $\nat$-graded. The projection $\der (\opd (V)) \twoheadrightarrow \der^0 (\opd (V))$ onto elements of degree zero is clearly a morphism of preLie algebras and gives rise to the split sequence of preLie algebras. On passing to the associated  Lie algebras, this corresponds to  the semi-direct product of Lie algebras.
\end{proof}

The following complements Proposition  \ref{prop:deriv_fo} (as in Section \ref{sect:free}, $v(\tr)$ is the set of internal vertices of a tree $\tr$):

\begin{prop}
\label{prop:deriv_fo_pos}
Let $\fo$ be the free (reduced) operad generated by the graded set $\gen$. Then, with respect to $S \in \ob \finj$, there are natural isomorphisms of preLie algebras:
\begin{eqnarray*}
\der^+ (\fo (R[S]))
& \cong & 
R [\rptpos (S)] \\
\der^0 (\fo (R[S])) & \cong &
\End_R (R[S])\op,
\end{eqnarray*}
where $\rptpos (S) \subset \rpt (S)$ is the subset $\{ \tr\in \rpt(S)\ | \ \ |v(\tr )|\geq 1 \}$ of trees containing at least one internal vertex and $R [\rptpos (S)] \subset R[\rpt (S)]$ is equipped with the sub preLie structure of that given by Proposition \ref{prop:rpt_prelie}.
\end{prop}

\begin{proof}
The first statement follows by inspection from the definition of the grading, with the identification of the preLie structure on $\der^+ (\fo (R[S]))$ following from Proposition \ref{prop:deriv_fo}.

Similarly,  $\der^0 (\fo (R[S]))$ has basis given by the trees $\tr \in \rpt(S)$ such that $|v(\tr )|=0$ (i.e., with no internal vertex). 
Such an $S$-labelled tree is identified by the ordered pair of the root label and the leaf label. The isomorphism of preLie structures in degree zero reflects the natural isomorphisms  $\End_R (R[S]) \cong (R|S])^\sharp \otimes R[S] \cong R[S \times S]$.
\end{proof}

\section{The associative algebra structure on pointed derivations}
\label{sect:pointed}

This Section  introduces the algebra of pointed derivations that is used to define the codomain of the generalized divergence  in Section \ref{sect:contract_trace}. 

For $\opd$ an operad, $\opd (\mathbf{1})$ has a natural unital associative algebra structure induced by the operad structure.  This is generalized here by considering a suitable sub preLie-algebra of $\der (\opd (V))$; this involves restricting to the pointed version  $\sppt$ of $\spmon$,  introduced in Section \ref{subsect:pointed}.

\subsection{Pointed derivations and the associative algebra structure}

Theorem \ref{thm:prelie} shows that $V \mapsto \der (\opd (V))$ defines a functor from $\spmon$ to preLie-algebras. Via the forgetful functor $\sppt \rightarrow \spmon$, this can be considered as a functor on $\sppt$. 

Proposition \ref{prop:relate_sppt} shows that $\sppt$ is equivalent to $\spmon$; in particular an object of $\sppt$ has underlying $R$-module that decomposes canonically as $R \oplus V $, where $V$ is considered as an object of  $\spmon$. This corresponds to $(R\oplus V, R)$, using the notation employed in Section \ref{subsect:pointed}; this is often simplified by writing $R \oplus V$, leaving the pointed structure implicit.

Using the notation introduced in Definition \ref{defn:BM;N},   one has the sub $R$-module $\opd (V; R ) \subset \opd (R \oplus V)$ of terms that are linear with respect to $R \subset R \oplus V$.

\begin{defn}
\label{defn:derpt}
For $V \in \ob \spmon$, let  $\derpt (\opd ( R\oplus V)) \subset \der (\opd (R \oplus V))$ be 
\[
\derpt (\opd (R \oplus V)) := \hom _R (R   , \opd (V; R )) \cong \opd (V;R),
\] 
considered as a sub-module of $\der (\opd (R \oplus V)) = \hom_R (R \oplus V, \opd (R \oplus V ))$ via the canonical projection $R \oplus V \twoheadrightarrow R $ and the canonical inclusion $\opd (V; R ) \subset \opd (R \oplus V)$.
\end{defn}

\begin{lem}
\label{lem:derpt_subfunctor}
\ 
\begin{enumerate}
\item 
The association  
$ V \mapsto  \derpt (\opd (R \oplus V) ) $ defines a functor from $\modr$ to 
$\nat \cup \{-1\}$-graded $R$-modules.
\item
The association $R \oplus V \mapsto\derpt (\opd (R \oplus V))$ defines a subfunctor of $R \oplus V \mapsto \der (\opd (R\oplus V))$ considered as a functor on $\sppt$ with values in $\nat \cup \{-1\}$-graded $R$-modules.
\item 
These structures are natural with respect to the operad $\opd$, for the naturality of $\der (\opd (-))$ given by Proposition \ref{prop:nat_der}. 
\end{enumerate}
\end{lem}

\begin{proof}
The first statement follows from the naturality of $V \mapsto \opd (V; R)$ together with the grading induced from operadic arity. 

The second statement follows similarly, using the fact that $\sppt$ is equivalent to the category $\spmon$ via the functor $R \oplus -$ (see Proposition \ref{prop:relate_sppt}).

The third follows from the fact that $\opd (V; R) \subset \opd (R \oplus V)$ is natural with respect to the operad $\opd$.
\end{proof}

\begin{rem}
Although $\derpt (\opd (R \oplus -) ) $ is a functor on $\modr$,  when considering additional structure it is frequently necessary to restrict to $\spmon$ via the forgetful functor $\spmon \rightarrow \modr$, since $\der (\opd (-))$ is only a functor on $\spmon$, not on $\modr$. For example, this is the case when considering the  $\der (\opd (-))$-action introduced in Section \ref{subsect:der_act_derpt},

By the equivalence of categories $R \oplus - : \spmon \stackrel{\cong}{\rightarrow} \sppt $ given by Proposition \ref{prop:relate_sppt}, considering the functor $\derpt (\opd (R \oplus -) ) $ restricted to $\spmon$ is equivalent to considering $\derpt (\opd(-))$ as a functor on $\sppt$.
\end{rem}

The following uses the notion of a pointed $S$-labelled rooted planar $\gen$-tree (for $\gen$ a graded set of generators) given in Definition \ref{defn:arboriculture}. Recall that Proposition \ref{prop:deriv_fo} provides the natural isomorphism $\der (\fo (R[S])) \cong R[\rpt (S)]$. For $(S,z)$ a finite pointed set, $R[S,z]$ denotes the associated object of $\sppt$, as in  Notation \ref{nota:fipt_sppt}.

\begin{prop}
\label{prop:pointed_der_fo}
Let $\fo$ be the free operad on the graded set of generators $\gen$.  For $(S,z) \in \ob \fipt$, 
\[
\derpt (\fo (R[S,z])) \subset \der (\fo (R[S])) \cong R[\rpt (S)]
\]
has sub-basis given  by the set of pointed $S$-labelled rooted planar $\gen$-trees with root labelled by $z$.

The preLie structure on $\der (\fo (R[S]))$ restricts to an associative, unital structure on the pointed derivations $\derpt (\fo (R[S,z]))$. 

With respect to the above identification, the product is induced by grafting of pointed $S$-labelled trees and the unit  represented by the pointed tree  with no internal vertex and leaf and root labelled by $z$.
\end{prop}

\begin{proof}
The first statement follows from the definition of $\derpt (\fo (R[S,z]))$. This uses the identification of a basis of $\fo (R[S \backslash \{ z \} ]) ; R z)$, which is facilitated by the fact that $\fo$ arises from a non-symmetric operad. 

Consider restricting the preLie structure of $R[\rpt (S)]$ to $\derpt (\fo (R[S,z]))$ under this identification. Given two basis elements represented by $S$-labelled $\gen$-trees $\tr_1$, $\tr_2$,  there is a unique possible grafting of $\tr_2$  onto $\tr_1$, namely  grafting the root of $\tr_2$ to the leaf of $\tr_1$ that is labelled by $z$. 

Schematically, forming either of the triple products $(\tr_1 \lhd \tr_2) \lhd \tr_3$ or $\tr_1\lhd (\tr_2 \lhd \tr_3)$ corresponds to the unique possible two-fold grafting:

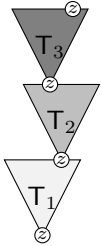
\qquad \qquad
\begin{tikzpicture}[scale=1]
\draw [fill=gray!10] (2,0) -- (2.5,1) -- (1.5,1) -- (2,0);
\draw [fill=lightgray] (2.25,1) -- (2.75,2) -- (1.75,2) -- (2.25,1);
\draw [fill=gray] (2.1,2) -- (2.6, 3)-- (1.6,3) -- (2.1, 2); 
\draw [fill=white] (2,0) circle [radius = 0.1];
\draw [fill=white] (2.1,2) circle [radius = 0.1];
\draw [fill=white] (2.25,1) circle [radius = 0.1];
\draw [fill=white] (2.4,3) circle [radius = 0.1];
\node at (2, .5) {$\tr_1$};
\node at (2.25,1.5) {$\tr_2$};
\node at (2.1,2.5) {$\tr_3$};
\node at (2,0) {$\scriptstyle{z}$};
\node at (2.1,2) {$\scriptstyle{z}$};
\node at (2.25,1) {$\scriptstyle{z}$};
\node at (2.4, 3) {$\scriptstyle{z}$};
\end{tikzpicture}

It follows that  $\derpt (\fo (R[S,z]))$ is a  preLie subalgebra of $\der (\fo (R[S]))$ and this subalgebra is associative. The statement concerning the unit is clear. 
\end{proof}

The behaviour exhibited in Proposition  \ref{prop:pointed_der_fo} extends to the case of an arbitrary operad; recall from Example \ref{exam:opd_lie_uass_ass} that $\uassopd$ denotes the unital associative operad:

\begin{thm}
\label{thm:assoc_alg}
\ 
\begin{enumerate}
\item 
The functor $\derpt (\opd (-))$ is a subfunctor of the composite 
$$
\sppt 
\stackrel{\mathrm{forget}}{\longrightarrow}
\spmon 
\stackrel{\der (\opd (-))}
{\longrightarrow} 
\prelieopd-\alg.
$$  
\item 
For $R \oplus V \in \ob \sppt$, the preLie structure on $\derpt (\opd(R \oplus V))$  is  associative  and the unit of the operad provides a natural unit, so that  $\derpt (\opd (-))$ factorizes 
\[
\derpt (\opd (-)) : \sppt 
\rightarrow 
\uassopd \dash\alg
\stackrel{\mathrm{forget}}{\longrightarrow} 
\prelieopd \dash  \alg.
\]
\item 
The functor $\derpt (\opd (-)): \sppt 
\rightarrow 
\uassopd \dash\alg $ is natural with respect to the operad $\opd$.
\end{enumerate}
\end{thm}

\begin{proof}
This result follows from Theorem \ref{thm:prelie}. The fact that $\derpt (\opd (R \oplus V))$ is a sub  preLie algebra of $\der (\opd (R \oplus V))$ is a direct verification from the construction of the preLie structure. One checks that the argument given in the proof of Proposition \ref{prop:pointed_der_fo} generalizes to show that the associator vanishes, so the preLie structure is in fact an associative algebra structure and that the operad unit induces a unit for this algebra. 

Naturality with respect to the operad follows from Proposition \ref{prop:nat_der}, together with the fact that the unit is natural with respect to $\opd$, which is clear from its definition.
\end{proof}

The following  stresses the grading in the case of a reduced operad (cf. Theorem \ref{thm:prelie}):

\begin{cor}
\label{cor:grading_derpt_reduced}
For  $\opd$  a reduced operad,  $\derpt (\opd (-))$ takes values naturally in $\nat$-graded, unital associative algebras. 
\end{cor}

By restriction along $R[-] : \fipt \rightarrow \sppt$ (see Proposition \ref{prop:R_free_pt}),  
$\derpt(\opd (-))$ gives a functor:
\[
\derpt (\opd (-)) : \fipt 
\rightarrow \uassopd \dash\alg.
\]

\subsection{The examples $\comopd$, $\lieopd$ and $\assopd$}

This Section extends Example \ref{exam:tau_opd}, by determining the algebra structure on $\derpt (\opd (R \oplus V))$ for  $V \in \ob \modr$ and $\opd \in \{\comopd, \lieopd, \assopd\}$. The isomorphism 
$\derpt (\opd (R\oplus V))\cong  \opd (V; R)$  provides the embedding $\derpt (\opd (R \oplus V) ) \subset \opd (R \oplus V)$. 

\begin{prop}
\label{prop:derpt_lie_assoc} 
For $V \in \ob \spmon$, there are natural isomorphisms of associative algebras:
\begin{enumerate}
\item 
$\derpt ( \lieopd (R \oplus V)) \cong T(V)$, the tensor algebra;
\item 
$\derpt( \assopd (R \oplus V)) \cong T(V) \otimes T(V)\op$, the enveloping algebra of $T(V)$; 
\item
$\derpt (\comopd (R \oplus V)) \cong S(V)$, the symmetric algebra.
\end{enumerate}
\end{prop}

\begin{proof}
For the $\lieopd$ case,  the isomorphism given in Example \ref{exam:tau_opd} can be interpreted via the composite $T(V) \stackrel{\cong}{\rightarrow} \derpt (\lieopd (R \oplus V)) \subset \lieopd (R \oplus V)$ as: 
\begin{eqnarray}
\label{eqn:iterated_lie}
v_1 \otimes \ldots \otimes v_n \mapsto [v_1 ,[v_2, [\ldots ,[v_n , 1]\ldots ],
\end{eqnarray}
writing $1$ for the generator of $R$. 

The product in $\derpt (\lieopd (R \oplus V))$ multiplying by the element $[w_1, [w_2, [\ldots, [w_k, 1] \ldots ]$ on the right is given by replacing $1$ in equation (\ref{eqn:iterated_lie}) by the iterated commutator formed from the $w_i$'s, since the product in $\derpt (\lieopd (R \oplus V))$ is induced by the Lie operad structure. The resulting element is the image of $v_1 \otimes \ldots \otimes v_n \otimes w_1\otimes \ldots \otimes w_k$.

For $\assopd$,  proceeding as in the Lie case, the isomorphism of Example \ref{exam:tau_opd} is interpreted via  $T(V) \otimes T(V) \stackrel{\cong}{\rightarrow} \derpt (\assopd (R \oplus V)) \subset \assopd (R \oplus V)$ as 
$
T(V) \otimes T(V) \hookrightarrow \overline{T} (R \oplus V) $
 that sends $\alpha \otimes \beta \mapsto \alpha \otimes 1 \otimes \beta$, writing $1$ for the generator of $R$ as above, and considering $\alpha \otimes 1 \otimes \beta$ as an element of $\overline{T} (R \oplus V) $. The product with $\alpha ' \otimes \beta '$ (on the right) corresponds to replacing the element $1$ by $\alpha' \otimes 1 \otimes \beta'$ via this embedding. 
Thus $(\alpha \otimes \beta)(\alpha' \otimes \beta') = \alpha \alpha' \otimes \beta' \beta$, as required. 

For $\comopd$, the analysis is similar to that of the associative case, but more straightforward. The details are left to the reader.

In each case, the isomorphisms are natural with respect to $V$. 
\end{proof}

The  naturality with respect to the operad $\opd$ is illustrated by the following example, which gives a conceptual explanation for the behaviour exhibited in \cite[Section 3.2]{MR3758425}.

\begin{exam}
\label{exam:nat_assoc_alg_opd}
Consider the morphism of operads $\lieopd \rightarrow \assopd$ that encodes the commutator Lie algebra of an associative algebra. This induces a morphism of associative algebras
\begin{eqnarray}
\label{eqn:derpt_lie_ass}
\derpt (\lieopd (R \oplus V) ) \rightarrow \derpt (\assopd (R \oplus V)).
\end{eqnarray}
Since $\derpt (\lieopd (R \oplus V) )$ is the free associative algebra on $V$, it suffices to consider the image of the generators $V \subset T(V)$. As in the proof of Proposition \ref{prop:derpt_lie_assoc}, this corresponds to the submodule of $\lieopd( R \oplus V)$ generated by commutators of the form $[v, 1]$, for $v \in V$. 

The image of $[v,1]$ in $\assopd (R \oplus V) \cong T(R \oplus V)$ is $v \otimes 1 - 1 \otimes v$, again using the above notation. It follows that the morphism of algebras (\ref{eqn:derpt_lie_ass}) identifies under the isomorphisms of Proposition \ref{prop:derpt_lie_assoc} as
\[
T(V) \rightarrow T(V) \otimes T(V)\op
\]
induced by $v \mapsto v \otimes 1 - 1 \otimes v$. This algebra morphism is 
$
\tilde{\Delta} := (1 \otimes \iota) \Delta
$
(in the notation of \cite[Section 3.2]{MR3758425}),  where $\Delta$ denotes the shuffle coproduct on the tensor algebra and $\iota$ denotes the conjugation for the associated Hopf structure. 
\end{exam}

\begin{exam}
\label{exam:nat_assoc_com_alg_opd}
Consider the morphism of operads $\assopd \rightarrow \comopd$ encoding the fact that a commutative algebras is associative. The induced morphism of algebras:
\begin{eqnarray*}
\derpt (\assopd (R \oplus V) )& \rightarrow& \derpt (\comopd (R \oplus V))
\\
T(V) \otimes T(V)\op
& \rightarrow 
&
S(V)
\end{eqnarray*}
is determined by $v \otimes 1 \mapsto v$, $1 \otimes v \mapsto v$, for $v \in V$. 
\end{exam}

\begin{rem}
An alternative approach to the above is to use the operadic enveloping algebra, exploiting Theorem \ref{thm:derpt_via_operad_env_alg} of Appendix \ref{sect:env_alg}. 
\end{rem}

 \subsection{The $\der (\opd (V))$-action}
\label{subsect:der_act_derpt} 
 
By Theorem \ref{thm:prelie}, for $V \in \ob \spmon$, $\der (\opd (V))$ has a natural preLie structure. Hence, by Remark \ref{rem:right_Lie_module}, it can be considered as a right module over the associated Lie algebra. 
 
By precomposition with the functor $R \oplus - : \spmon \rightarrow \sppt$, $V \mapsto \derpt (\opd (R \oplus V))$ is a functor on $\spmon$ with values in associative algebras, by Theorem \ref{thm:assoc_alg}. Forgetting that $R \oplus V$ is pointed, one has the morphism $V \hookrightarrow R\oplus V$ in $\spmon$ 
 which, by Theorem \ref{thm:prelie}, induces an inclusion of preLie algebras 
 $ 
 \der (\opd (V) ) \hookrightarrow \der (\opd (R\oplus V)).
 $

 \begin{prop}
 \label{prop:derpt_right_module}
Let $V$ be an object of $ \spmon$.
\begin{enumerate}
\item 
The preLie structure on $\der (\opd (R \oplus V))$ restricts to a right action 
 \[
 \derpt (\opd (R\oplus V) ) \otimes \der (\opd (V) ) \rightarrow  \derpt (\opd (R\oplus V ) ) 
 \]
 of the Lie algebra $\der (\opd (V) ) $ that is natural with respect to $V \in \ob \spmon$. 
\item 
If $\opd$ is reduced, this action is compatible with the $\nat$-gradings derived from the $\nat$-grading on $\der (\opd (-))$ given by Theorem \ref{thm:prelie}.   
\item
The action is natural with respect to the operad in the following sense. For $\opd \rightarrow \ppd$ a morphism of operads, the induced morphism given by Lemma \ref{lem:derpt_subfunctor} 
$$
 \derpt (\opd (R\oplus V) )  \rightarrow \derpt (\ppd (R \oplus V))
$$
 is a morphism of right $\der (\opd(V))$-modules, where $\derpt (\ppd (R \oplus V))$ is considered as a $\der (\opd (V))$-module by restriction along the morphism of Lie algebras $\der (\opd( V)) \rightarrow \der (\ppd (V))$ given by Proposition \ref{prop:nat_der}.
 \end{enumerate}
 \end{prop}

 \begin{proof}
 By construction, $\derpt (\opd (R\oplus V ) )$ is an $R$-module direct summand  of $\der (\opd (R\oplus V))$. Restricting the preLie structure on $\der (\opd (R \oplus V))$, this gives a right action of  $\der (\opd ( V) ) $  as stated, since the `basepoint' $R$ is left untouched, because of the restriction to $V \subset R \oplus V$. 
 
The grading of $\derpt (\opd (R \oplus V))$ is inherited from that of $\der (\opd (-))$, hence the second statement follows from the grading property given by Theorem \ref{thm:prelie}. 
 
The naturality with respect to the operad $\opd$ follows from the naturality of the preLie structure given by Proposition \ref{prop:nat_der}. 
 \end{proof}

Proposition \ref{prop:derpt_right_module} did not take into account the natural unital associative algebra structure on $ \derpt (\opd (R \oplus V) )$ given by Theorem \ref{thm:assoc_alg}. The following establishes that this  is compatible with the right action:
 
 \begin{prop}
 \label{prop:prod_morph_lie_mod}
 For $V \in \ob \spmon$, the natural associative product:
 \[
 \derpt (\opd (R \oplus V) ) \otimes  \derpt (\opd (R \oplus V) )
 \rightarrow 
  \derpt (\opd (R \oplus V) )
 \]
 is a morphism of right $\der (\opd (V))$-modules, where the domain is given by the tensor product module structure over the Lie algebra $\der (\opd ( V))$.
 \end{prop}
 
 \begin{proof}
This follows from the associativity properties of the partial composition operations for operads (cf. \cite[Section 5.3.4]{LV}). 
 \end{proof}
 
This implies that the right $\der (\opd ( V ))$-action passes to the quotient by the $R$-module of commutators:

 \begin{cor}
 \label{cor:derpt_module}
 \ 
 \begin{enumerate}
 \item
 The functor $\spmon \rightarrow \Rmod$ given by 
  $
V \mapsto  | \derpt (\opd (R \oplus  V) ) |
 $ 
 takes values naturally in right $\der (\opd ( V ))$-modules: the  action given in  Proposition \ref{prop:derpt_right_module} induces a right action
 \[
 | \derpt (\opd (R \oplus V) ) |
 \otimes 
 \der (\opd ( V) ) 
 \rightarrow
 | \derpt (\opd (R \oplus  V) ) |
 \]
 that is a natural transformation on $\spmon$.
 \item 
 This action is natural with respect to the operad $\opd$. Explicitly, a morphism of operads $\opd \rightarrow \ppd$ induces a natural morphism:
\[
| \derpt (\opd (R \oplus V) ) |
  \rightarrow 
| \derpt (\ppd (R \oplus V) ) |
\] 
that is a morphism of right $\der(\opd (V))$-modules, where $ | \derpt (\ppd (R \oplus V) ) |$ is considered as a $\der (\opd (V))$-module by restriction along the morphism of Lie algebras $\der (\opd( V)) \rightarrow \der (\ppd (V))$.
\end{enumerate}
These structures are compatible with the $\nat \cup \{-1 \}$-gradings.
 \end{cor}

\section{The generalized contraction and the generalized divergence}
\label{sect:contract_trace}

The purpose of this Section is to introduce the generalized divergence for an arbitrary  operad $\opd$. For the Lie operad, this corresponds to Satoh's trace map and, for the associative operad, to the double divergence. 

Two important structural results are established: Proposition \ref{prop:addelta_1-surjective} shows that the generalized divergence is almost surjective (in a precise sense, defined using torsion as introduced in Section \ref{sect:torsion}) and Theorem \ref{thm:1_cocycle} shows that it is a $1$-cocycle for the Lie algebra structure on derivations.

\subsection{The generalized contraction}
The generalized contraction map associated to an operad $\opd$ is introduced in Corollary \ref{cor:naturality_addelta_opd}. We start by considering the case of a $\fb\op$-module $B$ and its associated Schur functor so that, for $V \in \ob \modr$, $\delta^B_V$ gives a natural transformation 
\begin{eqnarray}
\label{eqn:delta}
B(V) \stackrel{\delta^B_V}{\rightarrow}
B (V; V) \cong  \tau B (V) \otimes V,
\end{eqnarray}
where the isomorphism is given by  Proposition \ref{prop:sigma_tau_BM;M}. This morphism is natural with respect to the $\fb\op$-module $B$, by Proposition \ref{prop:delta_morphism}.

By Proposition \ref{prop:strong_duality} and Remark \ref{rem:strong_duality}, one can form the following:

\begin{defn}
\label{defn:addelta}
For $V \in \ob \modr$, let $\addelta_V^B : 
 B(V) \otimes V^\sharp
\rightarrow 
\tau B(V)$ be  the 
 adjoint to (\ref{eqn:delta}).
\end{defn}

Recall from Proposition \ref{prop:duality:spmon} that $V \mapsto V^\sharp$ gives  a  functor $\spmon \rightarrow \modr$. This allows naturality to be considered using the following:

\begin{lem}
\label{lem:naturality_duality_spmon}
For $F$ a functor from $\spmon$ to $\Rmod$, 
\begin{enumerate}
\item 
$V \mapsto F(V) \otimes V^\sharp$ defines a functor $\spmon \rightarrow \Rmod$, where 
$F(V) \otimes V^\sharp$ is considered as the tensor product of $F$ with $V \mapsto V^\sharp$;
\item 
$V \mapsto \hom_R (V, F(V))$ defines a functor  $\spmon \rightarrow \Rmod$, where a morphism $(i,r): V \rightarrow W$ sends $f : V \rightarrow F(V)$ to the composite
$
W \stackrel{r}{ \rightarrow} V \stackrel{f}{\rightarrow} F(V) \stackrel{F(i)}{\rightarrow} F(W)$;
\item 
the isomorphism 
$F(V) \otimes V^\sharp
 \cong 
 \hom_R (V, F(V))$  of Lemma \ref{lem:strong_duality} 
 is natural with respect to $\spmon$ for the above structures. 
\end{enumerate}
\end{lem}

\begin{rem}
Lemma \ref{lem:naturality_duality_spmon} applies, in particular, if $F$ is the composite of a functor from $\modr$ to  $\Rmod$ with the forgetful functor $\spmon \rightarrow \modr$. 
\end{rem}

\begin{prop}
\label{prop:naturality_addelta}
Let $V \in \ob \spmon$. 
\begin{enumerate}
\item
The morphism $\addelta_V^B : 
 B(V) \otimes V^\sharp
\rightarrow 
\tau B(V)$ is a natural transformation of functors from $\spmon$ to $\Rmod$, where the domain is equipped with 
the structure given by Lemma \ref{lem:naturality_duality_spmon} and $\tau B $ is considered as a functor on $\spmon$ via the forgetful functor $\spmon \rightarrow \modr$.
\item 
The morphism $\addelta_V^B$ is natural with respect to the $\fb\op$-module $B$.
\end{enumerate}
\end{prop}

\begin{proof}
The first statement is a  case of the following general result. Suppose that $F(V) \rightarrow G(V) \otimes V$ is a natural transformation of functors from $\modr$ to $\Rmod$, where $F, G$ are functors on $\modr$. Naturality with respect to $i: V \rightarrow W$ translates (using the isomorphism of Lemma \ref{lem:strong_duality})  into the commutative diagram of solid arrows,
\[
\xymatrix{
&
F(V) \otimes V^\sharp 
\ar[r]
\ar@{.>}@/_1pc/[ld]
&
G(V) 
\ar[dd]^{G(i)} 
\\
F(V) \otimes W^\sharp 
\ar[ur]_{F(V) \otimes i^\sharp}
\ar[dr]^{F(i) \otimes W^\sharp}
\\
&
F(W) \otimes W^\sharp 
\ar[r]
&
G(W).
}
\]
The dotted arrow indicates the morphism $F(V) \otimes r^\sharp$ induced by some retract $r: W\rightarrow V$ to $i$.

Hence, taking $(i,r) \in \hom_{\spmon}( V, W)$, one obtains a commutative diagram 
\[
\xymatrix{
F(V) \otimes V^\sharp 
\ar[r]
\ar[d]_{F(i)\otimes r^\sharp}
&
G(V) 
\ar[d]^{G(i)}
\\
F(W) \otimes W^\sharp 
\ar[r]
&
G(W),
}
\]
as required. 

Naturality with respect to $B$ follows from the naturality of $\delta^B_V$ given by Proposition \ref{prop:delta_morphism} together with the naturality of the isomorphism of Proposition \ref{prop:strong_duality}.
\end{proof}

\begin{rem}
Via the isomorphism $B(V) \otimes V^\sharp
 \cong 
 \hom_R (V, B(V))$ of functors on $\spmon$ furnished by Lemma \ref{lem:naturality_duality_spmon}, $ \addelta^B_V$ can be considered as a natural transformation $ 
 \hom_R (V, B(V) ) 
 \rightarrow 
 \tau B (V)$  of functors on $\spmon$.
 \end{rem}
 
By definition, $ \derpt (\opd (R\oplus V)) = \hom _R (R  , \opd (V; R))$, hence there is an isomorphism $\derpt (\opd (R\oplus V)) \cong \opd (V; R)$ of $R$-modules.
 Thus, by Proposition \ref{prop:sigma_tau_BM;M}, 
\begin{eqnarray}
\label{eqn:derpt}
\derpt (\opd (R\oplus V)) \cong \tau \opd (V).
\end{eqnarray}

The objects appearing in the following statement are graded by Proposition \ref{prop:grading_der} and Definition \ref{defn:derpt}.

\begin{cor}
\label{cor:naturality_addelta_opd}
For $\opd$ an operad, the morphism  
 $ 
 \addelta^\opd _{V} : 
 \der (\opd (V)) \rightarrow \derpt (\opd (R\oplus V))
 $
 is grading-preserving and is natural  with respect to $V \in \ob \spmon$.
\end{cor}

The morphism $ \addelta^\opd _{V}$ is referred to as the generalized contraction map, since it generalizes Satoh's contraction map, as indicated in the following example.

\begin{exam}
\label{exam:addelta_opd}
Consider $\opd \in \{\comopd,  \lieopd, \assopd \}$, using the identifications given in Example \ref{exam:tau_opd}. 
\begin{enumerate}
\item 
$\addelta_V ^{\comopd} : \der (\comopd (V)) \rightarrow S (V)$ is the usual divergence map;  
\item 
$
\addelta_V^{\lieopd} : \der (\lieopd (V)) \cong \lieopd (V) \otimes  V^\sharp 
\rightarrow 
T (V)
$
identifies with the contraction morphism defined by Satoh \cite[Section 3]{MR2864772};
\item 
$
\addelta_V^{\assopd} : \der (\assopd (V))  
 \rightarrow 
T(V) \otimes T(V)\op
$
is the precursor (before passage to the quotient modulo commutators) of the double divergence \cite[Section 3.1]{MR3758425}.
\end{enumerate}

The morphism of operads $\lieopd \rightarrow \assopd$ induces the commutative diagram:
\[
\xymatrix{
\der (\lieopd (V))
\ar[d]
\ar[r]^{\addelta^\lieopd_V} 
&
T(V) 
\ar[d]
\\
\der (\assopd (V))
\ar[r]_{\addelta^\assopd_V} 
&
T(V) \otimes T(V)\op,
}
\]
where the  vertical arrows are given by Examples \ref{exam:der_lie_ass} and  \ref{exam:nat_assoc_alg_opd} respectively.
\end{exam}

\begin{exam}
\label{exam:contract_fo}
Consider the free operad $\fo$ on a graded generating set $\gen$  and take $V = R[S]$ for a finite set $S$. By Proposition  \ref{prop:deriv_fo}, $\der (\fo (R[S]))$ has basis indexed by $\rpt (S)$; after enlargement to the pointed set $(S_+,+)$, as in Proposition  \ref{prop:pointed_der_fo},  $\derpt (\fo (R[S_+, +]))$ has sub-basis given by pointed $\gen$-trees with the root labelled by $+$. 

The generalized contraction
$$
\addelta^{\fo} _{R[S]} 
:
\der (\fo (R[S]))
\rightarrow
\derpt (\fo (R[S_+,+]))
$$
sends $\tr\in \rpt (S)$ to the sum of the trees $\tr' \in \rpt(S_+)$ that can be obtained from $\tr$ by replacing the root label $z:= \rt (\tr)$ by $+$ and  relabelling one of the $z$-labelled leaves of $\tr$ by $+$. 

This serves as a universal example as follows. If $\opd$ is a reduced operad, then there exists a graded set of generators $\gen$ and a surjection $\fo \twoheadrightarrow \opd$ of operads. This induces a natural commutative diagram:
\[
\xymatrix{
\der (\fo (R[S]))
\ar@{->>}[d]
\ar[rr]^(.45){\addelta^{\fo}_{R[S]}} 
&&
\derpt (\fo (R[S_+,+]))
\ar@{->>}[d]
\\
 \der (\opd (R[S]))
\ar[rr]_(.4){\addelta^\opd_{R[S]}} 
&&
\derpt (\opd (R[S_+,+])),
}
\]
in which the surjectivity of the vertical morphisms follows from Proposition \ref{prop:surject_der}.
Hence the top horizontal map determines $\addelta^\opd_{R[S]}$.
\end{exam}

\subsection{The kernel of $ \addelta^B_V $}
The proof of Theorem \ref{THM:main} requires information on the kernel of $\addelta^\opd _V$. To this end, Definition \ref{defn:derkerphi} introduces $\derkerphi (\opd (-))$.

The material of this subsection is slightly technical and is only used in Part \ref{part:structure}, so the reader may prefer to  pass directly to Section \ref{subsect:gen_trace} on first reading.

\begin{defn}
\label{defn:derkerphi}
Let $\derkerphi (\opd (-)) \subset \der (\opd(-))$ be the smallest subfunctor of $ \der (\opd (-)) : \spmon \rightarrow \Rmod$ that contains $\opd (\vbar)\subset  \der (\opd (V))$ for each decomposition $V\cong R\oplus \vbar$, where the natural inclusion 
\[
\opd(\vbar) 
\cong 
\hom_R (R, \opd (\vbar)) 
\hookrightarrow \hom_R (V, \opd (V)) \cong  \der (\opd (V))
\]
is induced by the projection $V \twoheadrightarrow R$ and the inclusion $\opd(\vbar) \hookrightarrow \opd(V)$ given by $\vbar \subset V$. 
\end{defn}

The following identification, which follows directly from the definitions, illustrates the inclusion used in Definition \ref{defn:derkerphi}; it uses  the notion of a disjoint $S$-labelled $\gen$-tree from Definition \ref{defn:arboriculture}.

\begin{prop}
\label{prop:fo_disj}
Let $\fo$ be a free operad on the graded generating set $\gen$ and $(S,z)$ be a finite pointed set. 
 The subset $\fo (\overline{R[S]}) \subset \derpt (\fo (R[S,z])) $ has sub-basis 
 given by the set of disjoint  $S$-labelled rooted planar $\gen$-trees with root $z$.
\end{prop}

\begin{rem}
Proposition \ref{prop:fo_disj} motivated the notation $\derkerphi$ introduced in Definition \ref{defn:derkerphi}.
\end{rem}

\begin{prop}
\label{prop:derkerphi}
\ 
\begin{enumerate}
\item
The subfunctor $\derkerphi (\opd (-))$ of $\der (\opd(-))$ is contained  in $\ker \addelta^\opd_- $.
\item 
The inclusion $\derkerphi (\opd (-)) \subset \der (\opd(-))$ is natural with respect to the operad $\opd$, where $\der (\opd (-))$ is considered as a functor of $\opd$ by Proposition \ref{prop:nat_der}.
\end{enumerate}
\end{prop}

\begin{proof}
For the first statement, it suffices to show that for $V =R \oplus \vbar$ an object of $\sppt$,  $\opd  (\vbar)\subset \hom_R(V, \opd(V))$ is contained in the kernel of $\addelta^\opd_V$. 

We require to show that the composite
$
\opd(\vbar) \cong 
\opd(\vbar ) \otimes R^\sharp 
\hookrightarrow 
\opd(V) 
\stackrel{\addelta^\opd_V} {\rightarrow} 
\tau \opd (V) 
$ 
is zero. By adjunction, it is equivalent to show that the composite:
\[
\opd(\vbar ) 
\hookrightarrow 
\opd(V) 
\stackrel{\delta^\opd_V}{\rightarrow}
\tau \opd( V) \otimes V
\rightarrow 
\tau \opd( V) \otimes R
\]
is zero, where the first morphism is induced by $\vbar \subset V$ and the last by the projection $V \twoheadrightarrow R$. 

By naturality, the composite $\opd(\vbar ) \rightarrow \tau \opd( V) \otimes V$ factorizes across 
$ 
\delta^\opd_{\vbar} : \opd(\vbar ) \rightarrow \tau \opd( \vbar ) \otimes \vbar
$ 
via the inclusion induced by $\vbar \subset V$. The result follows, since the composite $\vbar \hookrightarrow V \twoheadrightarrow R$ is zero.

 The naturality with respect to $\opd$ is clear from the construction.
\end{proof}

The definition of  $\derkerphi (\opd (-))$ is made more explicit by the following, which is a consequence of Proposition \ref{prop:spmon_finj}:

\begin{lem}
Suppose that all finitely-generated stable free $R$-modules are free.
Then, for $R \oplus \vbar$ in $\sppt$, with underlying object $V \in \ob \spmon$,  $\derkerphi (\opd (V))$ is the sub $\aut (V)$-module of $\der (\opd (V))$ generated by $\opd (\vbar)$. 
\end{lem}

 \subsection{The generalized  divergence}
\label{subsect:gen_trace}
 
Let $R \oplus V$ be an object of $\sppt$, considering $V$ as an object of $\spmon$. By Theorem \ref{thm:assoc_alg}, $ \derpt (\opd (R \oplus V))$ has a natural associative algebra structure;  this allows the following Definition to be given:

 \begin{defn}
\label{defn:|addelta|}
 For $V \in \ob \spmon$, let $  \gendiv^\opd _{V}  :
  \der (\opd (V)) \rightarrow |\derpt (\opd (R \oplus V))|$ be the composite of $ \addelta^\opd _{V}$ with the passage to the quotient modulo commutators.
 \end{defn} 
 
 Recall that $\der (\opd (-))$ is natural with respect to the operad $\opd$ by Proposition \ref{prop:nat_der} and $  |\derpt (\opd (R\oplus  -))|$ by Corollary \ref{cor:derpt_module}. Corollary \ref{cor:naturality_addelta_opd} gives:
 
\begin{prop}
\label{prop:|addelta|_natural}
For an operad $\opd$ and $V \in \ob \spmon$, 
$
   \gendiv^\opd _{V}  :
  \der (\opd (V)) \rightarrow |\derpt (\opd (R \oplus V))|
  $ 
  is a natural transformation of functors from $\spmon$ to $\Rmod$.
 Moreover,   $\gendiv^\opd_- $ is  natural with respect to the operad $\opd$. 
\end{prop}

\begin{exam}
\label{exam:gen_trace}
\ 
\begin{enumerate}
\item
For $\opd= \lieopd$, by Example \ref{exam:addelta_opd} one obtains the  Satoh trace  \cite{MR2864772,MR2269583} (see also \cite{MR2846914}).  
\item 
For $\opd = \assopd$, one obtains the double divergence $\mathrm{Div}$ of \cite{MR3758425}. 
 Naturality with respect to $\lieopd \rightarrow \assopd$ gives the commutative diagram (for $V \in \ob \spmon$):
\[
\xymatrix{
\der (\lieopd (V))
\ar[d]
\ar[rr]^{\gendiv^\lieopd_V} 
&&
|T(V)| 
\ar[d]
\\
\der (\assopd (V))
\ar[rr]_(.45){\gendiv^\assopd_V} 
&&
|T(V) \otimes T(V)\op|
}
\]
(cf. Example \ref{exam:addelta_opd}).  This is the compatibility between the Satoh trace and the double divergence (cf. \cite[Lemma 8.1]{2018arXiv180409566A}).
\end{enumerate}
\end{exam}

\begin{exam}
\label{exam:comopd_gen_trace}
For $\comopd$ and $V \in \ob \spmon$,  since the algebra $\derpt (\comopd (R \oplus V)) \cong S (V)$ is commutative, the passage to the quotient modulo commutators changes nothing. Thus the generalized divergence $\gendiv^\comopd_V$ identifies with $\addelta^\comopd_V$ and with the divergence
$ 
\der (\comopd (V)) 
\rightarrow 
S(V)
$.

The morphism of operads $\assopd \rightarrow \comopd$ gives the following compatibility between the double divergence and the divergence:
\[
\xymatrix{
\der (\assopd (V))
\ar[rr]^(.45){\gendiv^\assopd_V} 
\ar@{->>}[d]
&&
|T(V) \otimes T(V)\op|
\ar@{->>}[d]
\\
\der (\comopd (V))
\ar[rr]_(.45){\gendiv^\comopd_V} 
&&
S (V).
}
\]
\end{exam}

Suppose now that $V$ is itself pointed, say $V = R\oplus \vbar$, so that there is an associated morphism 
$\vbar  \rightarrow V$ in $\spmon$. Thus one can consider pointed derivations (cf. Definition \ref{defn:derpt}): 
\[
\derpt (\opd (R \oplus \vbar)) \subset \der (\opd (V))
\] 
and restrict $\addelta^\opd _{V}$ (respectively $\gendiv^\opd _{V}$) to these. These restrictions  are identified  by the following:

\begin{prop}
\label{prop:addelta_point_quotient}
Let $V = R \oplus \vbar$ in $\sppt$. 
\begin{enumerate}
\item
The restriction of $\addelta^\opd _{V}$ to $\derpt (\opd (R \oplus \vbar)) \subset \der (\opd (V))$ is the monomorphism 
\begin{eqnarray}
\label{eqn:mono_ptd_der}
\derpt (\opd (R \oplus \vbar )) 
\hookrightarrow 
 \derpt (\opd (R\oplus V))
\end{eqnarray}
induced by $\vbar  \rightarrow V$ in $\modr$.
\item 
There is a natural commutative diagram 
\[
\xymatrix{
\derpt (\opd (R\oplus \vbar)) 
\ar@{^(->}[r]
\ar@{->>}[d]
&
  \der (\opd (V))
  \ar[d]^{  \gendiv^\opd _{V}}
  \\
|\derpt (\opd (R \oplus \vbar))|
\ar[r]
&
|\derpt (\opd (R\oplus V))|   
}
\]
 in which the left hand vertical arrow is the canonical surjection given by the associative algebra structure of Theorem \ref{thm:assoc_alg} and the bottom horizontal arrow is induced by $\vbar \rightarrow V$ in $\modr$.
\end{enumerate}
\end{prop} 
 
\begin{proof}
The identification of the restriction of $\addelta^\opd_V$ follows from an analysis of the definition of $\addelta^\opd _{V}$ and of $\derpt (\opd (-)) $. (This is transparent in the case of the free operad $\fo$ from the explicit description given in Example \ref{exam:contract_fo}; the general case can be deduced from this.) The fact that (\ref{eqn:mono_ptd_der}) is a monomorphism follows from Proposition \ref{prop:der_split_monos}.

The statement for $\gendiv^\opd _{V}$ then follows from the naturality of the algebra structure given by Theorem \ref{thm:assoc_alg}.
\end{proof}

\begin{cor}
\label{cor:kernel_addelta_ptd}
Let $V = R \oplus \vbar$ in $\sppt$. 
\begin{enumerate}
\item
The kernel of $\gendiv^\opd _{V}$ restricted to 
$\derpt (\opd (R \oplus \vbar)) \subset \der (\opd (V))$ maps under the inclusion (\ref{eqn:mono_ptd_der}) to the kernel of the quotient map:
\[
\derpt (\opd (R\oplus V)) 
\twoheadrightarrow 
|\derpt (\opd (R \oplus V))|.
\]
\item 
If $|\derpt (\opd (R \oplus \vbar))|
\rightarrow
|\derpt (\opd (R \oplus V))|  
$ is injective, then the kernel of $\gendiv^\opd _{V}$ restricted to 
$\derpt (\opd (R \oplus \vbar)) $ identifies with the kernel of the projection 
 $$
\derpt (\opd (R\oplus \vbar)) 
\twoheadrightarrow
|\derpt (\opd (R \oplus \vbar))|.
$$
\end{enumerate}
\end{cor}

\subsection{$1$-surjectivity of the generalized contraction and divergence}
In this Section, we establish one of the ingredients of Theorem \ref{THM:main}, that the natural generalized contraction map  
 $
 \addelta^\opd _{V} : 
 \der (\opd (V)) \rightarrow \derpt (\opd (R\oplus V)) 
$ is almost surjective. This uses the notion of $1$-surjectivity for functors on $\spmon$, as in Definition \ref{defn:surj_torsion}.

 \begin{prop}
 \label{prop:addelta_1-surjective}
 The natural transformations of functors on $\spmon$ 
  \begin{eqnarray*}
 \addelta^\opd _{V} &:& 
 \der (\opd (V)) \rightarrow \derpt (\opd (R \oplus V)) 
 \\
  \gendiv^\opd _{V}  &:&
  \der (\opd (V)) \rightarrow |\derpt (\opd (R \oplus V))|
 \end{eqnarray*}
 are $1$-surjective.  
 \end{prop}
 
 \begin{proof}
 The $1$-surjectivity of $\addelta^\opd_-$ implies that of $\gendiv^\opd_- $, hence consider the former. 
 
For  $V \in \ob \spmon$, set $W := R \oplus V$ equipped with the evident  morphism $V \rightarrow W$ of $\spmon$.  Consider the following diagram:
\[
\xymatrix{
 \der (\opd (V))
\ar[r]^(.4){ \addelta^\opd _{V}}
\ar@{^(->}[d]
&
\derpt (\opd (R \oplus V))
\ar@{^(->}[d] 
\ar@{.>}[ld]
\\
 \der (\opd (W))
 \ar[r]_(.4){ \addelta^\opd _{W}}
&
\derpt (\opd (R \oplus W))
}
\] 
in which the vertical arrows are induced by $V \rightarrow W$ and  the dotted arrow is given by forgetting the basepoint of $W \cong R \oplus V$. 

The outer square commutes, by the naturality of $ \addelta^\opd _{-}$ given by Corollary \ref{cor:naturality_addelta_opd}; the lower triangle commutes, by Proposition \ref{prop:addelta_point_quotient}. In particular,   the  commutative lower triangle exhibits the $1$-surjectivity of $ \addelta^\opd _{-}$. 
 \end{proof}

 \subsection{The $1$-cocycle condition}
By Corollary \ref{cor:derpt_module}, $|\derpt (\opd (R \oplus V))|$ takes values naturally in right modules over the Lie algebra $\der (\opd (V))$.   The following result is a generalization of \cite[Proposition 3.1]{MR3758425} from the case $\opd = \assopd$ to that of an arbitrary reduced operad. 
 
 \begin{thm}
 \label{thm:1_cocycle}
 Let $\opd$ be a reduced operad and $V \in \ob \spmon$. Then the natural morphism 
  \[
  \gendiv^\opd _{V}  :
  \der (\opd (V)) \rightarrow 
  |\derpt (\opd (R \oplus V))|
 \]
 is  a $1$-cocycle for the Lie algebra $  \der (\opd (V)) $.
 \end{thm}
 
 \begin{proof}
 Proposition \ref{prop:|addelta|_natural} gives that $ \gendiv^\opd _{V}$ is natural with respect to the operad. Hence, given a morphism of operads $\opd \rightarrow \ppd$, there is a commutative diagram 
\[
 \xymatrix{
  \der (\opd (V)) 
\ar[r]^(.4){ \gendiv^\opd _{V}}
 \ar[d]
 &
  |\derpt (\opd (R \oplus V))|
\ar[d] 
   \\
  \der (\ppd (V)) 
  \ar[r]_(.4){ \gendiv^\ppd _{V} }
&
    |\derpt (\ppd (R \oplus V))|.
 }
 \]
 Moreover,  Proposition \ref{prop:nat_der} gives that the left hand vertical arrow is a morphism of Lie algebras and Corollary \ref{cor:derpt_module} that the right hand vertical morphism is a morphism of right $\der (\opd (V))$-modules.
 
 If $\opd \rightarrow \ppd$ is surjective, then Proposition \ref{prop:surject_der} implies that $\der (\opd (V)) \rightarrow \der (\ppd (V))$ is surjective. 
Using these points, one reduces to the case where $\opd = \fo$ is a free operad (see Section \ref{subsect:fo}). 

Take $V= R[S]$, for a finite set $S$, so that $R \oplus V = R[S_+]$, pointed by $+$. By Proposition  \ref{prop:deriv_fo},  $ \der (\fo (R[S]))$ has a basis indexed by $\rpt(S)$ and with the preLie structure  induced by grafting of trees. 

Consider $\tr_1, \tr_2 \in \rpt(S)$ with $\rt(\tr_1)=x$ and $\rt(\tr_2)=y$. The preLie product $\tr_1 \lhd \tr_2$ is  the sum of the possible graftings of the root of $\tr_2$ to a leaf of $\tr_1$ labelled by $y$. For each tree $\tr'$ occurring in $\tr_1 \lhd \tr_2$, $\rt (\tr')=x$.

The generalized contraction $ \addelta^{\fo}_{R[S]}$ is described in Example \ref{exam:contract_fo}. For a tree $\tr'$ occurring in $\tr_1 \lhd \tr_2$, this depends on the leaves of $\tr'$ that are labelled by $x$. There are two possibilities:
either the leaf originated in $\tr_1$ or 
it originated in $\tr_2$.
  These possibilities are illustrated schematically by:

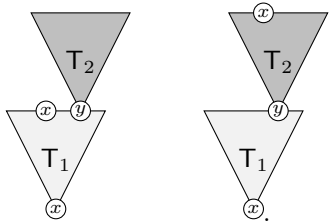
\qquad
\qquad
\begin{tikzpicture}[scale = 1.3]
\draw [fill=gray!10] (0,0) -- (.5,1) -- (-.5,1) -- (0,0);
\draw [fill=lightgray] (0.25,1) -- (.75,2) -- (-.25,2) -- (0.25,1);
\draw [fill=white] (0,0) circle [radius = 0.1];
\node at (0,0) {$\scriptstyle{x}$};
\draw [fill=white] (-.1,1) circle [radius = 0.1];
\node at (-.1,1) {$\scriptstyle{x}$};
\draw [fill=white] (0.25,1) circle [radius = 0.1];
\node at (0.25,1) {$\scriptstyle{y}$};
\node at (0, .5) {$\tr_1$};
\node at (.25,1.5) {$\tr_2$};
\draw [fill=gray!10] (2,0) -- (2.5,1) -- (1.5,1) -- (2,0);
\draw [fill=lightgray] (2.25,1) -- (2.75,2) -- (1.75,2) -- (2.25,1);
\draw [fill=white] (2,0) circle [radius = 0.1];
\node at (2,0) {$\scriptstyle{x}$};
\node [below right] at (2,0) {.};
\draw [fill=white] (2.1,2) circle [radius = 0.1];
\node at (2.1,2) {$\scriptstyle{x}$};
\draw [fill=white] (2.25,1) circle [radius = 0.1];
\node at (2.25,1) {$\scriptstyle{y}$};
\node at (2, .5) {$\tr_1$};
\node at (2.25,1.5) {$\tr_2$};
\end{tikzpicture}

In the second case, the corresponding contribution to the image of $\tr'$ under $ \addelta^{\fo} _{R[S]}$  is equal to $\tr'_1 \lhd \tr'_2$ in $\derpt (\fo (R[S_+]))$, where $\tr'_1, \tr'_2 \in \derpt (\fo (R[S_+]))$; this is  represented by:

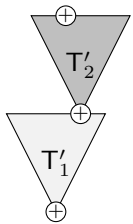
\qquad \qquad
\begin{tikzpicture}[scale=1.3]
\draw [fill=gray!10] (2,0) -- (2.5,1) -- (1.5,1) -- (2,0);
\draw [fill=lightgray] (2.25,1) -- (2.75,2) -- (1.75,2) -- (2.25,1);
\draw [fill=white] (2,0) circle [radius = 0.1];
\draw [fill=white] (2.1,2) circle [radius = 0.1];
\draw [fill=white] (2.25,1) circle [radius = 0.1];
\node at (2, .5) {$\tr'_1$};
\node at (2.25,1.5) {$\tr'_2$};
\node at (2,0) {$\scriptstyle{+}$};
\node at (2.1,2) {$\scriptstyle{+}$};
\node at (2.25,1) {$\scriptstyle{+}$};
\end{tikzpicture}

\noindent
obtained by relabelling the indicated $x$, $y$  by $+$.

Calculating  $\addelta^{\fo}_{R[S]} ([\tr_1, \tr_2])$, the contributions from terms of the first form give the terms in the cocycle relation, whereas the terms of the second form vanish on passage to $|\derpt (\fo (R[S_+]))|$. To see this, consider the term $\tr'_1 \lhd \tr'_2$ arising in the image of $\tr'$ as above; this term  is in bijective correspondence with the contribution  $- \tr'_2 \lhd \tr'_1$ that arises when considering $- \tr_2 \lhd \tr_1$ and the resulting commutator $[\tr'_1, \tr'_2]$ vanishes, by definition of $|\derpt (\fo (R[S_+]))|$.
 \end{proof}

\begin{rem}
The hypothesis that $\opd$ is reduced is only  imposed for convenience in reducing to a free operad of the form $\fo$, since the presentation of free operads in Appendix \ref{sect:free} restricts to the reduced case. The result extends to the general case without difficulty.
\end{rem}

\begin{cor}
\label{cor:ker_addelta_lie}
Let $\opd$ be a reduced operad and $V \in \ob \spmon$.
\begin{enumerate}
\item 
$\ker  \gendiv^\opd _{V}  $ is a sub Lie algebra of $\der (\opd (V))$ naturally with respect to $\spmon$.
\item 
$\ker  \gendiv^\opd _{-} $ is natural with respect to $\opd$; namely, for a morphism $\opd \rightarrow \ppd$ of operads, 
$
\ker  \gendiv^\opd _{V} 
\rightarrow 
\ker  \gendiv^\ppd _{V} 
$ 
is a natural morphism of Lie algebras.
\end{enumerate}
\end{cor}

\begin{proof} 
The first statement is an immediate consequence of Theorem \ref{thm:1_cocycle}. 

The naturality statement follows from the naturality of $ \gendiv^\opd _{V} $ given by Proposition \ref{prop:|addelta|_natural} and the naturality of the Lie algebra structure of $\der (\opd (V))$ given by 
Theorem \ref{thm:prelie}.

The naturality with respect to $\opd$ follows from that used in the proof of Theorem \ref{thm:1_cocycle}.
\end{proof}

\subsection{The generalized contraction and divergence  for positive derivations}

Since the grading upon pointed derivations is, by definition, obtained from that on derivations, the notion of positivity carries over to pointed derivations:

\begin{nota}
For $V \in \ob \spmon $, let $\derpt^+ (\opd (R\oplus V)) \subset \derpt (\opd (R \oplus V))$ denote:
\[
\derpt (\opd (R \oplus V)) \cap \der^+ (\opd (R \oplus V)).
\]
\end{nota}

The structure underlying the generalized contraction $\addelta^\opd_-$ and the generalized divergence $\gendiv^\opd_-$  restricts to positive derivations as follows:
 
\begin{prop}
\label{prop:contract+}
For $V\in \ob \spmon$, 
\begin{enumerate}
\item
the natural $\der (\opd( V))$-action on $\derpt (\opd (R\oplus V))$ restricts to an action of the Lie algebra $\der (\opd (V))$ on $\derpt^+ (\opd (R \oplus V))$; 
\item 
the natural morphism  $\addelta^\opd _{V}$ restricts to 
\[
\addelta^\opd _{V}
 : 
 \der^+ (\opd (V)) \rightarrow \derpt^+ (\opd (R \oplus V))
 \]
so that $ \gendiv^\opd _{V } $ restricts to 
 $
 \gendiv^\opd _{V }  :
  \der^+ (\opd (V)) \rightarrow |\derpt^+ (\opd (R \oplus V))|;
 $
 
 These are $1$-surjective as natural transformations on $\spmon$.
 \item 
$ \gendiv^\opd _{V} $ restricts to a $1$-cocycle for $ \der^+ (\opd (V)) $ with values in $|\derpt^+ (\opd (R \oplus V))|$.
\end{enumerate}
These structures are natural with respect to the reduced operad $\opd$.
\end{prop}

\begin{proof}
The first statement follows from the fact that the $\der (\opd (V))$ action  is compatible with the grading, by Proposition \ref{prop:derpt_right_module}, hence preserves the positive derivations. The second statement follows similarly from the grading statement of Corollary \ref{cor:naturality_addelta_opd}; the $1$-surjectivity is given by Proposition \ref{prop:addelta_1-surjective}.

Together with Theorem \ref{thm:1_cocycle}, the above properties give the third statement. 

The naturality with respect to the operad $\opd$ is an immediate consequence of the naturality of the grading given by Definition \ref{defn:schur_functor}.  
\end{proof}

 \section{Distinguished subalgebras of  $\der^+(\opd (-))$ and $\imderlie (-)$}
\label{sect:disting}

There are  two subalgebras  of $\der^+(\opd (-))$ which are the focus of Part \ref{part:structure}, $\derpl (\opd (-))$ and $\derlie (\opd (-))$. These are defined  with respect to the preLie (respectively Lie) structure on $\der (\opd (-))$ and are introduced in Section \ref{subsect:derpl_derplie}.

The main interest is in studying the functors $\derlie (\opd (-)) \subset \der^+(\opd (-))$ and how much these differ. The strategy adopted here employs the generalized divergence $ \gendiv^\opd _{V}$; this is explained in Section \ref{subsect:howto}. 

This relies on some important technical ingredients that are presented here: Section \ref{subsect:special}  introduces special pointed derivations; Section \ref{subsect:imlie} considers the image $\imderlie (V)$ of $\derlie (\opd (V))$ under $ \gendiv^\opd _{V}$. 
The special pointed derivations are exploited in Part \ref{part:structure}  to give some control over the image $\imderlie (V)$.

\medskip
Throughout the Section, the operad $\opd$ is reduced.
 
\subsection{Introducing the subalgebras $\derpl (\opd (-))$ and $\derlie (\opd (-))$}
\label{subsect:derpl_derplie}

By definition of the grading from Section \ref{subsect:grading_deriv}, 
$
\der^1 (\opd (V))
=
\hom_R (V, \opd _2 (V) )
.
$
One can consider the following sub (pre)Lie algebras:

\begin{defn}
\label{defn:derlie}
For $V \in \modr$, let 
\begin{enumerate}
\item 
$\derpl (\opd (V)) \subset \der (\opd (V))$ be the sub preLie algebra  generated by $\der^1 (\opd (V))$;
\item 
$\derlie (\opd (V)) \subset \der (\opd (V))$ be the sub Lie algebra generated by $\der^1 (\opd (V))$.
\end{enumerate}
\end{defn}

\begin{prop}
\label{prop:derpl_derlie_functorial}
\ 
\begin{enumerate}
\item 
$\derpl(\opd(-))$ is a subfunctor of $\der^+ (\opd (-)) : \spmon \rightarrow \prelieopd\dash\alg$; 
\item 
$\derlie(\opd(-))$ is a subfunctor of $\der^+ (\opd (-)) : \spmon \rightarrow \lieopd\dash\alg$; 
\item 
there are natural inclusions of functors from $\spmon$ to $\lieopd\dash\alg$:
\[
\derlie(\opd(-))
\hookrightarrow 
\derpl(\opd(-))
\hookrightarrow 
\der^+(\opd (-)),
\]
where the two right hand terms are given the associated Lie structure.
\end{enumerate}
\end{prop}

\begin{proof}
That $\derpl (\opd( -))$ and $\derlie (\opd (-))$ are both contained within the positive derivations follows from the fact that they are generated by elements of positive degree. 
The result then follows from the naturality of the preLie structure on $\der (\opd (-))$ given by Theorem \ref{thm:prelie}.
\end{proof}

Naturality with respect to the operad is important, based upon the naturality of $\der^+(\opd (-)) $ with  values in preLie-algebras that is given 
 by Proposition \ref{prop:nat_der}.

\begin{prop}
\label{prop:naturality_derprelie_derlie}
The structures given in Proposition \ref{prop:derpl_derlie_functorial} are natural with respect to the operad: 
for $\opd \rightarrow \ppd$ a morphism of reduced operads, there is a natural commutative diagram:
\[
\xymatrix{
\derlie(\opd(-))
\ar@{^(->}[r]
\ar[d]
&
\derpl(\opd(-))
\ar@{^(->}[r]
\ar[d]
&
\der^+(\opd (-))
\ar[d] 
\ar@{^(->}[r]
&
\der(\opd (-))
\ar[d]
\\
\derlie(\ppd(-))
\ar@{^(->}[r]
&
\derpl(\ppd(-))
\ar@{^(->}[r]
&
\der^+(\ppd (-))
\ar@{^(->}[r]
&
\der(\ppd (-)).
}
\]

Moreover, if $\opd \twoheadrightarrow \ppd$ is surjective, then each of the vertical maps is surjective.
\end{prop}

\begin{proof}
As for Proposition \ref{prop:nat_der}, the morphism of operads induces a natural transformation 
\[
\hom_R (V, \opd _2 (V) ) 
\rightarrow
\hom_R (V, \ppd _2 (V) ),
\]
natural with respect to $V \in \ob \spmon$, that is compatible with $\der^+(\opd (V)) \rightarrow \der^+ (\ppd(V))$. Moreover, as in Proposition \ref{prop:surject_der}, this is surjective if $\opd \twoheadrightarrow \ppd$ is. The result follows on passing to the respective subalgebras.
\end{proof}

\subsection{Special pointed derivations}
\label{subsect:special}

Let $R \oplus V$ be an object of $\sppt$, so that we may consider $\derpt(\opd (R \oplus V))$. By Definition \ref{defn:derpt}, it has underlying object $\opd (V; R)$. In 
degree one, $\derpt^1 (\opd (R \oplus V)) = \opd_2(V; R) $, where
\[
\opd_2(V; R) := \opd (V;R) \cap \opd_2 (R \oplus V),
\]
for $\opd_2(R\oplus V)$ as in Definition \ref{defn:schur_functor}.

We note the following:

\begin{lem}
\label{lem:opd2(V;R)}
The association $V \mapsto \opd_2 (V; R)$ defines a functor from $\modr$ to $\Rmod$ that is linear with respect to $V$. 
\end{lem}

\begin{defn}
\label{defn:derspec}
Let $\derspec (\opd (R\oplus V) )$, the special pointed derivations, be  the sub $\der(\opd (V))$-module of $\derpt^+(\opd (R \oplus V))$ generated by 
 $\derpt^1 (\opd (R \oplus V))$.
\end{defn}

\begin{prop}
\label{prop:derspec_functor}
The special pointed derivations $ R \oplus V \mapsto \derspec (\opd (R\oplus V) )$ define a subfunctor of $\derpt^+(\opd (-))$, considered as a functor from $\sppt$ to $\Rmod$. This is natural with respect to the operad $\opd$.
\end{prop}

\begin{rem}
The special derivations are  made explicit in the case of a free binary operad in Section \ref{subsect:bin_pruning} (see Proposition \ref{prop:bases_disjoint_pointed_special}). This example explains the choice of terminology.
\end{rem}

\subsection{The image of $\derlie (\opd (-))$ under $  \gendiv^\opd _{-} $}
\label{subsect:imlie}

We now turn to considering the image of $ \derlie (\opd (-))$ under the generalized divergence.

 For $V \in \ob \spmon$, the natural inclusion of Lie algebras
$
\derlie (\opd (V)) \subset \der^+ (\opd (V))
$
composed with 
 the generalized divergence of Proposition \ref{prop:contract+} 
 yields the composite:
\begin{eqnarray}
\label{eqn:derlie_addelta}
\derlie (\opd (V)) \hookrightarrow
 \der^+ (\opd (V)) \stackrel{\gendiv^\opd _{V} }{\rightarrow}
 |\derpt^+ (\opd (R \oplus V))|.
\end{eqnarray}

\begin{defn}
\label{defn:imderlie}
For $V \in \ob \spmon$, let $\imderlie (V) \subset |\derpt^+ (\opd (R \oplus V))|$
 be the image of $\derlie (\opd (V))$ under the composite (\ref{eqn:derlie_addelta}).
\end{defn}

As in Section \ref{subsect:special}, $\opd_2 (V; R)$ is a submodule of $\derpt^+ (\opd (R\oplus V))$. The following is clear:

\begin{lem}
\label{lem:opd_2(V;R)}
For $V \in \ob \spmon$, the composite
$
\opd_2 (V; R) \subset \derpt^+ (\opd (R\oplus V)) \twoheadrightarrow |\derpt^+ (\opd (R \oplus V))| 
$ 
is injective. 
\end{lem}

By Proposition \ref{prop:contract+},  $ |\derpt^+ (\opd (R \oplus V))|$ is naturally a right $\der(\opd (V))$-module and thus a right  $\derlie(\opd (V))$-module, by restriction along the inclusion of Lie algebras $\derlie(\opd (V))\subset \der(\opd (V))$. One has:

\begin{prop}
\label{prop:first_properties_imderlie}
\ 
\begin{enumerate}
\item 
The association $V \mapsto \imderlie (V)$ defines a subfunctor of $V \mapsto |\derpt^+ (\opd (R \oplus V))|$, considered as a functor from $\spmon$ to $\Rmod$. 
\item 
$\imderlie (V)$ is contained in the  $\derlie (\opd (V))$-submodule of $ |\derpt^+ (\opd (R \oplus V))|$ generated by the image of $\opd_2 (V; R) $ under the inclusion of Lemma \ref{lem:opd_2(V;R)}. 
\end{enumerate}
\end{prop}

\begin{proof}
The first statement follows from the naturality of $\derlie (\opd (-))\subset \der^+(\opd (-))$ given by Proposition \ref{prop:derpl_derlie_functorial}, together with the naturality of $\gendiv^\opd_-$ given by Proposition \ref{prop:|addelta|_natural}. 

For the second statement, first consider the image of  $\hom_R (V, \opd_2 (V)) \subset \derlie (\opd (V))$ under $\gendiv^\opd_{V}$. The morphism $\addelta^\opd _V$, when restricted to $
\hom_R (V, \opd_2 (V))$, takes values in $\opd_2 (V; R)$, by construction. Thus, on passage to $ |\derpt^+ (\opd (R \oplus V))|$, the image of $\hom_R (V, \opd_2 (V))$ lies in the image of 
$\opd_2 (V; R) $.

By Theorem \ref{thm:1_cocycle}, $\gendiv^\opd_V$ is a $1$-cocycle. By definition, $\derlie (\opd (V))$ is generated as a Lie algebra by $\hom_R (V, \opd_2 (V))$;   it follows that the image of $\derlie (\opd (V))$ is contained in the submodule of $ |\derpt^+ (\opd (R \oplus V))|$ generated by the image of $\opd_2 (V; R) $ considered above, as required.
\end{proof}

The construction of $\imderlie$ is natural with respect to the operad, extending the naturality given by Proposition \ref{prop:contract+}:

\begin{prop}
\label{prop:naturality_imderlie}
For $\opd \rightarrow \ppd$ a morphism of reduced operads, the canonical inclusions fit into a commutative  natural diagram
\[
\xymatrix{
\imderlie (V)
\ar@{^(->}[r]
\ar[d]
&
|\derpt^+ (\opd (R \oplus V))|
\ar[d]
\\
\imderlie [\ppd](V) 
\ar@{^(->}[r]
&
|\derpt^+ (\ppd (R \oplus V))|.
}
\] 
\end{prop}

\begin{defn}
\label{defn:imderliespec}
For $ V \in \ob \spmon$, let $\imderliespec (V) \subset |\derpt^+(\opd (R \oplus V))|$ be the image of $\derspec (\opd (R\oplus V) )$ under the composite 
$\derspec (\opd (R\oplus V) ) \subset \derpt^+ (\opd (R \oplus V)) \twoheadrightarrow |\derpt^+(\opd (R \oplus V))|$, where the surjection is the quotient modulo commutators.  
\end{defn}

Proposition \ref{prop:first_properties_imderlie} has the important consequence:

\begin{cor}
\label{cor:image_derspec_imderlie}
For $V \in  \ob \spmon$, there are natural inclusions:
\[
\imderlie (V)
\subseteq
\imderliespec (V)
\subseteq 
 |\derpt^+(\opd (R \oplus V))|.
\] 
\end{cor}

\begin{proof}
The  image $\imderliespec (V)$ of $\derspec (\opd (R\oplus V) )$ in $|\derpt (\opd (R \oplus V))|$ is a sub $\der(\opd (V))$-module, in particular, it is a sub $\derlie(\opd (V))$-module. 
 Moreover, this image contains the image of  $\opd_2 (V; R)$. The result therefore follows from the second statement of Proposition \ref{prop:first_properties_imderlie}.
\end{proof}

\subsection{How to analyse $\derlie (\opd (-))$}
\label{subsect:howto}

It is a fundamental problem to analyse $\derlie (\opd (-))$ and its relationship to $\der^+ (\opd (-))$. For instance, the cokernel of the natural inclusion $\derlie (\opd (-)) \subset \der^+ (\opd (-))$ measures the obstruction to $\der^+ (\opd(-))$ being generated as a Lie algebra by its degree one elements. 

\begin{defn}
\label{defn:KO}
For a reduced operad $\opd$, let $K^\opd(-)$ be the kernel of the surjection 
$ \derlie (\opd (-)) \twoheadrightarrow \imderlie (-)$. 
\end{defn}

The generalized divergence gives rise to the following commutative diagram that underlies the general strategy that is developed here:

\[
\xymatrix{
0 \ar[r]
&
K^\opd (-) 
\ar@{^(->}[r]
\ar@{^(->}[d]
&
\derlie (\opd (-))
\ar@{^(->}[d]
\ar@{->>}[r]
&
\imderlie (-)
\ar@{^(->}[d]
\ar[r]
&
0
\\
0 
\ar[r]
&
\mathrm{Ker}\  \gendiv^\opd _-
\ar@{^(->}[r]
&
\der^+ (\opd (-))
\ar[r]^{\gendiv^\opd _-}
\ar[d]
&
|\derpt^+ (\opd (R \oplus - ) )|
\ar@{->>}[d]
\\
&&
|\derpt^+ (\opd (R \oplus - ) )|/ \imderlie (-)
\ar@{=}[r]
&
|\derpt^+ (\opd (R \oplus - ) )|/ \imderlie (-).
}
\] 
Here: 
\begin{enumerate}
\item 
the rows are exact sequences; 
\item 
the right hand column is short exact and the middle column is a sequence. 
\end{enumerate}

\begin{prop}
\label{prop:KO_lie}
For a reduced operad $\opd$, there are natural inclusions
\[
K^\opd(-) \subset \mathrm{Ker}\  \gendiv^\opd _-  \subset \derlie (\opd (-))
\] 
of subfunctors of $\derlie (\opd (-)) : \spmon \rightarrow \lieopd \dash \alg$. In particular, 
$K^\opd(-)$ and $\mathrm{Ker}\  \gendiv^\opd _-$ take values in Lie algebras. 
\end{prop}

\begin{proof}
This follows directly from the fact that $\gendiv^\opd _V$ is a $1$-cocycle, for $V \in \ob \spmon$, by Theorem \ref{thm:1_cocycle}, and  $\derlie (\opd (V)) \subset \der^+ (\opd (V))$ is a sub Lie algebra, by construction.
\end{proof}

\begin{rem}
The diagram reduces the problem of understanding $\derlie (\opd (-))$ to the study of the subfunctor $K^\opd(-)$ and of the functor $\imderlie (-)$ together with the analysis of the extension of functors from $\spmon$ to $\Rmod$:
\[
0
\rightarrow 
K^\opd (-) 
\rightarrow 
\derlie (\opd (-))
\rightarrow 
\imderlie (\opd (-))
\rightarrow 
0.
\]
\end{rem}

\begin{prop}
\label{prop:KO_ker}
For a reduced operad $\opd$, the cokernel of $K^\opd (-) \hookrightarrow \mathrm{Ker}\  \gendiv^\opd _-$ is naturally isomorphic to the middle homology of the sequence 
\[
\derlie (\opd (-))
\hookrightarrow
\der^+ (\opd (-)) 
\rightarrow 
|\derpt^+ (\opd (R \oplus - ) )|/ \imderlie (-).
\] 

There is a short exact sequence of functors from $\spmon$ to $\Rmod$:
\[
0
\rightarrow 
\mathrm{Ker}\  \gendiv^\opd _- /K^\opd (-)
\rightarrow 
\der^+ (\opd (-))/ \derlie (\opd (-))
\rightarrow 
\mathrm{Image}\  \gendiv^\opd _-  / \imderlie (-)
\rightarrow 
0.
\]
\end{prop}

\begin{proof}
One can modify the diagram by replacing $|\derpt^+ (\opd (R \oplus -))|$ by the image of $\gendiv^\opd _- $, since $\imderlie (-)$ is a subfunctor of $\mathrm{Image}\ \gendiv^\opd _-  \subset |\derpt (\opd (R \oplus -))|$, by construction. 

With this modification, the middle row of the diagram becomes a short exact sequence and the second map of the middle column a surjection. The first statement then follows from the long exact sequence in homology associated to the diagram; the second is a reformulation. 
\end{proof}

\begin{rem}
The result of Proposition \ref{prop:KO_ker} can be interpreted as follows:
\begin{enumerate}
\item 
the functor $\mathrm{Image}\  \gendiv^\opd _-  / \imderlie (-)$ approximates $\der^+ (\opd (-))/ \derlie (\opd (-))$; 
\item 
the subfunctor $\mathrm{Ker}\  \gendiv^\opd _- /K^\opd (-)$ is an error term; 
\item 
this also governs the difference between the functor $K^\opd (-)$ and $\mathrm{Ker} \ \gendiv^\opd _- $.
\end{enumerate}

The main result of Part \ref{part:structure} (see Theorem \ref{thm:ses_up_to_torsion}), in the case of a binary operad, gives a precise sense in which the error term $\mathrm{Ker} \ \gendiv^\opd _-  /K^\opd (-)$ is small. In particular, it is a torsion functor on $\spmon$. 
\end{rem}

\subsection{Examples}

So as to indicate that very different behaviour can occur, the examples $\lieopd$, $\assopd$ and $\comopd$ are considered, exploiting the Examples of Section \ref{subsect:gen_trace}. These are all binary operads, so can be analysed further by the methods of Part \ref{part:structure} below. The functors below are evaluated on $V$, a finite-rank free $R$-module. 

First consider the case of the Lie operad, so that $\gendiv^\lieopd _-$ corresponds to Satoh's trace map.

\begin{exam}
\label{exam:lieopd_imderlie}
For $\opd = \lieopd$, $\derpt (\lieopd (R \oplus V)) \cong T(V)$ (see Proposition \ref{prop:derpt_lie_assoc}) and 
hence $|\derpt (\lieopd (R \oplus V))|= |T(V)|$. The image of $\hom_R (V, \lieopd _2 (V))$  in $|\derpt (\lieopd (R \oplus V))|$ identifies as
$
V \subset |T(V)|$. Moreover,  $\imderlie[\lieopd](V) = \imderliespec [\lieopd] (V)=  V \subset |T(V)|$; this follows by considering the sub $\der (\opd (V))$-module of  $|T(V)|$ generated by $V$: the antisymmetry of the Lie bracket implies that all higher terms vanish.

This gives the short exact sequence 
\[
0
\rightarrow 
K^\lieopd (V) 
\rightarrow 
\derlie (\lieopd (V))
\rightarrow 
V 
\rightarrow 
0,
\]
showing that $K^\lieopd (V)$ contains most of the information on $\derlie (\lieopd (V))$.

The fact that the functor $V \mapsto |\derpt (\lieopd (R \oplus V))|/ \imderlie[\lieopd] (V)$ is highly non-trivial suggests that $V \mapsto \der^+ (\lieopd (V))/ \derlie (\lieopd (V))$ is also; this is made precise by Theorem \ref{thm:ses_up_to_torsion}.
\end{exam}

In the case of the associative operad, $\gendiv^\assopd_-$ corresponds to the double divergence. 

\begin{exam}
\label{exam:assopd_imderlie}
For $\opd = \assopd$, $\derpt (\assopd (R \oplus V)) \cong T(V) \otimes T(V)\op$ (see Proposition \ref{prop:derpt_lie_assoc}) and 
hence $|\derpt (\assopd (R \oplus V))|= |T(V)\otimes T(V)\op|$. The image of $\hom_R (V, \assopd _2 (V))$ in $|\derpt (\assopd (R \oplus V))|$ identifies as: 
\[
V^{ \oplus 2} = V\otimes R\  \oplus R \ \otimes V \subset |T(V)\otimes T(V)\op |.
\]

Contrary to the case of the Lie operad, the description of $\imderlie[\assopd] (V)$ is not straightforward. One has: 
\[
\imderlie[\assopd] (V) 
\subseteq
\imderliespec [\assopd] (V) =  
|T(V) \otimes R |   \oplus  [R \otimes T(V)\op| \cong |T(V)|   \oplus  |T(V)\op|.
\]
and the inclusion  $\imderlie[\assopd] (V) 
\subseteq
\imderliespec [\assopd] (V) $ is proper if $V\neq 0$; for example, for $V=R$, it corresponds to the diagonal inclusion $R \subset R^{\oplus 2}$.

This gives the exact sequence:
\[
0
\rightarrow 
K^\assopd (V) 
\rightarrow 
\derlie (\assopd (V))
\rightarrow 
|T(V)|   \oplus  |T(V)\op|
\]
that relates $K^\assopd (V) $ and 
$\derlie (\assopd (V))$.  

Since $|\derpt^+(\assopd (R \oplus -)) |/ \imderlie [\assopd ](V)$ surjects onto $|\overline{T}(V)\otimes \overline{T}(V)\op |$ by the above,   $\der^+(\assopd (V))/ \derlie (\assopd (V))$ is highly non-trivial.
\end{exam}

The behaviour for the commutative operad $\comopd$ is very different. This is most transparent when working over  $R= \rat$, when:

\begin{prop}
\label{prop:case_comopd}
Let $R = \rat$. For $V \in \ob \spmon[\rat]$, 
\begin{enumerate}
\item 
the generalized divergence, $\gendiv^\comopd_V  : \der^+ (\comopd (V)) \rightarrow \overline{S} (V)$, is surjective;
\item 
the inclusion $\derlie (\comopd (V) ) \subset \der^+ (\comopd (V))$ is an equality. 
\end{enumerate}
Hence there is a short exact sequence 
\[
0
\rightarrow 
K^\comopd (V) 
=
\mathrm{Ker}\  \gendiv^\comopd_V
\rightarrow 
\derlie (\comopd (V) ) = \der^+ (\comopd (V))
\rightarrow 
\imderlie [\comopd] (V) 
=
\overline{S} (V) 
\rightarrow 
0.
\]
\end{prop} 

\begin{proof}
The first statement is straightforward. The second is proved using the techniques that are employed in Part \ref{part:structure}; however, in the commutative case working over $\rat$, these become much more elementary. The details are left as an exercice for the reader. 
\end{proof}

\part{Further structure of derivations for binary operads}
\label{part:structure}

\section{Binary pruning and the  preLie case}
\label{sect:structure}

The subalgebras    $\derpl (\opd (-))$ and $\derlie (\opd (-))$ of $\der^+(\opd (-))$ were introduced in Section \ref{sect:disting}.  These are of primary interest when the operad $\opd$ is binary, since their generators are defined in terms of $\opd (\mathbf{2})$. In this case, Theorem \ref{thm:derpl} shows that $\derpl (\opd (V))$ coincides with $\der^+(\opd (V))$ except when $V$ has rank $1$; this contrasts with the case $\derlie 
(\opd (V)) \subset \der^+ (\opd (V))$, which is much more subtle. 

In preparation for the proof of Theorem \ref{thm:derpl}, Section \ref{subsect:bin_pruning} introduces  techniques for pruning binary trees. These are applied by reduction to the universal example, namely the free binary  operads $\fob$, as introduced in Notation \ref{nota:free_bin}, which are described in terms of rooted binary planar $\bin_3$-trees. These techniques will be also be applied in the following Sections. 

\subsection{Pruning for binary trees}
\label{subsect:bin_pruning}
The operation of pruning is introduced in Section \ref{subsect:pruning}. Here we focus on the binary case, fixing a set of generators $\bin$, as in Section \ref{sect:free};
 $B_3$ denotes this set considered as graded, concentrated in degree $3$.

\begin{rem}
The set of generators $\bin$ does not intervene explicitly in the constructions below. Hence the principal ideas can be understood by considering the monogenic case, $\bin= \{* \}$.
\end{rem}

In the binary case, there is a simple relationship between the number of internal vertices and the number of leaves. (Recall that $v(\tr)$ denotes the set of internal vertices of a tree $\tr$.)

\begin{lem}
\label{lem:binary_leaves}
\ 
\begin{enumerate}
\item 
If $\tr$ is a rooted binary planar tree, then $\tr$ has $|v(\tr)|+1 $ leaves.  
\item 
If $\tr \in\brpt (S)$ is an $S$-labelled rooted binary planar $\bin_3$-tree, then the corresponding element of $\der (\fob (R[S]))$ has grading $|v(\tr)|$.
\end{enumerate}
\end{lem}

Below, by abuse of notation, the preLie operation $\lhd$ is used at the level of the generators; this is unambiguous, since the labellings ensure that there is a unique possible grafting.

\begin{lem}
\label{lem:prune_S_labelled}
Let $\tr \in \brpt (S)$ with $|v(\tr)|>1$. Then, for each internal edge, the associated trees given by pruning, $\tr'$ and $\tr''$,  inherit a unique $S_+$-labelling such that $\rt (\tr'') = +$ and 
$
\tr = \tr'\lhd \tr''.
$  
Moreover, $\tr''$ is disjoint. 
\end{lem}  

For $\tr$ a rooted (unlabelled) binary planar tree with $|v(\tr)|\geq 1$, consider the internal vertex attached to the root. The non-root edges are identified via the planar condition as the {\em left} and {\em right} edges respectively. Exactly one of the following holds: 

\begin{enumerate}
\item 
$|v(\tr)|= 1$ and both the left and right edges are external  (i.e., not internal); 
\item 
$|v(\tr)|>1$ and one of the following holds:
\begin{enumerate}
\item 
the left edge is internal and the right edge external;
\item 
the left edge is external and the right edge internal; 
\item 
both the left and right edges are internal. 
\end{enumerate}
\end{enumerate}

The following statement gives a labelled version of the above, respecting the numbering of the cases:

\begin{prop}
\label{prop:pruning}
For $\tr \in \brpt (S)$, one of the following holds:
\begin{enumerate}
\item 
\label{item:no_prune}
$|v(\tr)|= 1$ and both the left and right edges are external; 
\item 
\label{item:prune}
$|v(\tr)|>1$ and one of the following holds, in which $|v(\tbar)|=1$:
\begin{enumerate}
\item 
\label{item:prune_left}
$\tr= \tbar  \lhd \tr_l$ for $\tbar , \tr_l \in \brpt (S \amalg \{l \}) $, the left leaf of $\tbar$  and the root of $\tr_l$ labelled by $l$;
\item 
\label{item:prune_right}
$\tr= \tbar  \lhd \tr_r$ for $\tbar , \tr_r \in \brpt (S \amalg \{r \}) $, the right leaf of $\tbar$  and the root of $\tr_r$ labelled by $r$;
\item 
\label{item:prune_left_right}
$\tr= (\tbar  \lhd \tr_l) \lhd \tr_r = (\tbar \lhd \tr_r) \lhd \tr_l$ where $\tbar \in \brpt (S \amalg \{l, r \})$ with left leaf labelled by $l$ and right by $r$ and $\tr_l, \tr_r$ are as above;
\end{enumerate}
where, in each case $|v(\tr_l)|, |v(\tr_r)| \geq 1$ and the trees $\tr_l$ and $\tr_r$ are disjoint.
\end{enumerate}
\end{prop}

\begin{rem}
Case (\ref{item:prune}) of 
Proposition \ref{prop:pruning} can be illustrated schematically as follows, omitting all labels other than those occurring in the $\lhd$-product, the three possibilities are:

\noindent
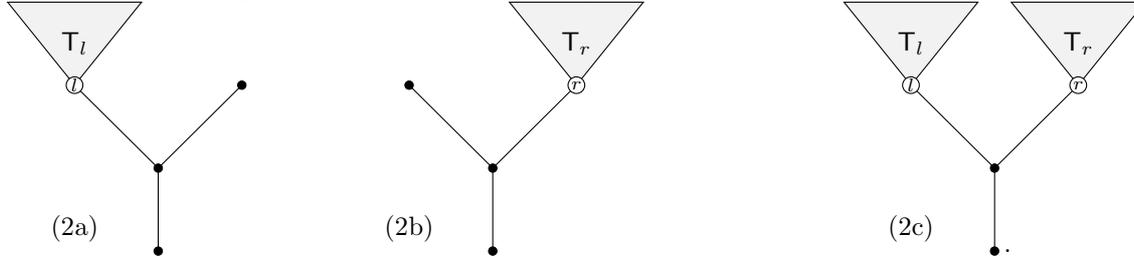
\begin{tikzpicture}[scale=1.1]
\draw [fill] (0,0)  circle [radius=.05];
\draw (0,0) -- (0,1) -- (1,2);
\draw (0,1) -- (-1,2); 
\draw [fill=gray!10] (-1,2) -- (-.2,3) -- (-1.8,3) -- (-1,2);
\draw [fill] (0,1)  circle [radius=.05];
\draw [fill] (1,2)  circle [radius=.05];
\draw [fill=white] (-1,2)  circle [radius=.1];
\node at (-1,2) {$\scriptstyle{l}$};
\node at (-1, 2.5) {$\tr_l$};
\node [above] at (-1,0) {(2a)};
\draw [fill] (4,0)  circle [radius=.05];
\draw (4,0) -- (4,1) -- (5,2);
\draw (4,1) -- (3,2); 
\draw [fill=gray!10] (5,2) -- (5.8,3) -- (4.2,3) -- (5,2);
\draw [fill] (4,1)  circle [radius=.05];
\draw [fill] (3,2)  circle [radius=.05];
\draw [fill=white] (5,2)  circle [radius=.1];
\node at (5,2) {$\scriptstyle{r}$};
\node at (5, 2.5) {$\tr_r$};
\node [above] at (3,0) {(2b)};
\draw [fill] (10,0)  circle [radius=.05];
\draw (10,0) -- (10,1) -- (11,2);
\draw (10,1) -- (9,2); 
\draw [fill=gray!10] (9,2) -- (9.8,3) -- (8.2,3) -- (9,2);
\draw [fill] (10,1)  circle [radius=.05];
\draw [fill=white] (9,2)  circle [radius=.1];
\node at (9,2) {$\scriptstyle{l}$};
\node at (9, 2.5) {$\tr_l$};
\node [above] at (9,0) {(2c)};
\draw [fill=gray!10] (11,2) --(11.8,3) -- (10.2, 3) -- (11,2);
\draw [fill=white] (11,2)  circle [radius=.1];
\node at (11,2) {$\scriptstyle{r}$};
\node at (11, 2.5) {$\tr_r$};
\node [right] at (10,0) {.};
\end{tikzpicture}
\end{rem}

\ 
\bigskip

The significance of the special pointed trees in the binary case (see Definition \ref{defn:arboriculture}) can be seen by the following dichotomy, which follows directly from Proposition \ref{prop:pruning}.

\begin{cor}
\label{cor:dichotomy_special_pointed}
Let $\tr \in \brpt (S)$ be a pointed tree with $|v(\tr)|>1$. Then precisely one of the following holds:
\begin{enumerate}
\item 
there exist disjoint, $S_+$-labelled trees $\tr'$, $\tr''$ with $|v(\tr')|, |v(\tr'')| \geq 1$, $\rt(\tr'')= +$ and such that $\tr= \tr' \lhd \tr''$;
\item 
$\tr$ is special pointed. 
\end{enumerate}
If $\tr$ is special pointed, then there exist $S_+$-labelled trees $\tbar$, $\tr''$ with $|v(\tbar)|=1$ and $\tbar$ special pointed, $ |v(\tr'')| \geq 1$ with $\rt(\tr'')= +$ and $\tr''$ disjoint, such that $\tr= \tbar \lhd \tr''$.
\end{cor}

\begin{exam}
The dichotomy of Corollary \ref{cor:dichotomy_special_pointed} is illustrated schematically as follows for pointed trees in case (\ref{item:prune_left}) of Proposition \ref{prop:pruning}, using the case labellings given by Corollary \ref{cor:dichotomy_special_pointed}:

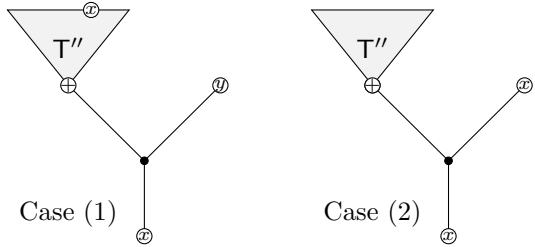
\noindent
\begin{tikzpicture}[scale=1]
\draw [fill] (0,0)  circle [radius=.05];
\draw (0,0) -- (0,1) -- (1,2);
\draw (0,1) -- (-1,2); 
\draw [fill=gray!10] (-1,2) -- (-.2,3) -- (-1.8,3) -- (-1,2);
\draw [fill] (0,1)  circle [radius=.05];
\draw [fill= white] (1,2)  circle [radius=.1];
\draw [fill= white] (3,2)  circle [radius=.1];
\draw [fill= white] (0,0)  circle [radius=.1];
\draw [fill=white] (-1,2)  circle [radius=.1];
\draw [fill=white] (-0.7,3)  circle [radius=.1];
\node at (-1,2){$\scriptstyle{+}$};
\node at (1,2) {$\scriptstyle{y}$};
\node at (0,0) {$\scriptstyle{x}$};
\node at (-0.7,3) {$\scriptstyle{x}$}; 
\node at (-1, 2.5) {$\tr''$};
\node [above] at ((-1,0) {Case (1)};
\draw (4,0) -- (4,1) -- (5,2);
\draw (4,1) -- (3,2); 
\draw [fill=gray!10] (3,2) -- (3.8,3) -- (2.2,3) -- (3,2);
\draw [fill] (4,1)  circle [radius=.05];
\draw [fill=white] (3,2)  circle [radius=.1];
\draw [fill=white] (5,2)  circle [radius=.1];
\draw [fill = white] (4,0)  circle [radius=.1];
\node at (4,0) {$\scriptstyle{x}$};
\node at (5,2) {$\scriptstyle{x}$};
\node at (3,2) {$\scriptstyle{+}$};
\node at (3, 2.5) {$\tr''$};
\node [above] at (3,0) {Case (2)};
\end{tikzpicture}

\noindent
where $x \neq y \in S$ in the first case and $x$ occurs once as a leaf label of $\tr''$; in the second, $x$ is not a leaf label of $\tr''$. Thus the second tree represents a special pointed tree: no non-trivial pruning can separate the root from the leaf labelled by $x$.
\end{exam}

The following illustrates the special derivations of Section \ref{subsect:special} in the case of a free binary operad $\fob$:

\begin{prop}
\label{prop:bases_disjoint_pointed_special}
Let $(S, z)$ be a finite pointed set. Then 
\begin{enumerate}
\item
$\derpt (\fob (R[S,z])) \subset \der (\fob (R[S]))$ has sub-basis given by the set of pointed $S$-labelled rooted binary planar $\bin_3$-trees with root $z$; 
\item 
$\derspec (\fob (R[S,z]))$ has a basis given by the set of special pointed $S$-labelled rooted binary planar $\bin_3$-trees with root $z$.
\end{enumerate}
\end{prop} 

\begin{proof}
The first statement is Proposition \ref{prop:pointed_der_fo}, restricted to the binary case. 

The special pointed $S$-labelled rooted binary planar trees with root $z$ form a subset of the  pointed $S$-labelled rooted binary planar $\bin_3$-trees with root $z$. By definition, these sets coincide for trees with one internal vertex; they are empty if $|S|=1$. 

Hence suppose that $|S|>1$. Proposition \ref{prop:pruning} implies that an $S$-labelled rooted binary planar tree $\tr$ with root $z$ and with  $|v(\tr)|> 1$ is special pointed if and only 
\[
\tr = 
\tbar \lhd \tr',
\]
where $|v(\tbar)|=1$ with $\tbar$ special pointed with root $z$ and where $\tr'$ is $S\backslash \{z\}$-labelled. The result follows from the definition of the action of derivations on pointed derivations, which is given by the preLie structure.
\end{proof}

\subsection{On $\derpl(\opd(-))$}

The subalgebra $\derpl (\opd (V)) \subset \der^+(\opd (V))$ only sees the suboperad of $\opd$ that is generated by $\opd (\mathbf{2})$. Hence one can only reasonably expect a statement as in Theorem \ref{thm:derpl} below for the case that $\opd$ is binary.

\begin{thm}
\label{thm:derpl}
Let $\opd$ be a binary operad. Then, for $V \in \ob \spmon$ such that $\rk_R (V) \neq 1$, the natural inclusion:
\[
\derpl (\opd (V)) 
\hookrightarrow 
\der^+(\opd (V))
\]
is an isomorphism.
\end{thm}

\begin{proof}
The case of rank $0$ (i.e., $V=0$) is clear, hence we may assume that $\rk_R (V) \geq 2$.

Using the naturality with respect to the operad $\opd$, together with the surjectivity property given in Proposition \ref{prop:naturality_derprelie_derlie}, one reduces to the case where $\opd$ is a free binary operad $\fob$. 

To establish surjectivity, using the free $R$-module functor, one can restrict to $\finj$, hence suppose that $V = R[S]$ with $|S|\geq 2$, and use the basis for $\der^+ (\fob (R[S]))$ give by Proposition  \ref{prop:deriv_fo}; as in Proposition \ref{prop:deriv_fo_pos},  positive derivations corresponds to restricting  to planar binary $\bin_3$-trees $\mathsf{T}$ with $|v(\mathsf{T})|\geq 1$.

It suffices to prove that any $\mathsf{T} \in \brpt (S)$ with $|v(\mathsf{T})|\geq 1$  is in $\derpl (\fob (R[S]))$. This is proved by induction on $|v(\mathsf{T})|$, starting from the case $|v(\mathsf{T})|=1$, which is clear. 

For the inductive step, consider an $S$-labelled $\bin_3$-tree $\tr$ with  $|v(\tr)|>1$. This can be pruned as in Proposition \ref{prop:pruning}; by hypothesis, we are in case (\ref{item:prune}) of the Proposition. The argument below adopts the numbering of the Proposition.  

We first reduce to the  cases (\ref{item:prune_left}) or (\ref{item:prune_right}) as follows. Suppose that we are in case (\ref{item:prune_left_right}), in particular  that the right branch of $\overline{\tr}$ has a tree  $\tr_r$ attached with $|v(\tr_r)|\geq 1$. Prune $\tr_r$ from  $\tr$ to give $S$-labelled trees $\tr'$ and  $\tr_r$, where the labellings are inherited from $\tr$ together with an arbitrary choice of label from $S$ at the cut, corresponding to the rightmost leaf of $\tr'$ and the root of $\tr_r$. Note that this label may also occur on other leafs of $\tr'$.

By construction, the preLie product $\tr' \lhd \tr_r$ is equal to $\tr + \sum_{i\in \mathcal{I}} \tr_i$, where $\mathcal{I}$ indexes a finite set of trees with the same number of leaves as $\tr$ and which fall into case (\ref{item:prune_left}), the terms indexed by $\mathcal{I}$ corresponding to the possible graftings of $\tr_r$  other than to the rightmost branch of $\tr'$.

By the inductive hypothesis, both the $S$-labelled trees $\tr'$ and $\tr_r$ lie in 
 $\derpl (\fob (R[S]))$, since  $|v(\tr')|, |v(\tr_r)| < |v(\tr)|$,  hence so does $\tr' \lhd \tr_r$. This reduces  to the case (\ref{item:prune_left}).
 
So suppose that $\tr$ is in case (\ref{item:prune_left}) (the case (\ref{item:prune_right}) is treated by the same argument, {\em mutatis mutandis})  and consider the associated pruning, which gives the trees $\overline{\tr}$ and  $\tr_l$, where $\overline{\tr}$ has two leaves. The right hand leaf of $\overline{\tr}$ is already labelled; since $|S |\geq 2$ by hypothesis, the left hand leaf can be labelled by a distinct element of $S$, which is used to label the root of $\tr_l$, so that both $\overline{\tr}$ and  $\tr_l$ are $S$-labelled.

Then, by construction, $\tr = \overline{ \tr} \lhd \tr_l$. As before, the inductive hypothesis ensures that both $\overline{\tr}$ and  $\tr_l$ are in  $\derpl (\fob (R[S]))$, which completes the proof of the inductive step.
\end{proof}

\begin{rem}
The restriction on the rank of  $V$ is sometimes necessary, as exhibited by the following:
\begin{enumerate}
\item 
If $\bin = \{*\}$, $\derpl (\fob (R) ) \subsetneq \der^+ (\fob (R))$. Namely, $\hom _R (R, (\fob)_2 (R))$ has a single generator $\mathsf{X}$ and $\mathsf{X} \lhd \mathsf{X} $ gives the sum of the two basis elements given by rooted planar binary trees with three leaves. 
\item 
For $\opd= \lieopd$, $\der^+ (\lieopd (R))=0$ (due to the anti-symmetry), so that the restriction on the rank of $V$ can be removed in this case.
\item
For $\opd= \assopd$, if $2$ is invertible in $R$, then  $\derpl (\assopd (R)) = \der^+ (\assopd (R))$. 
\end{enumerate}
\end{rem}

\begin{rem}
\label{rem:pl_versus_lie}
\ 
\begin{enumerate}
\item
Theorem \ref{thm:derpl} should be contrasted  with the inclusion 
$
\derlie (\opd (V)) \hookrightarrow \der^+ (\opd (V))
$
that is analysed in the following Sections. This   is  usually far from being an equality; however,  the case of the commutative operad $\comopd$ over $\rat$ shows that this is not always true (see Proposition \ref{prop:case_comopd}). 
\item 
In general,   $\derlie (\opd (-)) \hookrightarrow \derpl (\opd (-))$ is a proper inclusion, as opposed to the preLie case above. The above proof does not carry over, due to the Lie bracket being defined by making $\lhd$ antisymmetric. This means that, when seeking to recover  $\tr' \lhd \tr''$ for instance, the Lie bracket $[\tr', \tr'']$ gives the additional, potentially non-trivial term $- \tr'' \lhd \tr'$. Therein lies all the difficulty.
\end{enumerate}
\end{rem}

\section{Derivations modulo $\derlie(\opd (-))$}
\label{sect:derlie}

The remainder of the paper focuses upon  $\derlie(\opd (-))$. This Section provides the groundwork: the main result, Proposition \ref{prop:reduce_to_derpt}, gives  a weak normal form for derivations modulo $\derlie(\opd (-))$. This is the key step in proving Theorem \ref{thm:ses_up_to_torsion}: it allows reduction in Section \ref{sect:kertrace} to working with {\em pointed} derivations, for which the behaviour of the generalized divergence is much easier to understand, due to Proposition \ref{prop:addelta_point_quotient}. 

Proposition \ref{prop:reduce_to_derpt} involves working with derivations modulo $\derlie(\opd (-))$. The starting point for the arguments is to exhibit sufficiently many building blocks that lie in $\derlie (\opd (-))$; this is achieved in Section \ref{subsect:disjoint} (up to torsion) by using disjoint derivations. These then allow special pointed derivations to be considered in Section \ref{subsect:special_pointed}.

With these tools in hand, the remainder of the Section shows how to treat pointed derivations, leading to the weak normal form alluded to above. 

Throughout, $\opd$ is a binary operad. Moreover, the results involve working up to torsion; since the arguments reduce to working with functors on $\finj$ or $\fipt$, the ring $R$ is required to satisfy the following hypothesis, which allows  Proposition \ref{prop:equality_kappa} to be applied (this will be used without further mention).

\begin{hyp}
\label{hyp:stably_free}
All finitely-generated stably-free $R$-modules are free.
\end{hyp}

The results of this and Section \ref{sect:kertrace} hold for an arbitrary binary operad $\opd$. 
For such an operad, there exists a set $\bin$ and a surjection from the associated free binary operad:
\[
\fob \twoheadrightarrow \opd.
\]
This allows many proofs to be reduced to the case of a free binary operad. 

\begin{rem}
In the case $\opd = \fob$, restricting to $\finj$ along $R[-] : \finj \rightarrow \spmon$  allows arguments to be given using $S$-labelled rooted binary planar $\bin_3$-trees, by Proposition  \ref{prop:deriv_fo}. 
\end{rem}

Some of the proofs involve adopting a new basepoint for an object of $\sppt$. To avoid potential confusion, the following is used:

\begin{nota}
\label{nota:R_R'}
When an object of $\spmon$ has two potential choices of basepoint, the corresponding factors are distinguished via
$
R \oplus R' \oplus V, 
$ 
where $R'$ denotes a free $R$-module of rank one, so that the two associated pointed objects of $\sppt$ are $R \oplus (R' \oplus V)$ and $R' \oplus (R \oplus V)$, using the convention that a pointed object is denoted by $R \oplus W$.
\end{nota}

\subsection{Relation with $\derkerphi$}
\label{subsect:disjoint}

Disjoint derivations were introduced in Definition \ref{defn:derkerphi}. These restrict to positive degree as:
\[
\derkerphi^+ (\opd (-)) := \derkerphi (\opd (-)) \cap  \der^+ (\opd (-)).
\]
This Section shows that, up to torsion, these lie in $\derlie (\opd (-))$.

\begin{prop}
\label{prop:derkerphi_derlie}
Considered as functors from $\spmon$ to $\Rmod$, the image of the composite natural transformation
\[
\derkerphi^+ (\opd (-))
\subset
  \der^+ (\opd (-))
\twoheadrightarrow 
  \der^+ (\opd (-))/ \derlie (\opd (-) )
\]
is $1$-torsion.
\end{prop}

\begin{proof}
From the definition of $\derkerphi^+ (\opd (-))$, it suffices to show the result after restriction to $\opd (\vbar ) \subset \derkerphi^+ (\opd (V))$,  where $V = R \oplus \vbar$ in $\sppt$.

One reduces to the universal example $\opd = \fob$, with $V= R[S]$ and  $S$ pointed by $z$, so that $\vbar = R[S\backslash \{z\}]$. By Proposition  \ref{prop:fo_disj}, 
$\fob (\vbar )\subset  \derkerphi^+ (\fob (V))$ has basis given by the disjoint $S$-labelled binary rooted planar $\bin_3$-trees with root $z$. 

Consider such a tree $\mathsf{T}$ with $|v(\tr)|\geq 1$ (this condition corresponds to positivity, by Proposition \ref{prop:deriv_fo_pos}). By the definition of $1$-torsion (cf. Section \ref{sect:torsion}), it suffices to prove that, after enlarging $S$  to $S_+ := S\amalg \{+\}$, $\tr$ lies in $\derlie (\fob (R[S_+]))$.

The proof is by increasing induction on $|v(\tr)|$. In the case $|v(\tr)|=1$, there is nothing to prove, since $\tr$ lies in $\derlie (\fob (R[S]))$, since the latter coincides with $\der^+ (\fob (R[S]))$ in degree one, by construction.

The inductive step uses pruning, as in Proposition \ref{prop:pruning}. The argument is presented for the case (\ref{item:prune_left_right}) of that Proposition; the other cases are treated by a similar argument. After enlarging $S$ to $S_+$, one can write in $\der^+(\fob (R[S_+]))$: 
\[
\tr= (\overline{ \tr} \lhd  \tr_l)\lhd \tr_r, 
\]
where $|v(\tbar)|=1$ and $\tbar$ has root $z$, left leaf labelled $z$ and right leaf labelled $+$; $\rt (\tr_l) = z$ and $\rt(\tr_r)=+$. By construction, both $\tr_l$ and $\tr_r$ are disjoint; $\tr_l$ is $S$-labelled and $\tr_r$ is $S_+\backslash \{z\}$-labelled. 

 As above, the tree $\tbar$ lies in $\derlie (\fob (R[S_+]))$ and the inductive hypothesis  implies that  the trees  $\tr_l$ and $\tr_r$ lie in $\derlie (\fob  (R[S_+]))$ (for $\tr_r$, the inductive hypothesis is applied with respect to $S_+\backslash \{z\} \subset S_+$).

Now,  $\tr =  [[\overline{ \tr}, \tr_l], \tr_r]$, by the disjointness properties; this completes the inductive step in this case.
\end{proof}

\begin{rem}
By Proposition \ref{prop:derkerphi}, the disjoint derivations lie in the kernel of the generalized divergence $\gendiv^\opd_{-}$. Hence the above  is a {\em necessary} step in proving  Theorem \ref{thm:ses_up_to_torsion} of  Section \ref{sect:kertrace}.
\end{rem}

\subsection{The case of  special pointed derivations}
\label{subsect:special_pointed}

Special pointed derivations were introduced in Definition \ref{defn:derspec}. For  $\opd= \fob$, these are exceptional in that they cannot be decomposed as the $\lhd$-product of two disjoint derivations (cf. Corollary \ref{cor:dichotomy_special_pointed}). It is thus essential to treat these  directly.

For $V \in \ob \spmon$, forgetting the basepoint provides a natural transformation
\[
\derspec (\opd (R \oplus V) ) 
\subset 
\derpt^+ (\opd (R\oplus V)) 
\rightarrow 
\der^+ (\opd (R \oplus V)).
\]

\begin{prop}
\label{prop:derspec_derlie}
The image of the composite natural transformation of functors on $ \spmon$
\[
\derspec (\opd (R\oplus -) )
\rightarrow 
\der^+ (\opd (R \oplus -))
\twoheadrightarrow 
\der^+ (\opd (R \oplus -))/\derlie (\opd (R\oplus -))
\]
is $1$-torsion.
\end{prop}

The proof is based upon the following Lemma for a free binary operad $\fob$, using the basis given by Proposition \ref{prop:bases_disjoint_pointed_special}:

\begin{lem}
\label{lem:special-pointed}
Let $(S,z)$ be a finite pointed set and $\tr \in \derspec (\fob (R[S,z]))$ represent a special pointed derivation with $\rt (\tr) = z$ and $|v(\tr)|>1$.
 Then, in $\der (\fob(R[S_+]))$, 
\[
\tr = \tbar \lhd \tr' = [\tbar , \tr'], 
\]
 where 
$|v(\tbar)|=1$ and  $\tr'$ is a disjoint $S_+\backslash \{z \}$-labelled tree with $\rt(\tr') =+$.
\end{lem}

\begin{proof}
This follows from Proposition \ref{prop:pruning}.
\end{proof}

\begin{proof}[Proof of Proposition \ref{prop:derspec_derlie}]
One reduces to the universal example $\opd = \fob$, taking $R \oplus V$ to be $R[S,z]$, where $(S,z)$ is a finite pointed set.

By Proposition \ref{prop:bases_disjoint_pointed_special}, using the notation of Lemma \ref{lem:special-pointed}, it suffices to show that, after enlarging $S$ to $S_+$,  $\tr$ lies in $\derlie (\fob(R[S_+]))$. 
 Lemma \ref{lem:special-pointed} gives $\tr =  [\tbar , \tr']$ in $\der^+ (\fob (R[S_+]))$, where $\tbar$ lies in $\derlie (\fob(R[S_+]))$, since $|v(\tbar)|=1$, and $\tr'$ lies in  $\derlie (\fob (R[S_+]))$ by Proposition \ref{prop:derkerphi_derlie}, since it is disjoint $S_+\backslash \{z \}$-labelled. The result follows.
\end{proof}

\subsection{Exchanging  basepoints}
\label{subsect:swapping}

For $V = R \oplus \vbar \in \sppt$, one has the canonical inclusion $\derpt^+ (\opd (R\oplus \vbar)) \subset \der^+ (\opd (V))$ given by forgetting the splitting and the basepoint. Moreover,  using  Notation \ref{nota:R_R'}, the (non-pointed) embedding $V \subset R'\oplus V $ induces   $\der (\opd (V) ) \hookrightarrow \der (\opd (R'\oplus V))$.
 Hence, composing with the canonical surjection gives:
\begin{eqnarray*}
\label{eqn:comp_change_basept}
\alpha_{R,\vbar} : \derpt^+ (\opd (R\oplus \vbar))
\rightarrow 
\der  (\opd (R' \oplus R \oplus \vbar)) 
/
\derlie  (\opd (R' \oplus R \oplus \vbar)).
\end{eqnarray*}
Similarly one has $\alpha_{R',\vbar}$, by switching the rôle of the basepoints.

\begin{defn}
\label{defn:inclusion_up_to_n_torsion}
For $n \in \nat$ and a $\spmon$-module $G$ with subobjects $F_1, F_2 \subset G$, $F_1$ is contained in $F_2$ up to $n$-torsion if, $\forall V \in \ob \spmon$, $\forall x_1 \in F_1 (V)$, $\exists x_2 \in F_2 (V)$ such that $x_1 - x_2 \in G (V)$ is $n$-torsion.
\end{defn}

\begin{prop}
\label{prop:change_basepoint}
For $\vbar \in \ob \spmon$  
\[
\xymatrix{
&
\derpt^+ (\opd (R'\oplus \vbar))
\ar[d]^{\alpha_{R', \vbar}}
\\
\derpt^+ (\opd (R\oplus \vbar))
\ar[r]_(.3){\alpha_{R, \vbar}}
&
\der^+  (\opd (R' \oplus R \oplus \vbar)) 
/
\derlie (\opd (R' \oplus R \oplus \vbar)), 
}
\]
the image of $\alpha_{R, \vbar}$ is contained in the image of $\alpha_{R', \vbar}$ up to $1$-torsion.
\end{prop}

\begin{rem}
The occurrences of $\derpt^+$ in this Proposition are defined with respect to the different basepoints (corresponding to $R$ and $R'$). The result is symmetric with respect to these. 
\end{rem}

\begin{proof}[Proof of Proposition \ref{prop:change_basepoint}]
The proof is given for the universal example  $\opd = \fob$ with $V= R[S,z]$; this allows one to work with the basis of $\der (\fob  (R[S]))$  given by $\brpt(S)$. The passage to $R' \oplus V$ corresponds to the inclusion $S \subset S_+$.

Let $\tr$ be an $S$-labelled tree that represents a generator of $\derpt^+ (\fob  (R[S,z]))$. If $|v(\tr)|=1$, then it maps to zero in $\der^+  (\fob (R[S_+])) 
/
\derlie  (\fob (R[S_+]))$, since the latter is zero in degree one, so the result is clear in this case.

Suppose now that $|v(\tr)|>1$. If $\tr$ is special, the result holds by Proposition \ref{prop:derspec_derlie}. Otherwise, we proceed by pruning (cf. Proposition \ref{prop:pruning}), similarly to the proof of Lemma \ref{lem:special-pointed}. There exist disjoint, $S_+$-labelled trees $\tr'$, $\tr''$ with $|v(\tr')|, |v(\tr'')|\geq 1$ and $\rt (\tr') = z$, $\rt (\tr'') = +$, such that 
$
\tr = \tr' \lhd \tr''.
$ 
Here $z$ is a leaf of $\tr''$ and not one of $\tr'$, by the disjoint hypothesis.
 Thus, $\tr = [\tr', \tr''] + \tr_+$, where $\tr_+ := \tr'' \lhd \tr'$. By construction, $\tr_+$ represents a generator of $\derpt^+ (\fob (R[S_+\backslash\{z \}, +]))$. 

Since $\tr'$, $\tr''$ are disjoint, Proposition \ref{prop:derkerphi_derlie} applies: after further enlargement to $S \amalg \{ +, *\}$, both $\tr'$ and $\tr''$ lie in  $\derlie (\fob (S \amalg \{ +, *\}))$, hence so does their commutator. 

Thus, by naturality of the Lie structure with respect to $S_+ \subset S \amalg \{ +, *\}$, the term $[\tr', \tr'']$ (considered here in $\der^+ (\opd (R[S_+]))$) is $1$-torsion modulo $\derlie (\fob (-))$. This gives the congruence $ \tr \equiv \tr_+$ in $\der^+  (\fob (R[S_+])) 
/
\derlie (\fob (R[S_+])$ up to $1$-torsion, as required.
\end{proof}

\begin{rem}
The  trees $\tr \equiv \tr_+$ arising in the proof of Proposition \ref{prop:change_basepoint} can be illustrated schematically by:

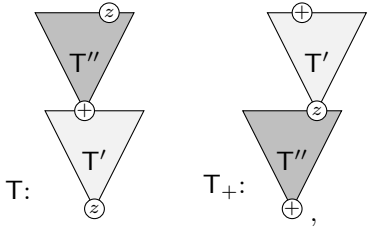
\qquad
\qquad
\begin{tikzpicture}[scale = 1.3]
\node [above right] at (-1,0) {$\tr$:};
\draw [fill=gray!10] (0,0) -- (.5,1) -- (-.5,1) -- (0,0);
\draw [fill=lightgray] (-.1,1) -- (.4,2) -- (-.6,2) -- (-0.1,1);
\draw [fill=white] (0,0) circle [radius = 0.1];
\node at (0,0) {$\scriptstyle{z}$};
\draw [fill=white] (-.1,1) circle [radius = 0.1];
\node at (-.1,1) {$\scriptstyle{+}$};
\draw [fill=white] (0.15,2) circle [radius = 0.1];
\node at (0.15,2) {$\scriptstyle{z}$};
\node at (0, .5) {$\tr'$};
\node at (-.1,1.5) {$\tr''$};
\node [above right] at (1, 0) {$\tr_+$:};
\draw [fill=lightgray] (2,0) -- (2.5,1) -- (1.5,1) -- (2,0);
\draw [fill=gray!10] (2.25,1) -- (2.75,2) -- (1.75,2) -- (2.25,1);
\draw [fill=white] (2,0) circle [radius = 0.1];
\node at (2,0) {$\scriptstyle{+}$};
\node [below right] at (2,0) {\  ,};
\draw [fill=white] (2.1,2) circle [radius = 0.1];
\node at (2.1,2) {$\scriptstyle{+}$};
\draw [fill=white] (2.25,1) circle [radius = 0.1];
\node at (2.25,1) {$\scriptstyle{z}$};
\node at (2, .5) {$\tr''$};
\node at (2.25,1.5) {$\tr'$};
\end{tikzpicture}

\noindent
in which $\tr'$ and $\tr''$ are disjoint. 
\end{rem}

\subsection{A weak normal form for derivations modulo $\derlie (\opd (-))$}
Here,  Proposition \ref{prop:change_basepoint} is generalized, replacing $\derpt^+ (\opd (R\oplus \vbar))$ by $\der^+ (\opd (V))$. This is at the expense of having to slightly weaken the torsion condition.

\begin{nota}
\label{nota:Pi}
For $R\oplus V \in \ob \sppt$, let   $\Pi_{R \oplus V}$ be the composite natural transformation of functors from $\sppt$ to $\Rmod$.
\[
\Pi_{R \oplus V} : 
\derpt^+ (\opd (R \oplus  V))
\rightarrow 
\der^+  (\opd (R \oplus V)) / \derlie (\opd (R \oplus V))
\]
given by forgetting the basepoint and passing to the quotient. 
\end{nota}

\begin{prop}
\label{prop:reduce_to_derpt}
For $V \in \ob \spmon$  
\[
\xymatrix{
&&
\derpt^+ (\opd (R \oplus V))
\ar[d]^{\Pi_{R \oplus V}}
\\
\der^+ (\opd (V))
\ar@{^(->}[r]
&
\der^+ (\opd (R \oplus V))
\ar@{->>}[r]
&
\der^+  (\opd (R \oplus V)) 
/
\derlie (\opd (R\oplus V)), 
}
\]
the image of $\der^+ (\opd (V))$ in $\der^+  (\opd (R\oplus V)) 
/
\derlie  (\opd (R \oplus V))$ via the horizontal composite is contained in the image of $\Pi_{R \oplus V}$ up to $2$-torsion.
\end{prop}

\begin{proof}
The proof is given for the universal example $\opd = \fob$. Take $V = R[S]$ and consider the split inclusion $V= R[S]\hookrightarrow R[S_+] \cong R \oplus V$ induced by $S\hookrightarrow S_+$.

Let $\tr \in \brpt(S)$ with $|v(\tr)|\geq 1$, representing a generator of $\der ^+ (\fob (R[S]))$ as in Proposition  \ref{prop:deriv_fo_pos}.  The  following  cases are treated directly, as indicated:
\begin{enumerate}
\item 
If $|v(\tr)|=1$, then this lies in $\derlie (\fob (R[S]))$;
\item 
if $\tr$ is disjoint, by Proposition \ref{prop:derkerphi_derlie}, up to $1$-torsion it lies in $\derlie (\fob (R[S]))$; 
\item 
if $\tr$ is special pointed, likewise, by appealing to Proposition \ref{prop:derspec_derlie}. 
\end{enumerate}

Otherwise, one proceeds using Proposition \ref{prop:pruning}. We treat the case (\ref{item:prune_left_right}); the other cases are treated by a similar argument. Passing to the indexing set $\tilde{S}:= S \amalg \{l ,r \}$, one has 
\[
\tr = (\tbar \lhd  \tr_l) \lhd \tr_r,
\]
where $|v(\tbar)|=1$, with leaves labelled by $l$ and $r$; $\rt (\tbar) = \rt (\tr) = z \in S$; $\rt (\tr_l) = l$ and $\rt (\tr_r) =r$. Now  $\tr_l$ is disjoint with labelling set   $S \amalg \{ l\}$ and $\tr_r$ is disjoint with labelling set $S \amalg \{ r\}$; thus, by Proposition \ref{prop:derkerphi_derlie}, their images in $\der (\fob (R[\tilde{S}]))$ both lie in $\derlie (\fob (R[\tilde{S}]))$; also, $\tbar$ lies in $\derlie (\fob (R[\tilde{S}]))$.

Now:
\begin{eqnarray*}
\tr&=&
[(\tbar \lhd   \tr_l),  \tr_r] +   \tr_r \lhd (\tbar \lhd   \tr_l)
\\
&=& 
[ [\tbar,  \tr_l],  \tr_r] + [\tr_l \lhd \tbar,   \tr_r] + \tr_r \lhd (\tbar \lhd   \tr_l)
\\
&=&
[ [\tbar,  \tr_l],  \tr_r] + (\tr_l \lhd \tbar)\lhd   \tr_r + \tr_r \lhd (\tbar \lhd   \tr_l),
\end{eqnarray*}
where the first and second equalities since $[\  ,\ ]$ is the Lie bracket associated to $\lhd$; the third equality holds because $\tr_l \lhd \tbar$ is a sum of trees with root $l$ and $l$ is not a leaf of $\tr_r$. 

Here, $[ [\tbar,  \tr_l],  \tr_r] \in \derlie (\fob (R[\tilde{S}]))$, since each of the terms belong to $\derlie (\fob (R[\tilde{S}]))$.
 Moreover, by construction, $ (\tr_l \lhd \tbar)\lhd   \tr_r$ is a sum of pointed trees with root $l$ and  $\tr_r \lhd (\tbar \lhd   \tr_l)$ is a sum of pointed trees with root $r$, representing elements of  $\derpt ^+ (\fob ( R[S\amalg \{l\}, l]))$ and $\derpt ^+ (\fob ( R[S\amalg\{ r\}, r]))$ respectively. 

To conclude, one uses 
 Proposition \ref{prop:change_basepoint} to change the pointed trees with root $l$ to pointed trees with root $r$. 
Explicitly: there exists $X \in \derpt^+(\fob(R[S\amalg \{r \},r]))$  such that the element $(\tr_l \lhd \tbar) \lhd \tr_r - X $ is $1$-torsion  in $\der^+ (\fob (R[\tilde{S}])) / \derlie (\fob (R[\tilde{S}]))$. 

This establishes that $\tr_r \lhd (\tbar \lhd   \tr_l) + X $ lies in  $\derpt^+ (\fob (R[S \amalg \{r \}, r]))$ and its  image under $\Pi_{R \oplus R[S]}$  is equivalent to $\tr$  up to $2$-torsion, as required.
 \end{proof}

 \section{Relating $\derlie (\opd (-))$ to  the generalized divergence}
 \label{sect:kertrace}
 
Throughout, $\opd$ is a binary operad and the ring $R$ satisfies Hypothesis \ref{hyp:stably_free}. 
This Section is the culmination of the previous work, obtaining information on the structure of $\derlie (\opd (-))$, considered as a functor on $\spmon$. This is based upon the strategy outlined in Section \ref{subsect:howto}, using the generalized divergence $\gendiv^\opd_{-}$ to analyse the inclusion $\derlie (\opd (-)) \subset \der^+(\opd (-))$. 

The main result establishes  that, up to torsion, $\derlie (\opd (-))$ is determined by the generalized divergence $\gendiv^\opd_{-}$:

\begin{thm}
\label{thm:ses_up_to_torsion}
Let $\opd$ be a binary operad and suppose that $R$ satisfies Hypothesis \ref{hyp:stably_free}.
The inclusion $\derlie (\opd (-)) \subset 
 \der^+ (\opd (-))$ and the generalized divergence $\gendiv^\opd_- : \der^+ (\opd (-)) 
\longrightarrow |\derpt^+ (\opd (R \oplus -))|$ induce a  natural sequence of functors from $\spmon $ to $\Rmod$:
\[
\derlie (\opd (-) )
\rightarrow 
\der^+ (\opd (-))
  \rightarrow 
|\derpt^+ (\opd (R \oplus -))|/\imderlie (-)
\]
that is  short exact up to torsion. 

More precisely, 
\begin{enumerate}
\item 
$\derlie (\opd (-) )
\rightarrow 
\der^+ (\opd (-))$ is a natural monomorphism; 
\item 
the natural morphism $\der^+ (\opd (-))
  \rightarrow 
|\derpt^+ (\opd (R \oplus -))|/\imderlie (-)$ is $1$-surjective; 
\item 
the middle homology is $6$-torsion.
\end{enumerate}
\end{thm}

\begin{proof}
By construction, $\derlie (\opd (V) )
\hookrightarrow 
\der^+ (\opd (V))$ is a monomorphism. The $1$-surjectivity is given by Proposition \ref{prop:addelta_1-surjective} and the $6$-torsion statement follows from Proposition \ref{prop:ker_Gamma_torsion}.
\end{proof}

Thus it remains to prove Proposition \ref{prop:ker_Gamma_torsion}; the proof occupies most of the Section.
This uses the weak normal form result, Proposition \ref{prop:reduce_to_derpt}, to reduce to considering pointed derivations.  The key step then corresponds to understanding the kernel of $\gendiv^\opd_{V}$  restricted to pointed derivations; this reduces to the result given as Proposition \ref{prop:mod_commutators_precise} of Section \ref{subsect:alg_commutators}. 

\begin{rem}
 Theorem \ref{thm:ses_up_to_torsion} is illustrated by the cases of the Lie and associative operads in Section \ref{subsect:main_lie_ass}. In these cases, the torsion statement can be refined. For instance, for $\lieopd$, Proposition \ref{prop:ker_Gamma_lieopd} shows that the middle homology is $3$-torsion, rather than the $6$-torsion given by Theorem \ref{thm:ses_up_to_torsion}.
\end{rem}

\subsection{Dealing with algebra commutators}
\label{subsect:alg_commutators}

Consider $R \oplus V \in \ob \sppt$; by  Theorem \ref{thm:assoc_alg}, $\derpt ^+ (\opd (R \oplus V))$ has the structure of a unital associative algebra, so that one has the submodule of commutators 
$
[\derpt ^+ (\opd (R \oplus V)), \derpt ^+ (\opd (R \oplus V))] \subset \derpt ^+ (\opd (R\oplus V)).
$

In the following, $\Pi_{R \oplus V}$ is as in Notation \ref{nota:Pi}:

\begin{prop}
\label{prop:mod_commutators_precise}
For $R \oplus V \in \ob \sppt$,  the  $\Pi_{R \oplus V}$-image of 
$$[\derpt ^+ (\opd (R \oplus V)), \derpt ^+ (\opd (R \oplus V))]$$
 in $\der ^+ (\opd (R \oplus V)) / \derlie (\opd (R \oplus V)) $ is $2$-torsion.
\end{prop}

\begin{proof}
The proof is presented for the universal case  $\opd= \fob$, taking $V= R[S]$ and considering $R[S_+,+] = R \oplus V$. 

It suffices to work with (commutators of) basis elements of $ \derpt ^+ (\fob( R[S_+, +]))$; these  are represented by  pointed $S_+$-labelled trees with root labelled by $+$, as in Proposition  \ref{prop:pointed_der_fo}.  
 Consider two such trees $\mathsf{T}_1$, $\mathsf{T}_2$  and their commutator
$
[\mathsf{T}_1, \mathsf{T}_2] = \mathsf{T}_1 \lhd \mathsf{T}_2 - \mathsf{T}_2 \lhd \mathsf{T}_1.
$

If both $\mathsf{T}_1$ and $\mathsf{T}_2$ are special pointed then,  up to $1$-torsion, they both belong to $\derlie (\fob (R[S_+]) )$ by Lemma \ref{lem:special-pointed}, hence so does their commutator. 

Otherwise, without loss of generality, we may assume that  $\mathsf{T}_2$ is not special pointed and proceed as in the proof of Proposition \ref{prop:change_basepoint}, writing:
\[
\mathsf{T}_2 = \mathsf{T}'_2 \lhd \mathsf{T}''_2,
\]
where $|\mathsf{T}'_2|, |\mathsf{T}''_2|\geq 1$, $\rt (\mathsf{T}'_2) = \rt (\mathsf{T}_2)=+$, $\rt(\mathsf{T}''_2) = *$  and both $\mathsf{T}'_2$ and $\mathsf{T}''_2$ are disjoint $S_+ \amalg \{ *\}$-labelled trees (i.e., the new element $*$ has been used to label the root and the leaf created by pruning). 

From the construction, it is straightforward to verify the following (which correspond to the vanishing of the respective associators):
\begin{eqnarray*}
\mathsf{T}_1 \lhd (\mathsf{T}'_2 \lhd \mathsf{T}''_2) &=& (\mathsf{T}_1 \lhd \mathsf{T}'_2  ) \lhd \mathsf{T}''_2
\\
\mathsf{T}''_2 \lhd (\mathsf{T}_1 \lhd \mathsf{T}'_2  ) &=& (\mathsf{T}''_2 \lhd \mathsf{T}_1) \lhd \mathsf{T}'_2  
\\
\mathsf{T}'_2 \lhd (\mathsf{T}''_2 \lhd \mathsf{T}_1) &=& (\mathsf{T}'_2 \lhd \mathsf{T}''_2) \lhd \mathsf{T}_1.
\end{eqnarray*}
For instance, the first equality follows since $*$ does not label a leaf of $\mathsf{T}_1$.

This gives the equalities:
\begin{eqnarray*}
\mathsf{T}_1 \lhd \mathsf{T}_2 &=& (\mathsf{T}_1 \lhd \mathsf{T}'_2  ) \lhd \mathsf{T}''_2
\\
&=& [(\mathsf{T}_1 \lhd \mathsf{T}'_2 ) ,   \mathsf{T}''_2] + \mathsf{T}''_2\lhd (\mathsf{T}_1 \lhd \mathsf{T}'_2 )
\\
&=&
[(\mathsf{T}_1 \lhd \mathsf{T}'_2 ) ,   \mathsf{T}''_2] + (\mathsf{T}''_2\lhd \mathsf{T}_1 )\lhd \mathsf{T}'_2 
\\
&=&
[(\mathsf{T}_1 \lhd \mathsf{T}'_2 ) ,   \mathsf{T}''_2] + [(\mathsf{T}''_2\lhd \mathsf{T}_1 ), \mathsf{T}'_2] + \mathsf{T}'_2 \lhd (\mathsf{T}''_2\lhd \mathsf{T}_1 )
\\
&=&
[(\mathsf{T}_1 \lhd \mathsf{T}'_2 ) ,   \mathsf{T}''_2] + [(\mathsf{T}''_2\lhd \mathsf{T}_1 ), \mathsf{T}'_2] + \mathsf{T}_2\lhd \mathsf{T}_1.
\end{eqnarray*}

Again from the construction, the following $S \amalg \{ +, *\}$-labelled (sums of) trees are disjoint:  $\mathsf{T}_1 \lhd \mathsf{T}'_2$,   $\mathsf{T}''_2$,  $\mathsf{T}''_2\lhd \mathsf{T}_1$,  
$\mathsf{T}'_2$. Therefore, by Proposition \ref{prop:derkerphi_derlie}, they lie in  $
\derlie (\fob (R[S \amalg \{+, *\}]))
$ 
up to $1$-torsion, as do the respective Lie brackets.

This gives the congruence
\[
\mathsf{T}_1 \lhd \mathsf{T}_2 - \mathsf{T}_2 \lhd \mathsf{T}_1 \equiv 0, 
\]
modulo $\derlie (\fob (-))$. Here both of the terms are defined in $\derpt^+ (\fob (R[S_+,+]))$, but the argument above required enlargement of $S_+$ to $S_+ \amalg \{ *\}$ and gave a congruence up to $1$-torsion. This leads to the $2$-torsion in the statement.
\end{proof}

\subsection{On the kernel of $\trq$}

The following is clear and serves to define the natural transformation $\trq$, using the functor $\imderlie$  introduced in Definition \ref{defn:imderlie}. 

\begin{lem}
\label{lem:Gamma}
The natural transformation $\gendiv ^\opd _{-}$ induces a natural transformation of functors from $\spmon$ to $\Rmod$:
\[
\trq[-] : 
  \der^+ (\opd (-))/ \derlie (\opd (-) )
  \rightarrow 
  |\derpt^+ (\opd (R \oplus -))|/\imderlie (-).
\]
Moreover, $\trq$ is natural with respect to the binary operad $\opd$.
\end{lem}

The outstanding ingredient to the proof of Theorem \ref{thm:ses_up_to_torsion} is: 

\begin{prop}
\label{prop:ker_Gamma_torsion}
The kernel of the natural transformation 
\[
\trq[-] : 
  \der^+ (\opd (-))/ \derlie (\opd (-) )
  \rightarrow 
  |\derpt^+ (\opd (R \oplus -))|/\imderlie (-)
\]
of functors on $\spmon$ is $6$-torsion.
\end{prop}

This result is a consequence of the slightly stronger formulation given in Proposition \ref{prop:ker_Gamma_torsion_variant}. For this, recall from 
Corollary \ref{cor:image_derspec_imderlie} that there are natural inclusions (with respect to $V \in \ob \spmon$):
\[
\imderlie (V) \subseteq \imderliespec (V) \subseteq |\derpt^+(\opd ( R \oplus V))|.
\]
Hence, there is a natural surjection $  |\derpt^+ (\opd (R \oplus V))|/ \imderlie (V) \twoheadrightarrow  |\derpt^+ (\opd (R \oplus V))|/ \imderliespec (V)$, so that $\trq[V]$ induces the composite natural transformation of the following Proposition:

\begin{prop}
\label{prop:ker_Gamma_torsion_variant}
The kernel of the composite natural transformation 
\[
  \der^+ (\opd (-))/ \derlie (\opd (-) )
  \rightarrow 
  |\derpt^+ (\opd (R \oplus -))|/ \imderliespec (-)
\]
of functors on $\spmon$ is $6$-torsion.

If the functor  $|\derpt^+ (\opd (R \oplus -))|/ \imderliespec (-)$ on $\spmon$ is torsion-free, then the kernel is $4$-torsion. 
\end{prop}

\begin{proof}
Consider an element of the kernel represented by an element  $x \in \der^+(\opd (V))$; the kernel condition is equivalent to  $\gendiv^\opd_V (x) \in \imderliespec (V)$. 

The first step is to use Proposition \ref{prop:reduce_to_derpt} to pass to a pointed derivation. 
The argument uses the commutative diagram:
\\
\noindent
\scalebox{0.85}{
$
\xymatrix{
&
\derpt^+  (\opd (R \oplus V) )
\ar[r]
\ar[d]_{\Pi _{R\oplus V}} 
&
\derpt^+  (\opd (R \oplus \vbar_3) )
\ar[d]^{\Pi _{R\oplus \vbar_3}}
\\
\der^+ (\opd (V))/\derlie(\opd  (V))
\ar[r]
\ar[d]
&
\der^+ (\opd (V_1))/\derlie (\opd(V_1))
\ar[d]
\ar[r]
&
 \der^+ (\opd (V_3))/\derlie (\opd(V_3))
\ar[d]
\\
|\derpt^+ (\opd (R \oplus V))|/ \imderliespec (V)
\ar[r]
&
 |\derpt^+ (\opd (R \oplus V_1))|/ \imderliespec (V_1)
\ar[r]
&
|\derpt^+ (\opd (R \oplus V_3))|/ \imderliespec (V_3),
}
$
}
\\

\medskip
\noindent
in which the horizontal morphisms are induced by the inclusions $V \hookrightarrow V_1 = R \oplus V \hookrightarrow V_3 = R \oplus \vbar_3 = R^{\oplus 3} \oplus V$ and the lower vertical maps are given by $\gendiv^\opd_{-}$ composed with the quotient modulo $\imderliespec (-)$. 

By hypothesis, $x$ represents an element of $\der^+ (\opd (V))/\derlie(\opd  (V))$ that lies in the kernel of the left hand vertical map. Proposition \ref{prop:reduce_to_derpt} yields an element $x_+ \in \derpt^+  (\opd (R \oplus V) )$ such that its image $x_3$ in $ \der^+ (\opd (V_3))/\derlie (\opd(V_3))$ coincides with the image of $x$. By construction, $x_3$ is the image of a pointed derivation $ \tilde{x}_3 \in \derpt^+  (\opd (R \oplus \vbar_3) )$. Considering $\tilde{x}_3$ as an element of $ \der^+ (\opd (V_3))$, commutativity of the diagram implies that $\gendiv^\opd _{V_3}  (\tilde{x}_3) \in \imderliespec (V_3)$ is the image of an element $\alpha \in \derspec (R \oplus V_3)$, say. 

Thus we replace $x$ by $\tilde{x}_3$, which arises from a pointed derivation in $\derpt^+  (\opd (R \oplus \vbar_3) )$, by construction. This allows Proposition \ref{prop:addelta_point_quotient} to be applied, which provides the commutative diagram:
\[
\xymatrix{
\derpt^+ (\opd (R\oplus \vbar_3))
\ar[r]
\ar@{^(->}[d] 
&
\derpt^+ (\opd (R\oplus V_3))
\ar@{->>}[d]
\\
\der^+(\opd (V_3))
\ar[r]_(.4){\gendiv^\opd_{V_3}}
&
|\derpt^+ (\opd (R \oplus V_3))|,
}
\]
where the top horizontal morphism is induced by $\vbar _3 \subset V_3$, the left hand vertical map is the canonical inclusion, and the right hand vertical map is the quotient modulo commutators.

Let $\tilde{x}_4 \in \derpt^+ (\opd (R\oplus V_3))$ denote the image of $\tilde{x}_3$ (and also its image in $\der^+ (\opd (R\oplus V_3))$ after forgetting the basepoint). 
 Consider the element $y_4 := \tilde{x}_4 - \alpha \in \derpt^+ (\opd (R \oplus V_3))$ (forgetting  that $\alpha$ arose from a {\em special} pointed derivation). By construction, $y_4$ lies in the kernel of the quotient map: 
\[
\derpt^+ (\opd (R \oplus V_3)) 
\twoheadrightarrow 
|\derpt^+ (\opd (R \oplus V_3)) |. 
\]
Proposition \ref{prop:mod_commutators_precise} then implies that, after passing to $V_6:= R \oplus V_3 \oplus R^{\oplus 2}$, $y_4$ lies in $\derlie (\opd (V_6))$. 

This shows that $y_4$ lies in $\derlie (\opd (V_6))$; it remains to deduce the analogous conclusion for $\tilde{x}_4$. Since $\tilde{x}_4 = y_4 + \alpha$,  it suffices to show that the image of $\alpha$ in $\der^+ (\opd (V_6))$ lies in $\derlie (\opd (V_6))$; this follows from Proposition \ref{prop:derspec_derlie}.

For the second statement, under the hypothesis that the functor $$V \mapsto |\derpt^+ (\opd (R \oplus V))|/ \imderliespec (V)$$
 is torsion-free, one checks from the commutative diagram leading to $x_3$ that  the above argument can be refined by starting from $x_+ \in \derpt^+ (\opd (R\oplus V))$ rather than from $\tilde{x}_3$. The details are left to the reader.  
\end{proof}

The proof of Proposition \ref{prop:ker_Gamma_torsion}  from  Proposition \ref{prop:ker_Gamma_torsion_variant} contains the following information:

\begin{cor}
\label{cor:imderlie_imderliespec_torsion}
The functor $\imderliespec (-)/ \imderlie(-) $ on $\spmon$ is torsion. Hence,  if the functor $ |\derpt^+ (\opd (R \oplus - ))|/\imderlie (-)$ is torsion-free, then $\imderlie = \imderliespec$.
\end{cor}

\begin{rem}
The Corollary is principally  of theoretical interest, since it is not expected that
 this will allow the calculation of $\imderlie (-)$, except in cases where the functor is already understood.
\end{rem}

\subsection{The cases $\lieopd$ and $\assopd$}
\label{subsect:main_lie_ass}

In the case $\opd = \lieopd$, the conclusion of Proposition \ref{prop:ker_Gamma_torsion_variant} can be refined further since it is possible to neglect $\imderlie[\lieopd]$. 

\begin{prop}
\label{prop:ker_Gamma_lieopd}
The kernel of the natural (with respect to $V \in \ob \spmon$) transformation 
\[
\trqlie[V] : 
  \der^+ (\lieopd (V))/ \derlie (\lieopd (V) )
  \rightarrow 
  |\derpt^+ (\lieopd (R \oplus V))|/\imderlie[\lieopd] (V)
\cong 
|\overline{T}(V)|/ V   
\]
is $3$-torsion.
\end{prop}

\begin{proof}
The proof is a refinement of that of Proposition \ref{prop:ker_Gamma_torsion_variant}, using the fact that $\imderlie [\lieopd] (V) = \imderliespec [\lieopd] (V)= V \subset |T(V)|$ (see Example \ref{exam:lieopd_imderlie}) and the fact that the functor $V \mapsto |T(V)|/ V  $ is torsion-free, since it arises from a functor defined on $\modr$.

Corollary \ref{cor:naturality_addelta_opd} shows that the morphism $\trqlie[V]$ respects the natural gradings, hence it suffices to work with homogeneous derivations. The case of degree one is clear, since $\derlie (\lieopd (V) )$ coincides with $ \der^+ (\lieopd (V))$ in degree one and $|T(V)|/ V   =0$ in degree one. Hence we may assume that the derivations have degree greater than one and consider the kernel of 
\[
\gendiv^\lieopd_V  : 
  \der^+ (\lieopd (V))/ \derlie (\lieopd (V) )
  \rightarrow 
|T(V)|, 
\]
i.e., we may neglect $\imderlie [\lieopd]$.

Consider an element of the kernel represented by an element  $x \in \der^+(\lieopd (V))$.
As in the proof of Proposition \ref{prop:ker_Gamma_torsion_variant}, after passage to $V_1:= R\oplus V$, we may assume that this element arises from a pointed derivation 
 $x_1\in \derpt^+(\lieopd (R \oplus V))\subset \der^+ (\lieopd (V_1))$ modulo $\derlie (\lieopd (V_1))$ up to $2$-torsion.

Since the codomain of $\gendiv^\lieopd_{-} $ is torsion-free, the hypothesis that $x$ lies in the kernel of $\gendiv^\lieopd_V$ implies that $\gendiv^\lieopd_{V_1}(x_1)=0$. This is checked by a diagram chase in the appropriate modification of the diagram appearing in the proof of Proposition \ref{prop:ker_Gamma_torsion_variant}.

The second statement of Corollary \ref{cor:kernel_addelta_ptd} then implies that $x_1$ lies in  
$$[\derpt^+(\lieopd (R \oplus V)), \derpt^+(\lieopd (R \oplus V))].$$ 

Proposition  \ref{prop:mod_commutators_precise} then gives that, up to $2$-torsion, $x_1$ lies in $\derlie (\lieopd (R \oplus V))$. This implies that the image of $x \in \der^+(\lieopd(V))$ under the morphism $V \rightarrow  R^{\oplus 3} \oplus V$ of $\spmon$ lies in $\derlie(\lieopd ( R^{\oplus 3} \oplus V))$, which gives the result. 
\end{proof}

In the case $\opd = \assopd$, one has the slightly weaker conclusion:

\begin{prop}
\label{prop:ker_Gamma_assopd}
The kernel of the natural transformation induced by $\trqass[V]$
\[
  \der^+ (\assopd (V))/ \derlie (\assopd (V) )
  \rightarrow 
  |\derpt^+ (\assopd (R \oplus V))|/\imderliespec[\assopd] (V)
\cong 
|\overline{T}(V)| \otimes |\overline{T} (V) \op| 
\]
is $4$-torsion.
\end{prop}

\begin{proof}
This follows from Proposition \ref{prop:ker_Gamma_torsion_variant} using the fact that the functor $V \mapsto |\overline{T}(V)| \otimes |\overline{T} (V) \op| $ is torsion-free. 
\end{proof}

\appendix
\part{Appendices}
\label{part:appendices}

 \section{Free reduced operads, planar trees and labellings}
\label{sect:free} 

This appendix provides background on planar trees and the construction of free operads. Section \ref{subsect:fo} outlines the construction of free operads whereas Section \ref{subsect:S-label} treats the labelling of leaves and roots.

 \subsection{Free operads}
\label{subsect:fo}
 
The construction of the free operad on an $\fb\op$-module is given in \cite[Section 5.5]{LV} and \cite[Section II.1.9]{MSS}. Here the focus is  on the fully free case, i.e.,  where the $\fb\op$-module is free on a set of generators. Only reduced operads are considered, so as to simplify the exposition.

\begin{hyp}
\label{hyp:gen}
Let $\gen$ be the graded set of generators 
$
\gen = \amalg_{n\geq 2} \gen (n),
$ 
where $\gen (n)$ corresponds to operadic generators of arity $(n-1)$. 
\end{hyp}
 
\begin{rem}
The free operad (in $R$-modules) on $\gen$, denoted $\fo$, satisfies the following universal property. For an operad $\opd$, there is a natural isomorphism:
\[
\hom (\fo, \opd) 
\cong 
\prod_{n\geq 2} 
\hom_{\sets} (\gen (n), \opd (n-1)).
\]

Hence, if $\opd$ is a reduced operad, there exists a generating set $\gen$ and a surjection of operads 
\[
\fo
\twoheadrightarrow
\opd;
\]
for instance,  take $\gen (n) := \opd (n-1)$ for $n \geq 2$ and the morphism induced by the identity. 
\end{rem}

Such free operads arise from non-symmetric operads in sets and are closely related to the tree operads of \cite[Section I.1.5]{MSS}. (See \cite[Section 5.9]{LV} for non-symmetric operads and their relation with symmetric operads.) In particular, the following constructions are based on planar, rooted trees.

\begin{nota}
All trees considered here are planar and rooted and have a finite number of vertices. 
The degree of a vertex $v$ is the number of half edges attached, written $\deg (v)$; a vertex is internal if 
$\deg(v) \geq 2$, otherwise it is a leaf or the root.   The set of internal vertices of a tree $\tr$ is written $v (\tr)$ and the set of leaves $l (\tr)$.  Thus $v(\tr) = \emptyset$ if and only if $\tr$ has a single leaf and no internal vertex. 
\end{nota}

\begin{defn}
A rooted planar tree $\tr$ is binary if all internal vertices have degree $3$.
\end{defn}

A rooted planar tree $\tr$ has an embedding in the plane, hence the leaves inherit a natural numbering, for example consider the (non-binary) tree:

 \qquad
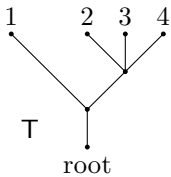
\begin{tikzpicture}[scale=.5]
\node [above right] at (-2, 0) {$\tr$};
\draw [fill] (0,0)  circle [radius=.05];
\draw (0,0) -- (0,1) -- (-2, 3);
\draw (0,1) -- (1,2) -- (0,3); 
\draw (1,2) -- (2,3);
\draw (1,2) -- (1,3);   
\node [below] at (0,0) {root}; 
\draw [fill] (0,1)  circle [radius=.05];
\draw [fill] (1,2)  circle [radius=.05];
\draw [fill] (-2,3)  circle [radius=.05];
\node [above] at (-2,3) {$1$};
\draw [fill] (0,3)  circle [radius=.05];
\node [above] at (0,3) {$2$};
\draw [fill] (2,3)  circle [radius=.05];
\node [above] at (2,3) {$4$};
\draw [fill] (1,3)  circle [radius=.05]; 
 \node [above] at (1,3) {$3$};
\end{tikzpicture}

Grafting of rooted trees is a fundamental operation, defined as follows:

\begin{defn}
\label{defn:grafting_rooted_planar}
Given rooted planar  trees $\tr'$ and $\tr''$, for  the $\ell$th leaf of $\tr'$, the tree 
$\tr' \circ_\ell \tr''$ is the rooted planar tree obtained by grafting the root of $\tr''$ to the  $\ell$th leaf (forgetting the resulting degree 
$2$ vertex). 
\end{defn}

The $\ell$th grafting operation can be represented schematically by:

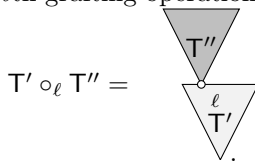
\qquad
\begin{tikzpicture}[scale=1]
\draw [fill=gray!10] (0,0) -- (.5,1) -- (-.5,1) -- (0,0);
\node [right] at (0,0) {.};
\draw [fill=lightgray] (-0.25,1) -- (.25,2) -- (-.75,2) -- (-0.25,1);
\draw [fill=white] (-0.25,1) circle [radius = 0.05];
\node [below right] at (-0.25,1) {$\scriptstyle{\ell}$};
\node at (0, .5) {$\tr'$};
\node at (-.25,1.5) {$\tr''$};
\node at (-2,1) {$\tr'\circ_{\ell}\tr''=$};
\end{tikzpicture}

\begin{defn}
A $\gen$-tree is a rooted planar tree $\tr$ equipped with a graded labelling of the internal vertices $v(\tr) \rightarrow \gen$ (i.e., such that $v \mapsto \gen (\deg(v))$). 
\end{defn}

\begin{prop}
\label{prop:op_gen}
The free non-symmetric set operad on $\gen$ has $n$-operations the set of $\gen$-trees with $n$ leaves and composition given by grafting; the identity is given by the rooted planar tree with no internal vertex. 

The free set operad on $\gen$ is the associated symmetric operad; in particular, an $n$-operation is given by a  $\gen$-tree $\tr$ with $n$ leaves, equipped with a bijection $l(\tr) \stackrel{\cong}{\rightarrow} \mathbf{n}$. 
\end{prop}

One passes from set operads to operads in $\Rmod$ using the $R$-linearization functor $R[-]$. 

\begin{cor}
\label{cor:fo}
The free operad $\fo$ on the set $\gen$ has, for $n \in \nat$,  
$
\fo(\mathbf{n}) 
$ 
the free $R$-module with basis given by $\gen$-trees $\tr$ with $n$ leaves, equipped with a bijection $l(\tr) \stackrel{\cong}{\rightarrow} \mathbf{n}$.

The operadic composition is induced by grafting of trees. 
\end{cor}

Since the binary case is of significant interest here, the following notation is introduced:

\begin{nota}
\label{nota:free_bin}
For $\bin$ a set, let $\fob$ denote the free operad on $\bin_3$ (i.e., considering $\bin$  as a graded set concentrated in degree $3$). Thus $\fob$ is the free binary operad on the set of generators $\bin$.
\end{nota}

\begin{exam}
Let $\bin = \{*\}$; then $\fob$ is the free {\em binary} operad on a single generator. This is the magmatic operad  that encodes (non-unital) free, binary (non-associative) algebras (see \cite[Sections 13.8 and C.1]{LV}).
\end{exam}

\subsection{$S$-labelled trees}
\label{subsect:S-label}

\begin{defn}
\label{defn:rpt}
For $S$ a finite set, an $S$-labelled rooted planar $\gen$-tree is a rooted planar $\gen$-tree $\tr$ equipped with a map from the set of degree one vertices of $\tr$ to $S$; the root label of $\tr$ is written $\rt (\tr) \in S$. 

Denote by 
\begin{enumerate}
\item 
$\rpt (S)$ the set of $S$-labelled rooted  planar $\gen$-trees;
\item 
$\brpt (S)$ the set of $S$-labelled rooted binary planar $\bin_3$-trees (i.e., with internal vertices labelled by $\bin$), so that $\brpt (S)= \rpt(S)$ for $\gen = \bin_3$. 
\end{enumerate}
\end{defn}

\begin{lem}
\label{lem:rpt_functor_fin}
The association $S \mapsto \rpt (S)$ defines a functor from the category of finite sets to sets.
\end{lem}

\begin{proof}
For $f: S \rightarrow S'$ a map of finite sets,  $\rpt (S) \rightarrow \rpt (S')$ is given by postcomposing the labelling of the leaves by $f$.
\end{proof}

The following distinguished classes of rooted  planar $\gen$-trees are important:

\begin{defn}
\label{defn:arboriculture}
For a finite set $S$, a $\gen$-tree $\tr \in \rpt (S)$ is said to be: 
\begin{enumerate}
\item 
{\bf disjoint} if the root label does not also occur as a leaf label; 
\item 
{\bf pointed} if the root label occurs precisely once as a leaf label (often the set $S$ has a specified basepoint and the root is labelled by this); 
\item 
{\bf special pointed} if it is pointed and the path from the $\rt(\tr)$ to the leaf labelled by $\rt(\tr)$ contains at most one internal vertex. 
 \end{enumerate}
\end{defn}

Grafting of $S$-labelled trees induces an operation $\lhd$ on the $R$-linearization $ R [\rpt (S)]$ :

\begin{defn}
\label{defn:graft_rpt}
For $\tr_1, \tr_2 \in \rpt (S)$, let $\tr_1 \lhd \tr_2 \in R [\rpt (S)]$ denote the sum of the elements of $\rpt (S)$ 
that are obtained by grafting the root of $\tr_2$ to a leaf of $\tr_1$  with the same label and  forgetting this vertex.

Extend this by $R$-linearity to $\lhd : R [\rpt (S)] \otimes R [\rpt (S)]\rightarrow R [\rpt (S)]$.
\end{defn}

\begin{rem}
The diagrammatic representation for grafting of rooted trees adapts to the operation $\lhd$ as follows: the   numbering of the leaf $\ell$ is replaced by the label $\rt(\tr_2)$ and the sum over all possible such graftings is usually left implicit. 
\end{rem}

\begin{prop}
\label{prop:rpt_prelie}
For $S$ a finite set, $(R[\rpt (S)], \lhd)$ is a preLie algebra; this defines a functor
$
R[\rpt (-)] : 
\finj  \rightarrow  \prelieopd \dash \alg$.
\end{prop}

\begin{proof}
That $\lhd$ defines a preLie structure on $R[\rpt (S)]$ is a standard argument that goes back to Gerstenhaber \cite{MR161898}. This can be seen explicitly as follows: consider  $\gen$-trees $\tr_i \in \rpt(S)$, $i \in \{1,2, 3\}$,  with roots $x, y, z \in S$ respectively. 
One checks that the  associator 
$
(\tr_1 \lhd \tr_2) \lhd \tr_3 
- \tr_1 \lhd (\tr_2 \lhd \tr_3 )
$ 
 is the sum of all possible `double' graftings onto $\tr_1$; i.e., corresponding to $S$-labelled trees of one of the following forms:

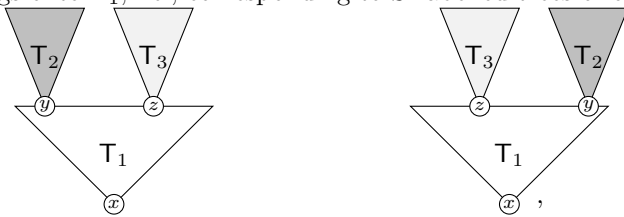
\qquad
\begin{tikzpicture}[scale = 1.3]
\draw (0,0) -- (1,1) -- (-1,1) -- (0,0); 
\node at (0,.5) {$\tr_1$}; 
\draw [fill=white] (0,0) circle [radius= 0.1];
\node at (0,0) {$\scriptstyle{x}$} ;
\draw [fill = lightgray](-.7,1) -- (-.3,2) -- (-1.1,2) -- (-.7,1); 
\node at (-.7, 1.5) {$\tr_2$}; 
\draw [fill=white] (-.7,1) circle [radius= 0.1];
\node at (-.7,1) {$\scriptstyle{y}$} ;
\draw [fill = gray!10](.4,1) -- (.8,2) -- (0,2) -- (.4,1); 
\node at (.4, 1.5) {$\tr_3$};
 \draw [fill=white] (.4,1) circle [radius= 0.1];
\node at (.4,1) {$\scriptstyle{z}$} ;
\draw (4,0) -- (5,1) -- (3,1) -- (4,0); 
\draw [fill=white] (4,0) circle [radius= 0.1];
\node at (4,0) {$\scriptstyle{x}$} ;
\node at (4,.5) {$\tr_1$}; 
\draw [fill = gray!10](3.7,1) -- (4.1,2) -- (3.3,2) -- (3.7,1); 
\node at (3.7, 1.5) {$\tr_3$}; 
\draw [fill=white] (3.7,1) circle [radius= 0.1];
\node at (3.7,1) {$\scriptstyle{z}$} ;
\draw [fill = lightgray](4.8,1) -- (5.2,2) -- (4.4,2) -- (4.8,1); 
\node at (4.8, 1.5) {$\tr_2$};
 \draw [fill=white] (4.8,1) circle [radius= 0.1];
\node at (4.8,1) {$\scriptstyle{y}$} ;
\node [right] at (4,0) {\ \ ,};
\end{tikzpicture}

\noindent
noting that the labels $y, z \in S$ can appear more than once amongst the leaves of $\tr_1$ and in either order. The above trees retain no information on the order in which the double grafting was carried out; this gives the preLie property.

Naturality of $\rpt (S)$ as a functor from finite sets to $\Rmod$ is given by Lemma \ref{lem:rpt_functor_fin}. Upon restriction to $\finj$, this is compatible with the preLie structure $\lhd$.
\end{proof}

\subsection{Pruning}
\label{subsect:pruning}

The pruning operation  considered below is a basic technique that can be viewed as the inverse operation to grafting. 

Given a rooted (unlabelled)  planar tree $\tr$ with $|v(\tr)|>1$ and a choice of internal edge (i.e., an edge between two internal vertices), form two rooted planar trees $\tr'$ and $\tr''$ by cutting the internal edge, thus  creating a new leaf and a new root. By construction $|v(\tr')|, |v(\tr'')| \geq 1$ and 
\[
|v(\tr')| +  |v(\tr'')| = |v(\tr)|. 
\]
Moreover, if $\tr'$ contains the root of $\tr$ and the new leaf is numbered $\ell$, then 
$
\tr = \tr' \circ_\ell \tr''
$, using the grafting operation.

This process can be carried out for $S$-labelled trees; the only subtlety is that one has to choose a label for the new leaf of $\tr'$ and the root of $\tr''$.  To avoid an arbitrary choice, this is usually carried out by passing to the enlarged set $S_+ : = S\amalg \{ +\}$ and using $+$ as this label.

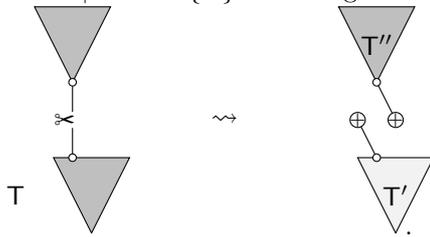
\qquad
\begin{tikzpicture}[scale=1]
\draw [fill=lightgray] (0,0) -- (.5,1) -- (-.5,1) -- (0,0);
\draw [fill=lightgray] (-0.25,2) -- (.25,3) -- (-.75,3) -- (-0.25,2);
\draw (-0.25,1)-- (-0.25,2);
\draw [fill=white] (-0.25,1) circle [radius = 0.05];
\draw [fill=white] (-0.25,2) circle [radius = 0.05];
\draw [fill=white, draw=white] (-0.25,1.5) circle [radius = 0.1];
\node at (-0.35,1.5) {\Rightscissors};
\node at (-1, .5) {$\tr$};
\node at (1.75, 1.5) {$\leadsto$}; 
\draw [fill=gray!10] (4,0) -- (4.5,1) -- (3.5,1) -- (4,0);
\node [right] at (4,0) {.};
\draw [fill=lightgray] (3.75,2) -- (4.25,3) -- (3.25,3) -- (3.75,2);
\draw (3.75,1)-- (3.5,1.5);
\draw (3.75,2)-- (4,1.5);
\draw [fill=white] (3.75,1) circle [radius = 0.05];
\draw [fill=white] (3.75,2) circle [radius = 0.05];
\draw [fill=white] (3.5,1.5) circle [radius = 0.1];
\node at (3.5,1.5) {$\scriptstyle{+}$};
\draw [fill=white] (4,1.5) circle [radius = 0.1];
\node at (4,1.5) {$\scriptstyle{+}$};
\node at (4, .5) {$\tr'$};
\node at (3.75,2.5) {$\tr''$};
\end{tikzpicture}

By construction, $\tr = \tr' \lhd \tr''$, since $+$ labels the root of $\tr''$ and a unique leaf of $\tr'$, namely the new leaf.

 \section{Relating to the enveloping algebra}
 \label{sect:env_alg}
 
This Section serves to outline alternative approaches to the natural associative algebra structure on $\derpt (\opd (R\oplus V))$ that was introduced in Section \ref{sect:pointed}.

\subsection{Kähler differentials and enveloping algebras}
\label{subsect:Kahler}

This  is an addendum to Section \ref{sect:opd}, giving the relationship with other standard constructions in operad theory.

\begin{nota}
For $A$ an $\opd$-algebra, let $\modopd_A$ denote the category of $A$-modules.
\end{nota}

The forgetful functor $\modopd _A \rightarrow \Rmod$ admits a left adjoint 
\[
A \otimes^\opd - \ : \  \Rmod \rightarrow \modopd_A
\]
(see \cite[Theorem 12.3.4]{LV}). (For $M \in \ob \Rmod$, $A \otimes ^\opd M$ is constructed as a quotient of $\opd (A; M)$ via an explicit coequalizer diagram.) 

\begin{defn}
\label{defn:env_alg}
(Cf. \cite[Section 12.3.4]{LV}.)
For $A$ an $\opd$-algebra, let 
$
U_\opd A := A \otimes ^\opd R
$ 
be the enveloping algebra of $A$, equipped with its canonical unital, associative algebra structure.
\end{defn}

\begin{rem}
\label{rem:left_modules_env.alg.}
One significance of the enveloping algebra  is that $\modopd_A$ is equivalent to the category of left $U_\opd A$-modules (see \cite[Proposition 12.3.8]{LV}). For instance, when $\opd = \lieopd$ and $\mathfrak{g}$ is a Lie algebra, $U_\lieopd \mathfrak{g}$ is the usual universal enveloping algebra $U \mathfrak{g}$; for a non-unital associative algebra $\mathcal{A}$, $U_\assopd \mathcal{A}$ is the enveloping algebra $(\mathcal{A} \otimes \mathcal{A}\op) \oplus R$, where the $\oplus R$ serves to make the algebra unital.
\end{rem} 

The above can be considered for $A= \opd (V)$,  the free $\opd$-algebra on $V\in \ob \modr$.

\begin{prop}
\label{prop:free_mod_over_free_alg}
For $V, M \in \ob \modr$, there is a natural isomorphism 
$
\opd (V) \otimes ^\opd M \cong \opd (V; M).
$
\end{prop}

\begin{proof}
This can be deduced from \cite[Proposition 12.3.5]{LV} and can also be proved directly as indicated below.

The unit of the operad induces a morphism of $R$-modules $\opd (V; M) \rightarrow \opd (\opd (V);M)$ and hence $\opd (V; M) \rightarrow \opd (V) \otimes ^\opd M$. To construct the inverse, one  shows that $\opd (V;M)$ has a natural $\opd (V)$-module structure; this is induced by the operad structure of $\opd$. 
\end{proof}

Proposition \ref{prop:free_mod_over_free_alg} has the immediate corollary:

\begin{cor}
\label{cor:env_alg_identify}
For $V \in \ob \modr$, the underlying $\opd (V)$-module of $U_\opd \opd (V)$ is isomorphic to $\opd (V; R)$.
\end{cor}

For an $\opd$-algebra $A$, there is an operadic version of the module of Kähler differentials (see \cite[Section 12.3.8]{LV}). This is the  $A$-module $\Omega_\opd A$ that is defined by the  coequalizer in $A$-modules of 
\[
A \otimes ^\opd \opd (A) 
\rightrightarrows 
A \otimes ^\opd A
\]
for the $A$-module morphisms induced by  $\mu_A : \opd (A) \rightarrow A$ and by the composite
\[
\opd (A) 
\stackrel{\delta^\opd_A}{\rightarrow }
\opd (A; A) 
\twoheadrightarrow 
A\otimes ^\opd A
\]
where the second map is given by the construction of $A \otimes ^\opd A$. In particular, these give  the universal derivation 
$ 
A \rightarrow \Omega_\opd A. 
$ 
This induces the natural isomorphism 
\[
\hom_{\modopd_A} (\Omega_\opd A , M) 
\cong 
\der _A (A, M)
\]
for $M \in \ob \modopd_A$ (see \cite[Proposition 12.3.13]{LV}). 

In the case of a free $\opd$-algebra, one has the following identification:

\begin{prop}
\label{prop:identify_Kahler_free}
(Cf. \cite[Section 12.3.8]{LV}.) 
For $V \in \ob \modr$, there are canonical isomorphisms of $\opd (V)$-modules 
$\Omega_\opd \opd (V) 
\cong 
\opd (V; V) 
\cong 
\opd (V) \otimes ^\opd V.
$ 

The universal derivation $\opd (V) \rightarrow \Omega_\opd \opd (V) $ identifies with 
$
\delta^\opd_V : \opd (V) \rightarrow \opd (V;V ).
$
\end{prop}

\begin{rem}
The isomorphism $\Omega_\opd \opd (V) 
\cong 
\opd (V; V) 
\cong 
\opd (V) \otimes ^\opd V$ together with the fact that $\Omega_\opd A$ corepresents $\der_A (A, -)$ gives another interpretation of Proposition \ref{prop:deriv_free} in the case $A=\opd (V)$. The universal derivation explains the construction outlined in the proof of that result. In particular, this explains the significance of the morphism $\delta^\opd _V$. 
\end{rem}

\subsection{An alternative approach to pointed derivations}
The purpose of this Section is to give an alternative description of the associative algebra structure on $\derpt (\opd (R\oplus V))$ given by  Theorem \ref{thm:assoc_alg}.

Recall from equation (\ref{eqn:derpt}) that there is an identification $
\derpt (\opd (R\oplus V)) \cong \tau \opd (V)$. Hence,  the associative algebra on $\derpt (\opd (R\oplus V)) $ given by Theorem \ref{thm:assoc_alg} induces an associative product 
$
\tau \opd (V) \otimes \tau \opd (V) 
\rightarrow 
\tau \opd (V)$
 that is natural with respect to $V$. This arises from the corresponding structure on the $\fb\op$-module $\tau \opd$ via the Schur functor construction. 

\begin{lem}
\label{lem:mutau}
For $m,n \in \nat$, the operad structure of $\opd$ restricts to a morphism of $R[\sym_m \times \sym_{n+1}]$-modules
\begin{eqnarray}
\label{eqn:partial}
\opd (\mathbf{m+1})\downarrow_{\sym_m}  \otimes \opd (\mathbf{n+1}) \rightarrow \opd (\mathbf{m+n+1}), 
\end{eqnarray}
where the codomain is given the restricted structure via 
$\sym_m \times \sym_{n+1} \subset \sym_{m+n+1}$ induced by the identification $\mathbf{m} \amalg (\mathbf{n+1}) \cong \mathbf{m+n+1}$.

In particular, restricting to $\sym_m \times \sym_n \subset \sym_{m+n} \subset \sym_{m+n+1}$, this gives a morphism of $R[\sym_m \times \sym_{n}]$-modules
\[ 
\nu^\opd_{\mathbf{m},\mathbf{n}}  : \tau \opd (\mathbf{m}) \otimes \tau \opd (\mathbf{n}) \rightarrow \tau \opd (\mathbf{m+n}).
\]
\end{lem}

\begin{proof}
(Sketch.)
The morphisms of equation (\ref{eqn:partial}) encode the partial compositions of the operad $\opd$ (see \cite[Section 5.3.4]{LV} and Remark \ref{rem:partial_comp}). Upon restriction, one obtains the morphisms $\nu^\opd_{\mathbf{m},\mathbf{n}} $, as stated.
\end{proof}

\begin{rem}
\ 
\begin{enumerate}
\item 
The operations $\nu^\opd_{*,*}$ respects the grading by arity.
\item 
By definition, $\tau \opd (\mathbf{0}) = \opd (\mathbf{1})$ and this contains the unit. Then
$
\nu^\opd_{\mathbf{0}, \mathbf{0}} : \tau \opd (\mathbf{0}) \otimes \tau \opd (\mathbf{0})
\rightarrow 
\tau \opd (\mathbf{0}) 
$ 
corresponds  to the usual unital associative algebra structure on $\opd (\mathbf{1})$. 
\end{enumerate}
\end{rem}

The following uses the tensor product of $\fb\op$-modules (see Definition \ref{defn:fbop-modules_tensor}): 

\begin{defn}
\label{defn:mutau_fbop}
Let $ 
\tilde{\nu}^\opd : \tau \opd \otimes \tau \opd \rightarrow \tau \opd
$ 
be the morphism of $\fb\op$-modules encoding the morphisms $\nu^\opd_{\mathbf{m},\mathbf{n}}$ for $m,n \in \nat$.
\end{defn}

\begin{prop}
\label{prop:mutau_associative}
\ 
\begin{enumerate}
\item 
The morphism  $\tilde{\nu}^\opd : \tau \opd \otimes \tau \opd \rightarrow \tau \opd$ defines a unital, associative algebra structure on $\tau\opd$ in the category of $\fb\op$-modules.
\item 
For $V \in \ob\modr$, $\tau \opd (V)$ has a natural, unital associative algebra structure and this is natural with respect to the operad $\opd$.
\item 
The morphism $\tilde{\nu}^\opd$ is natural with respect to the operad $\opd$. Thus $\tau$ defines a functor from the category of operads to the category of unital associative algebras in the  category of $\fb\op$-modules. 
\end{enumerate}
\end{prop}

\begin{proof}
Associativity for $\tilde{\nu}^\opd$ follows from the associativity property of partial compositions and the unital property follows from the interpretation of the operadic unit in terms of partial compositions.  The second  statement follows by passage to the associated Schur functors, by Proposition \ref{prop:tensor_fbop_schur}. 
Naturality with respect to the operad $\opd$ is clear.
\end{proof}

The following Theorem shows that the associative algebra structure on $\derpt (\opd (R \oplus V))$ given by Theorem \ref{thm:assoc_alg} is induced by $(\tau \opd, \tilde{\nu}^\opd)$.

\begin{thm}
\label{thm:derpt_via_operad}
For $ V \in \ob \modr$, the natural unital associative algebra structure on $\derpt (\opd (R \oplus V))$ is naturally isomorphic to the unital associative algebra structure on $\tau \opd (V)$ that is induced by $\tilde{\nu}^\opd$.  This isomorphism is natural with respect to the operad $\opd$.
\end{thm}

\begin{proof} (Sketch.)
By (\ref{eqn:derpt}), there is a natural isomorphism of $R$-modules $
\derpt (\opd (R\oplus V)) \cong \tau \opd (V)$. It remains to show that this induces an isomorphism of the respective natural unital, associative algebra structures. 
This follows by  analysing the construction of the preLie structure on $\der (\opd (R \oplus V))$ and its restriction to an associative structure on $\derpt (\opd (R \oplus V))$ given in Theorem \ref{thm:assoc_alg}. 
\end{proof}

There is an alternative to Theorem \ref{thm:derpt_via_operad}, using the enveloping algebra of the free $\opd$-algebra on $V$: 

\begin{thm}
\label{thm:derpt_via_operad_env_alg}
For $V \in \ob \modr$, the natural unital associative algebra structure on $\derpt (\opd (R \oplus V))$ is naturally isomorphic to that on $U_\opd \opd (V)$.  This isomorphism is natural with respect to the operad $\opd$.
\end{thm}

\begin{proof}(Sketch.)
By construction, $V \mapsto U_\opd \opd (V)$ is a functor from $\modr$ to unital associative algebras and this is natural with respect to $\opd$. 

Corollary \ref{cor:env_alg_identify} identifies the underlying $R$-module of $ U_\opd \opd (V)$ with  $\opd (V; R)$. That the algebra structures are equivalent follows from  \cite[Section 12.3.4]{LV} (see the paragraph following \cite[Lemma 12.3.7]{LV}).
\end{proof}

\begin{rem}
In operad theory, the enveloping algebra construction is usually viewed as giving a functor from $\opd \dash\alg$ to unital associative algebras (see \cite[Proposition 12.3.9]{LV}). Proposition \ref{prop:mutau_associative} gives the appropriate {\em universal} construction as a functor from operads to unital associative algebras in $\fb\op$-modules. 
\end{rem}


\begin{thebibliography}{AKKN18b}

\bibitem[AKKN18a]{MR3758425}
Anton Alekseev, Nariya Kawazumi, Yusuke Kuno, and Florian Naef, \emph{The
  {G}oldman-{T}uraev {L}ie bialgebra in genus zero and the {K}ashiwara-{V}ergne
  problem}, Adv. Math. \textbf{326} (2018), 1--53. \MR{3758425}

\bibitem[AKKN18b]{2018arXiv180409566A}
Anton {Alekseev}, Nariya {Kawazumi}, Yusuke {Kuno}, and Florian {Naef},
  \emph{{The Goldman-Turaev Lie bialgebra and the Kashiwara-Vergne problem in
  higher genera}}, arXiv e-prints (2018), arXiv:1804.09566.

\bibitem[Dar19]{MR3975077}
Jacques Darn\'{e}, \emph{On the stable {A}ndreadakis problem}, J. Pure Appl.
  Algebra \textbf{223} (2019), no.~12, 5484--5525. \MR{3975077}

\bibitem[ES11]{MR2846914}
Naoya Enomoto and Takao Satoh, \emph{On the derivation algebra of the free
  {L}ie algebra and trace maps}, Algebr. Geom. Topol. \textbf{11} (2011),
  no.~5, 2861--2901. \MR{2846914}

\bibitem[Ger63]{MR161898}
Murray Gerstenhaber, \emph{The cohomology structure of an associative ring},
  Ann. of Math. (2) \textbf{78} (1963), 267--288. \MR{161898}

\bibitem[KM01]{Kap_Manin}
M.~Kapranov and Yu. Manin, \emph{Modules and {M}orita theorem for operads},
  Amer. J. Math. \textbf{123} (2001), no.~5, 811--838. \MR{1854112}

\bibitem[LV12]{LV}
Jean-Louis Loday and Bruno Vallette, \emph{Algebraic operads}, Grundlehren der
  Mathematischen Wissenschaften [Fundamental Principles of Mathematical
  Sciences], vol. 346, Springer, Heidelberg, 2012. \MR{2954392}

\bibitem[MSS02]{MSS}
Martin Markl, Steve Shnider, and Jim Stasheff, \emph{Operads in algebra,
  topology and physics}, Mathematical Surveys and Monographs, vol.~96, American
  Mathematical Society, Providence, RI, 2002. \MR{1898414}

\bibitem[Reu93]{MR1231799}
Christophe Reutenauer, \emph{Free {L}ie algebras}, London Mathematical Society
  Monographs. New Series, vol.~7, The Clarendon Press, Oxford University Press,
  New York, 1993, Oxford Science Publications. \MR{1231799}

\bibitem[Sat06]{MR2269583}
Takao Satoh, \emph{New obstructions for the surjectivity of the {J}ohnson
  homomorphism of the automorphism group of a free group}, J. London Math. Soc.
  (2) \textbf{74} (2006), no.~2, 341--360. \MR{2269583}

\bibitem[Sat12]{MR2864772}
\bysame, \emph{On the lower central series of the {IA}-automorphism group of a
  free group}, J. Pure Appl. Algebra \textbf{216} (2012), no.~3, 709--717.
  \MR{2864772}

\bibitem[Wei13]{Weibel_K}
Charles~A. Weibel, \emph{The {$K$}-book}, Graduate Studies in Mathematics, vol.
  145, American Mathematical Society, Providence, RI, 2013, An introduction to
  algebraic $K$-theory. \MR{3076731}

\end{thebibliography}
\providecommand{\bysame}{\leavevmode\hbox to3em{\hrulefill}\thinspace}
\providecommand{\MR}{\relax\ifhmode\unskip\space\fi MR }
\providecommand{\MRhref}[2]{%
  \href{http://www.ams.org/mathscinet-getitem?mr=#1}{#2}
}
\providecommand{\href}[2]{#2}

\end{document}